\documentclass[reqno,11pt]{amsart}

\usepackage{amsmath,amssymb,amsthm,graphicx,mathrsfs,url, enumerate,transparent, dsfont}

\usepackage[normalem]{ulem}
\usepackage[usenames,dvipsnames]{color}
\usepackage[colorlinks=true,linkcolor=Red,citecolor=Green]{hyperref}
\usepackage[super]{nth}
\usepackage[open, openlevel=2, depth=3, atend]{bookmark}
\hypersetup{pdfstartview=FitH,bookmarksopen=true,pdfpagemode=UseOutlines}
\hypersetup{pdftitle=SRB Measures for Anosov actions,pdfdisplaydoctitle=true}

\usepackage[font=footnotesize]{caption}
\usepackage{a4wide}

\newtheorem{theorem}{Theorem}
\newtheorem{prop}{Proposition}[section]
\newtheorem{lemma}[prop]{Lemma}
\newtheorem{cor}[prop]{Corollary}
\newtheorem*{lemma*}{Lemma}
\theoremstyle{definition}
\newtheorem{Def}[prop]{Definition}

{\bf}{\it}
\theoremstyle{remark}
\newtheorem{rem}[prop]{Remark}
\numberwithin{equation}{section}

\newcommand{\rr}{\mathbb{R}}
\newcommand{\N}{\mathbb{N}}

\newcommand{\R}{\mathbb{R}}
\newcommand{\Z}{\mathbb{Z}}
\newcommand{\C}{\mathbb{C}}

\newcommand{\T}{\mathcal{T}}

\newcommand{\M}{\mathcal{M}}
\newcommand{\W}{\mathcal{W}}

\renewcommand{\a}{\mathfrak a}

\renewcommand{\Re}{\operatorname{Re}}
\renewcommand{\Im}{\operatorname{Im}}

\DeclareMathOperator{\Res}{Res}

\DeclareMathOperator{\comp}{comp}

\DeclareMathOperator{\Op}{Op}

\DeclareMathOperator{\rank}{rank}

\DeclareMathOperator{\supp}{supp}

\DeclareMathOperator{\WF}{WF}\def\WFh{\WF_h}
\DeclareMathOperator{\Tr}{Tr}
\DeclareMathOperator{\ran}{ran}

\DeclareMathOperator{\Diffeo}{Diffeo}

\newcommand{\mc}{\mathcal}
\newcommand{\cjg}{\langle}
\newcommand{\cjd}{\rangle}
\newcommand{\la}{\lambda}
\newcommand{\pl}{\partial}
\newcommand{\eps}{\epsilon}

\newcommand{\bbar}{\overline}

\newcommand{\tu}[1]{\textup{#1}}
\newcommand{\Abb}[4]{\begin{cases}\begin{aligned} #1 & \rightarrow  #2 \\ #3 &\mapsto  #4\end{aligned}\end{cases}}

\definecolor{blue-violet}{rgb}{0.54, 0.17, 0.89}
\definecolor{blue-green}{rgb}{0.0, 0.87, 0.87}
\definecolor{auburn}{rgb}{0.43, 0.21, 0.1}
\definecolor{beaver}{rgb}{0.62, 0.51, 0.44}
\definecolor{cadmiumorange}{rgb}{0.93, 0.53, 0.18}
\definecolor{carrotorange}{rgb}{0.93, 0.57, 0.13}

\title[SRB measures for Anosov actions]{SRB measures for Anosov actions}

\author[Y. Guedes Bonthonneau]{Yannick Guedes Bonthonneau}
\email{bonthonneau@math.univ-paris13.fr}
\address{LAGA, Institut Galil\'ee, 99 avenue Jean Baptiste cl\'ement, 93430 Villetaneuse, France.}
\author[C. Guillarmou]{Colin Guillarmou}
\email{colin.guillarmou@universite-paris-saclay.fr}
\address{Universit\'e Paris-Saclay, CNRS,  Laboratoire de math\'ematiques d'Orsay, 91405, Orsay, France.}
\author[T. Weich]{Tobias Weich}
\email{weich@math.upb.de}
\address{Universit\"at Paderborn, Warburgerstr. 100, 33098 Paderborn, Germany}

\date{\today}

\begin{document}

\begin{abstract}
Given a general Anosov $\R^\kappa$ action on a closed manifold, we study properties of certain invariant measures that have recently been introduced in \cite{BGHW20} using the theory of Ruelle-Taylor resonances. We show that these measures share many properties of Sinai-Ruelle-Bowen measures for general Anosov flows such as smooth disintegrations along the unstable foliation, positive Lebesgue measure basins of attraction and a Bowen formula in terms of periodic orbits. Finally we show that if the action in the positive Weyl chamber is transitive, the measure is unique and has full support.
\end{abstract}

\maketitle
\sloppy

\section*{Introduction}

On a closed, smooth Riemannian manifold $(\M, g)$ (normalized with volume $1$) we consider a locally free abelian action $\tau :\R^\kappa\to \Diffeo(\M)$. Assume that $\tau$ is Anosov, and denote by $\mathcal{W}\subset \R^\kappa$ the maximal cone of transversely hyperbolic elements (see Section~\ref{sec:anosov_action} for a precise definition of all these terms). In \cite{BGHW20} it was proved that there exists a Radon probability measure $\mu$, called \emph{the equilibrium measure}, such that for every function $f\in C^0(\M)$, every open proper\footnote{proper meaning that $\overline{ \mc{C}}\setminus \{0\}\subset \mc{W}$.} subcone $\mathcal{C}\subset \W$ and linear form $e$ on $\R^\kappa$, positive on slightly larger cone then $\mc{C}$, the following holds true
\begin{equation}
 \label{eq:Cone-averaging-physical-measure}
\mu(f) = \lim_{T\to +\infty} \frac{1}{|\mathcal{C}_T|} \int_{A\in\mathcal{C}_T} \int_\M f(\tau(-A)(x)) dv_{g}(x) dA.
\end{equation}
Here $\mathcal{C}_T = \{ A \in \mathcal{C}\ |\ e(A) \leq T\}$ and $|\mc C_T|$ denotes the Euclidean volume of 
$\mc{C}_T\subset \R^\kappa$ and $dv_g$ is the Riemannian density on $\mc{M}$ induced by $g$.\footnote{Note that the right hand side of \eqref{eq:Cone-averaging-ergodic-SRB-measure} a priori depends on the choice of subcone $\mc C$, the linear from $e$ and the metric $g$, but it is part of the statement that the expression is in fact independent of $e$ and $\mc C$ and -- provided that the Anosov action is transitive -- also independent of $v_g$.} In this article we will explore the properties of the measure $\mu$, proving in particular:
\begin{theorem}\label{thm:SRB_intro}
Let $\tau$ be a transitive, smooth, locally free, $\R^\kappa$ Anosov action. Let $\mu$ be an invariant Radon probability measure on $\M$, then the following conditions are equivalent:
\begin{enumerate}
  \item \label{it:equilibrium}$\mu$ is the equilibrium measure which then does not depend on the choice of $v_g$.
  \item \label{it:basin} For every continuous $f$, every open proper subcone $\mathcal{C}\subset\W$, and Lebesgue almost every $x\in \M$,
  \[
  \mu(f) = \lim_{T\to +\infty} \frac{1}{|\mathcal{C}_T|} \int_{A\in\mathcal{C}_T} f(\tau(-A)(x)) dA.
  \]
  \item \label{it:abs_cont} The measure $\mu$ has an absolute continuous disintegration\footnote{in the sense of Definition~\ref{def:abscont}} with respect to the local stable foliation, $W^s_{\mathrm{loc}}$ and is thus a Sinai-Ruelle-Bowen (SRB) measure\footnote{in the sense of Definition~\ref{def:SRB}}.
  \item \label{it:WF} The measure $\mu$ has wavefront set\footnote{See \cite[Section 8.1]{Hoe03} for a classical introduction to wavefront sets}  $\WF(\mu)\subset E_s^*$.
\end{enumerate}
Such a measure $\mu$ is always ergodic. If in addition we assume that the action is positively transitive in the sense of Definition~\ref{def:pos_trans}, then $\supp(\mu)=\mc{M}$.
\end{theorem}

In the rank one case, i.e. the case of Anosov flows, the properties (\ref{it:equilibrium}), (\ref{it:basin}) and (\ref{it:abs_cont}) are very similar to the classical properties of the SRB measure for transitive Anosov flows, which was studied extensively by Sinai, Ruelle and Bowen \cite{SinaiSRB,BowenSRB,RuelleSRB,BR75}. The characterisation (\ref{it:WF}) of the equilibrium or SRB measure in terms of wavefront sets is a more modern point of view. For Anosov flow this characterization can be derived by combining the results of Butterley-Liverani \cite{BL07} and Faure-Sjöstrand \cite {FS11} (although the characterization of SRB measures in terms of wavefront sets is not explicitely formulated in these references).
As for the results of this article, we also obtain a more general and more detailed version (Theorem~\ref{thm:physical}) without the transitivity assumption.

For a given smooth Anosov flow, the structural stability implies that any small perturbation of the flow is again an Anosov flow. Furthermore for any fixed Anosov flow, one can associate with each potential $V$ a so-called invariant Gibbs measure that has positive entropy. The world of smooth Anosov flows is thus very rich and, for a fixed flow, there is also a rich ergodic theory due to the different Gibbs measures. In contrast, for higher rank Anosov action the situation is conjectured (and partially known) to be very rigid: in \cite{KaSp94} Katok and Spatzier proved that for a list of algebraic Anosov actions, called \emph{standard Anosov actions}, any small perturbation of the Anosov action is Hölder conjugate to the original action.
Until recently it was conjectured  (see \cite[Conjecture 16.8]{Has07}) that whenever a higher rank Anosov action cannot be factored into a product of an Anosov flow with another action, they are algebraic in the sense that they come from quotients of symmetric spaces or Lie groups. Assuming that this rigidity conjecture holds, the classification of invariant measures reduces to analyzing homogeneous dynamics, i.e $\R^\kappa$ invariant measures on homogeneous spaces. Such measure classifications in homogeneous dynamics  have been intensively studied in the past decades, starting with the works of Katok and Spatzier \cite{Katok-Spatzier96,Katok-Spatzier98} and culminating in more recent works of Einsiedler, Katok and Lindenstrauss \cite{EKL06, EL15}.

Very recently it however turned out that the rigidity conjecture is much more subtle than expected: while for the special case of totally Cartan Anosov actions the rigidity  conjecture has recently be  proven \cite{SV19}, there were also first counter examples in the general setting.
Indeed, in  \cite{Vin22} Vinhage presented an approach to construct non algebraic Anosov actions that have no rank one factor, by applying a non trivial time change to the product of two Anosov flows.
If at least one of the Anosov flows has a non smooth SRB measure, this example leads to Anosov actions without rank one factors that have a non smooth SRB measure. This gives another motivation of considering SRB measures in this setting without assuming that the dynamics is homogeneous.

No matter in which way the rigidity conjecture has to be modified (see \cite{Vin22} for different options), in order to make progress in this direction it is important to understand as many dynamical properties of Anosov actions as possible, without assuming that these actions are homogeneous. In particular it is important to understand and to construct meaningful invariant measures\footnote{As explained to us by Ralf Spatzier, the existence of ergodic measures with full support is an important tool in the direction of proving the rigidity conjecture (see e.g. \cite{Kalinin-Spatzier} where this assumption is crucially used, as well as the discussions in \cite{SV19}).}.

Let us mention two results that are related to Theorem~\ref{thm:SRB_intro} in this direction: in \cite{Kalinin-Katok-Rodriguez}, Kalinin-Katok-Rodriguez Hertz obtain the following: for a locally free abelian Anosov action with $\dim \mc{M}=2\kappa+1$ with $\kappa\geq 2$, an invariant ergodic measure $\mu$ which has positive entropy for some $A\in \mathbb \R^\kappa$ is absolutely continuous under certain assumptions on the Lyapunov exponents and hyperplanes of $\mu$ (it is thus the same as our SRB measure).

Independently to our work and with different methods, Carrasco and Rodriguez-Hertz \cite{CaRH21} (see also \cite{CRH21, CRH21b} for a review of the results and further applications) have  constructed equilibrium measures for center isometries and proved the existence and uniqueness of an invariant measure that has absolutely continuous conditional measures with respect to $W^s_{\textup{loc}}$. If their center isometry stems from a regular element of a general Anosov action, they show that there is a unique measure fulfilling the condition (\ref{it:abs_cont}) of Theorem 1. This provides an alternative construction of the same measure.
They use geometric methods and first construct leaf-wise measures that they combine in a sophisticated way to an invariant measure.
An advantage of their method is that it requires less assumptions (they only assume $C^2$-regularity and do not require that the center isometry stems from an Anosov action) and they also construct a measure of maximal entropy and prove the Bernoulli property. On the other hand the characterisations (1), (2) and (4) of Theorem~\ref{thm:SRB_intro} are not shown in \cite{CaRH21}. Another aspect of our spectral/microlocal approach that is absent in \cite{CaRH21} is a Bowen-like \cite{bowen1972,Parry-Pollicott-90} formula for the measure $\mu$, showing equidistribution of regular periodic orbits in arbitrary small cones for higher rank actions. 

In order to state the second main result of this paper let us first introduce some notation: A point $x\in \mc{M}$ is said to be a periodic point if there exists $A\in\R^\kappa\setminus \{0\}$ such that $\tau(A)(x)=x$. Periodic orbits may have a complicated shape in general, but it is well-known (see e.g. \cite[Sec. 2, Prop. 1]{spatzierthesis} for a statement for Weyl chamber flows)\footnote{for the sake of self completeness we also provide a short proof in  Lemma~\ref{lem:torus}} that if $\tau(A_0)(x)=x$ for some $A_0\in \mc{W}$, then the orbit set $T=T_x:=\{\tau(A)(x)\in \mc{M}\,|\, A\in \R^\kappa\}$ is a $\kappa$-dimensional torus -- we say that the orbit is regular. We denote by $\mc{T}$ the set of such periodic tori of $\tau$ and, for $T\in \mc{T}$,
we denote by $L(T):=\{A\in \R^\kappa\,|\,\exists x\in T: \tau(A)(x)=x\}$ the associated lattice. Pushing forward the Lebesgue measure on $\R^\kappa/L(T)$ to $T\subset \M$, we obtain a natural measure on each torus orbit which we denote by $\lambda_T$.

\begin{theorem}\label{thm:formuledebowenintro}
Let $\tau$ be a transitive $\R^\kappa$-Anosov action, with Weyl chamber $\W$. Let $\mc{C}\subset\W$ be a proper subcone and $\eta\in {\R^\kappa}^*$ a dual element that is positive on a slightly larger conic neighbourhood of $\mc C$. Define $\mc{C}_{a,b}:=\{A\in \mc{C}\,|\, \eta(A)\in [a,b]\}$ if $a,b>0$. Let $\mu$ be the SRB measure and $a,b>0$. Then for each $f\in C^0(\M)$, we have
\begin{equation}\label{eq:bowenformula}
\mu(f)=\lim_{N\to \infty}\frac{1}{|\mc{C}_{aN,bN}|} \sum_{T\in \mc{T}}
\sum_{A\in \mc{C}_{aN,bN}\cap L(T)}\frac{\int_{T}f\,\,d\lambda_T}{|\det(1-\mc{P}_{A})|}
\end{equation}
where $\mc{P}_A$ is the linearized Poincar\'e map of the periodic orbit $A$ restricted to $E_u\oplus E_s$.
\end{theorem}

This result is proved using microlocal methods inspired by Dyatlov-Zworski \cite{DZ16a} in the rank $1$ case:
one needs in our setting to combine the Guillemin trace formula with the analysis of the wavefront set of a certain meromorphic function $F_\lambda(X_1,\dots,X_\kappa)$ of the family of commuting vector fields $(X_1,\dots, X_\kappa)$ generating the Anosov action, and this function has a simple pole at $\lambda=0$ with residue given by $\mu(f)$.
The result \eqref{eq:bowenformula} shows some equidistribution of the periodic orbits just as in the rank $1$ case, except that here the periodic orbits come as $\kappa$-dimensional tori.
Note that the most prominent examples of an Anosov action are the Weyl chamber flows for locally symmetric spaces. If ${\bf G}$ is a real semisimple Lie group with Iwasawa decomposition ${\bf G}={\bf KAN}$, these Weyl chamber flows are given by the right ${\bf A}$ action on $\mc{M}=\Gamma\backslash {\bf G}/Z_{\bf K}({\bf A})$ (with $Z_{\bf K}({\bf A})$ the centralizer of ${\bf A}$ in ${\bf K}$)  and the SRB measure is the Haar measure, by uniqueness. Thus Theorem \ref{thm:formuledebowenintro} gives an expression of the Haar measure in terms of periodic tori: by \eqref{det1-P_A}, there is $\eps>0$ so that for all $A\in L(T)\cap \mc{C}$,
$\det(1-\mc{P}_A)=e^{2\rho(A)}(1+\mc{O}(e^{-\eps|A|}))$ where $\rho$ is the half sum of the positive roots of the abelian Lie algebra $\mathfrak a$ of $A$. Therefore \eqref{eq:bowenformula} reduces to
\begin{equation}\label{haarbowen}
\mu(f)=\lim_{N\to \infty}\frac{1}{|\mc{C}_{aN,bN}|} \sum_{T\in \mc{T}}
\sum_{A\in \mc{C}_{aN,bN}\cap L(T)}e^{-2\rho(A)}\int_{T}f\,d\lambda_T.
\end{equation}
We notice that even for locally symmetric space where $\mu$ is the Haar measure, the identities \eqref{eq:bowenformula}
and \eqref{haarbowen} were not known, and our result is new also in
that setting.

As a rather direct consequence of \eqref{eq:bowenformula} we get the following result on the counting of periodic tori:

\begin{cor}\label{cor:comptage}
Assume there is a linear form $\eta\in {\R^\kappa}^*$ that is positive on $\W$ and such that for any proper subcone $\mc C\subset \W$ there is $\varepsilon>0$ such that $|\det(1-\mc P_A)| = e^{\eta(A)}(1-\mc{O}(e^{-\varepsilon|A|}))$ for all $A\in \mc{C}$.
 For any proper subcone $\mc C\subset \W$ let $\mc C_N:= \{A\in \mc C, \eta(A)/\|\eta\|\leq N\}$ then
\[
  \lim_{N\to \infty}\frac{1}{N}\log\Big(\sum_{T\in\mc T}\sum_{A\in L(T)\cap \mc C_{N}} \la_T(T)\Big)= \|\eta\|.
\]
\end{cor}
Note that we will show in the discussion after Proposition~\ref{countingprop} that the above assumptions are fulfilled for all standard Anosov actions (as introduced in \cite[Sec. 2]{KaSp94}). In the special case of Weyl chamber flows, Spatzier \cite{spatzierthesis} proved a related result when the cone $\mc{C}$ is the whole Weyl chamber: more precisely he proved that if
$s(T):=\min \{|A|\,| A\in L(T)\cap \mc{W}\}$ denotes the regular systole of a periodic torus $T$, then
\[
\lim_{N\to \infty}\frac{1}{N}\log\Big(\sum_{T\in\mc{T}, s(T)\leq N}\la_T(T)\Big)=2\|\rho\|.
\]
Recall from above that for Weyl chamber flows one has $\eta=2\rho$ and $2\|\rho\|$ corresponds also to the topological entropy of the associated geodesic flow. The same asymptotics for torus orbits of Weyl chamber flows have been obtained by Deitmar \cite{Dei04} (yet with slightly different counting region) using trace formulae on higher rank locally symmetric spaces and Lefschetz formulae.

Another aspect of our method is that, a byproduct of the proof of Theorem \ref{thm:formuledebowenintro}, we can construct some zeta-like functions (see Theorem \ref{thm:zeta-function}). For each function $\psi\in C^\infty_c(\W)$ with small enough support, we obtain a function $d_\psi(\lambda)$ holomorphic on $\C^\kappa$ that vanishes exactly when $\hat{\psi}(\lambda-\zeta) =1$ for some Taylor-Ruelle resonance $\zeta$ of the action (as was introduced in \cite{BGHW20}). Here $\hat{\psi}$ is the Laplace transform of $\psi$. As far as we know this is the first example of a globally holomorphic zeta-like function for higher rank actions.

Let us mention three results that are related to the Bowen formula in Theorem~\ref{thm:formuledebowenintro} for the special case of the Anosov action being a Weyl chamber flow on a compact locally symmetric space: Knieper \cite{Knieper05} studies the measure of maximal entropy for geodesic flows on compact locally symmetric spaces and showed its uniqueness. From this uniqueness he derives a Bowen formula for $\epsilon$-separated geodesics. Furthermore, Einsiedler, Lindenstrauss, Michel and Venkatesh studied the distribution of torus orbits of Weyl chamber flows in \cite{ELMV09, ELMV11}. In the special case of Weyl chamber flows on $SL(3,\R)/SL(3,\Z)$ they obtain a strong equidistribution result of periodic torus orbits \cite[Theorem 1.4]{ELMV11}  that among others would imply the Bowen type formula above\footnote{Note however that our result does not hold for $SL(3,\R)/SL(3,\Z)$ due to the non compactness of this space}. In \cite{ELMV09} the authors also study  torus orbits on certain compact locally symmetric spaces that are constructed from orders in central simple algebras. They also obtain equidistribution results (see \cite[Corollary 1.7]{ELMV09}) which are, however, weaker than those obtained for $SL(3,\R)/SL(3,\Z)$ and they seem not to imply Theorem~\ref{thm:formuledebowenintro} for this special class of compact locally symmetric spaces. Finally, after the appearance of our paper on arxiv, Dang and  Li \cite{DL22} posted a very interesting preprint where they study equidistribution and counting results of periodic torus orbits for Weyl chamber flows on locally symmetric spaces. They use completely different methods (like structure theory of symmetric spaces, compactifications, Patterson-Sullivan measures etc) and obtain a  result similar to Theorem~\ref{thm:formuledebowenintro} using the refined group orbit counting estimates of Gorodnick-Nevo \cite{GoNe12}.
While their techniques do not allow to localize in arbitrary small subcones for the average, they obtain \eqref{eq:bowenformula} without the determinant factors and with an exponential remainder term. The exponential remainder term is a consequence of property T for such groups, 
and is strongly dependent on the fact that the space is locally symmetric. 
Our spectral approach 
could lead to such an exponential remainder provided one could prove a spectral gap for the Ruelle resonance spectrum. 
Outside the world of locally symmetric space, it is not yet clear what kind of assumptions are needed to prove such a gap, and this difficult question is under investigation; we recall that for Anosov flows, the contact assumption is sufficient \cite{Liv04} to get such a gap. 
Let us finally emphasize that all the above mentioned results use techniques that are restricted to the case of Weyl chamber flows on locally symmetric spaces, whereas Theorem~\ref{thm:formuledebowenintro} holds for arbitrary Anosov actions.\\

Before closing this introduction, let us briefly mention the tools and techniques we employ in this work. We build on
our previous work \cite{BGHW20} using microlocal methods in the spirit of Faure-Sj\"ostrand and Dyatlov-Zworski \cite{FS11,DZ16a} in the
framework of anisotropic spaces (developed originally in dynamical systems by Blank, Keller, Liverani, Baladi, Tsujii, Gou\"ezel, Butterley \cite{BKL02,Gouezel-Liverani06,BL07,BT07}).  These techniques have a successful history in the context of Anosov flows, and we use them intensively in this work. For the proof of Theorem \ref{thm:SRB_intro}, it is sufficient to be familiar with the notion of H\"ormander wavefront set. For the proof of Theorems \ref{thm:formuledebowenintro} and \ref{thm:zeta-function}, however, we assume that the reader is somewhat familiar with more involved techniques, as were used for example in \cite{DZ16a,FRS08}. We will also be using some classical techniques from the study of dynamical foliations (absolute continuity, Rokhlin disintegrations...).\\

\textbf{Outline of the paper.}

In Section~\ref{sec:AnosovFoliations} we introduce and collect different preliminaries which are later used in order to state and prove the main theorems:
In Section~\ref{sec:anosov_action} we give the definition of $\R^\kappa$-Anosov actions and introduce some related basic notations.

In Section~\ref{subsec:AnosovFoliations} we collect and discuss crucial properties of the stable and unstable foliations related to Anosov actions which we shall need in the sequel. In particular we give a proof that the conditional densities of Lebesgue measure along the weak-(un)stable foliations are smooth along the orbits. While this fact seems folklore, we couldn't find a precise reference and as we crucially need this in order to apply our microlocal methods, we took the effort to work this out in details.

In Section~\ref{sec:spec_theory} we  recall how invariant measures for Anosov actions can be constructed using the spectral theory of Ruelle-Taylor resonances as presented in \cite{BGHW20} and we prove some additional facts that facilitate to work with these measures later on.

Section~\ref{sec:basin} and  Section~\ref{sec:bowen} are the core of the paper:

Section~\ref{sec:basin} contains the proofs for the different equivalent characterisations of the invariant  measures (Theorem~\ref{thm:SRB_intro}). This is done in two steps: in Section~\ref{sec:SRB} we prove that the measures obtained by the spectral theory are precisely the SRB measures i.e. measures which are absolutely continuous with respect to stable foliation (Theorem~\ref{thm:absolutecontinuity}). In Section~\ref{sec:physical} we analyze the basins of attractions that have positive Lebesgue measures and prove Theorem~\ref{thm:physical}. Note that neither Theorem~\ref{thm:absolutecontinuity} nor Theorem~\ref{thm:physical} assume transitivity of the action, thus the statements are more general and more precise (but also a little bit more technical to formulate) compared to Theorem~\ref{thm:SRB_intro} that is stated in the introduction above.

In Section~\ref{sec:bowen} we use microlocal analysis and a higher rank Guillemin trace formula to prove the Bowen formula (Theorem~\ref{thm:formuledebowenintro}).

Finally in Section~\ref{sec:counting} we shortly discuss the applications to counting of periodic tori.\\

We also prove some new statements in this context such as  that will allow us to show that the measures defined by spectral theory are always absolutely continuous along the stable foliation.

\textbf{Acknowledgements.}
We warmly thank the $6$ referees for their careful reading, their numerous comments and suggestions in order to improve the paper. We thank S\'ebastien Gou\"ezel for helpful discussions and explaining us the arguments in the proof of Proposition~\ref{prop:smooth_disintegration_Lebesgue}. 
This project has received funding from the European Research Council (ERC) under the European Union’s Horizon 2020 research and innovation programme (grant agreement No. 725967) and from the Deutsche Forschungsgemeinschaft (DFG, German Research Foundation) via the Grands WE 6173/1-1  (Emmy Noether group ``Microlocal Methods for Hyperbolic Dynamics'') as well as SFB-TRR 358/1 2023 — 491392403 (CRC ``Integral Structurs in Geometry and Representation Theory'').

\section{Anosov actions, dynamical foliations and invariant measures}\label{sec:AnosovFoliations}

To study regularity of functions and distributions, we will rely on microlocal techniques in the spirit of \cite{DZ16a}. In particular, we will use the notion of wavefront set, see \cite[Section 8.1]{Hoe03} for an introduction and properties of wavefront set.

\subsection{Anosov actions}\label{sec:anosov_action}

Let $(\mc M, g)$ be a closed Riemannian manifold, $\tau : \mathbb A \to \Diffeo(\M)$ be a locally free action of an abelian Lie group $\mathbb A\cong \R^\kappa$ on $\M$. Let $\mathfrak a:= \tu{Lie}(\mathbb A)\cong \R^\kappa$ be the associated commutative Lie algebra and $\exp:\mathfrak a\to\mathbb A$ the Lie group exponential map. After identifying $\mathbb A\cong \mathfrak a\cong \R^\kappa$, this exponential is the identity, but it will be useful to have a coordinate-free notation that distinguishes between transformations $\mathbb A$ and infinitesimal transformations $\mathfrak a$.
Taking the derivative of the $\mathbb A$-action one obtains an injective Lie algebra homomorphism
\begin{equation}\label{XAnosov}
 X:\Abb{\mathfrak a}{C^\infty(\M;T\M)}{A}{X_A:=\tfrac{d}{dt}_{|t=0}\tau(\exp(At))}
\end{equation}
which we call the infinitesimal action.
By commutativity of $\mathfrak a$, $\ran (X)\subset C^\infty(\M;T\mc M)$ is a $\kappa$-dimensional subspace of commuting vector fields. Since the action is locally free, they span a $\kappa$-dimensional smooth subbundle which we call the \emph{neutral subbundle} or \emph{center subbundle} $E_0\subset T\mc M$. It is tangent to the $\mathbb A$-orbits on $\mc M$. We will often study the one-parameter flow generated by a vector field $X_A$ which we denote by $\varphi^{A}_t$. One has the obvious identity $\varphi^{A}_t = \tau(\exp(At))$ for $t\in \R$.
\begin{Def}\label{def:anosov_action}
An element $A\in \mathfrak a$ and its corresponding vector field $X_A$ are called \emph{transversely hyperbolic} if there is a continuous splitting
\begin{equation}\label{eq:invariant-splitting}
 T\M = E_0\oplus E_u\oplus E_s,
\end{equation}
that is invariant under the flow $\varphi^{A}_t$ and such that there are $\nu>0, C>0$ with
\begin{equation}
 \label{eq:stable}
 \|d\varphi_t^{A}v\|\leq Ce^{-\nu |t|}\|v\|, \quad  \forall v\in E_s,~ \forall t\geq 0,
\end{equation}
\begin{equation}\label{eq:unstable}
 \|d\varphi_t^{A}v\|\leq Ce^{-\nu |t|}\|v\|, \quad  \forall v\in E_u,~ \forall t\leq 0.
\end{equation}
We say that the $\mathbb A$-action is \emph{Anosov} if there exists an $A_0\in \mathfrak a$ such that $X_{A_0}$ is transversely hyperbolic. Observe that if both $E^u$ and $E^s$ are trivial bundles, then $\M$ must be a torus, or a finite quotient thereof. We will rule out this case in all our arguments. 
\end{Def}
We define the dual bundles $E_u^*,E_s^*, E_0^*\subset T^*\M$ by\footnote{Note that $E_{s/u}^*$ are not the usual dual bundles of $E_{s/u}$ that vanish on $E_{u/s}\oplus E_0$. The notation that we use has grown historically in the semiclassical approach to Ruelle resonances and is justified by the fact that the symplectic lift of $\tau$ to $T^*\mc M$ is expanding in the $E_u^*$ fibre and contracting in the $E_s^*$ fibre.}
\begin{equation}\label{dualbundle}
E_u^*(E_u\oplus E_0)=0, \quad E_s^*(E_s\oplus E_0)=0, \quad E_0(E_u\oplus E_s)=0.
\end{equation}
Throughout the paper, we will assume that there exists such a transversely hyperbolic element and denote it $A_0\in\mathfrak{a}$. 
We define the corresponding \emph{positive Weyl chamber} $\mathcal W\subset \mathfrak a$
to be the set of  $A \in \mathfrak a$ which are transversely hyperbolic with the same stable/unstable bundle as $A_0$. The following statement is well known -- a proof can for example be found in \cite[Lemma 2.2]{BGHW20}.
\begin{lemma}
Given an Anosov action and a transversely hyperbolic element $A_0\in \a$, the corresponding positive Weyl chamber $\mathcal W\subset \mathfrak a$ is an open convex cone. It is the maximal open convex cone containing $A_0$ and comprised only of transversely Anosov elements.
\end{lemma}

Since we assume that $E_u$, $E_s$ are not trivial, $\mathcal{W}$ cannot be the whole of $\mathfrak{a}^\ast$. For the record, let us mention that when the set of transversely Anosov elements is dense in $\mathfrak{a}$, the action is called \emph{totally Anosov}. In that case, there is a finite number of Weyl chambers, which are separated by hyperplanes (see \S4 in \cite{SV19}). Note that there are different concrete constructions of Anosov actions and we refer to \cite[Section 2.2]{KaSp94} for examples.

Some very classical notions of $\Z$ or $\R$ dynamics have natural extensions to group dynamics; since the reader may not have already encountered them, we recall the following:
\begin{Def}
Let $\tau$ be an $\R^\kappa$ action on some compact manifold $\mc{M}$. 
\begin{itemize}
	\item One says that $\tau$ is (topologically) transitive if there exists $x\in \mc{M}$ with dense orbit. It is equivalent to require that for any open sets $U,V$, there exists $A\in \R^\kappa$ such that $\tau(A)(U)\cap V\neq \emptyset$. 
	\item If $\mu$ is a $\tau$-invariant Radon probability measure on $\mc{M}$, $\mu$ is said to be ergodic if the only $\tau$-invariant $f\in L^1(\mu)$ are constant.
\end{itemize}
\end{Def}

As usual for a compact hyperbolic dynamical system, the choice of the smooth Riemannian metric $g$ is not intrinsic to the Anosov action: $A_0$ would be an Anosov element of the action $\tau$ for any other choice of a smooth Riemannian metric $g$ with the same hyperbolic splitting. We nevertheless fix the metric $g$ once and for all. This does not only give us  a Riemannian distance function to measure distances and regularity of functions as well as norms $\|\cdot\|$ on $T\M$ and $T^\ast \M$, respectively. It also fixes a volume form $v_g$ --- which we will assume to be normalized with volume $1$, but which will in general not be invariant under the action. More generally $g$ also endows any smooth submanifold of $\mc{M}$ with the corresponding volume form.

Throughout the paper, using our smooth volume, we will identify $C^\infty(\mc{M})$ as a subspace of its dual $\mathcal{D}'(\mc{M})$. We will use the duality bracket defined originally for $u,v\in C^\infty$ by
\[
\langle u, v\rangle = \int_M uv dv_g,
\]
and extended by density whenever it makes sense.

All the important objects will be invariant, or at least equivariant, under changing $g$.

\subsection{Dynamical foliations and absolute continuity}\label{subsec:AnosovFoliations}

Since the point of this article is to study in detail the SRB measure of Anosov actions, we will have to consider disintegration of measures along stable and unstable foliations. It will be crucial that these foliations are \emph{absolutely continuous}. This fact is well established. However, for our purposes, we will need that some conditional densities are $C^\infty$. This seems to be folklore, but we have not found a complete proof written down. We have thus decided to recall the relevant definitions, and explain how the regularity of the conditional measures can be derived from existing results in the literature.

\begin{Def}
Let $F$ be a partition of $\mc M$ and, given $m\in \mc{M}$, let $F(m)$ be the unique element in $F$ containing $m$. Given a neighbourhood $U$ of $m$ denote by $F_{\mathrm{loc}}(m)$ the connected component of $F(m)\cap U$ containing $m$.

The partition $F$ is called a \emph{continuous (resp. H\"older) foliation with n-dimensional $C^k$-leaves} ($k\in \N\cup \{\infty\}$) if for any $m\in \mc M$ there is a neighbourhood $U\subset \M$ and a continuous (resp. H\"older) map $f: U\to C^k(D^n, \mc M)$ such that for any $\tilde m\in U$, $f(\tilde{m})$ is a diffeomorphism of the $n$-dimensional unit disk $D^n$ onto $F_{{\rm loc}}(\tilde m)$. In that case we denote by $TF$ the subbundle of $T\M$ given by $TF(y) = T_y(F(y))$.

A particular case: the foliation is called a \emph{$C^k$ foliation} if for any $m$ there is a neighbourhood $U$ and a $C^k$ chart $\psi: U\to D^n\times D^{\dim \M-n}$ with $F_\mathrm{loc}(\psi^{-1}(0,y)) = \psi^{-1}(D^n\times\{y\})$. Notice that leaves are then $C^k$.
\end{Def}

In the following all foliations will be H\"older with $C^\infty$ leaves, if they are not outright $C^\infty$.\\

It turns out that for $t>0$, $\varphi^{A_0}_t$ is an example of a partially hyperbolic diffeomorphism. More precisely in the terms of Pesin \cite[\S 2.2]{Pes04}, it is \emph{partially hyperbolic in the narrow sense} with respect to the splitting \eqref{eq:invariant-splitting}. We can thus apply some classical results on such dynamics.

By the stable manifold theorem for partially hyperbolic diffeomorphisms (see e.g. \cite[Theorem 4.1]{Pes04} we get for any $m\in \mc M$ a unique $n_s$-dimensional immersed  $C^\infty$-submanifold $W^{s}(m)\subset \mc{M}$ which is tangent to the stable foliation (i.e. $T_x(W^s(m))=E_s(x)$). We call $W^s(m)$ the \emph{stable manifold} of $m\in \mc{M}$ and there exists $\nu>0$ such that
\begin{equation}\label{eq:stable_mfld}
 W^s(m) = \left\{m'\in \M\, \middle|\, \exists\ C>0,\  d_g(\varphi_t^{{A_0}}(m'),\varphi_t^{{A_0}}(m)) \leq  Ce^{-\nu t} \text{ for all }t>0\right\}.
\end{equation}
It is known that the partition of $\mc{M}$ into stable manifolds is a H\"older foliation with $C^\infty$ leaves of the manifold $\mc M$, called the \emph{stable foliation}. Note that by \eqref{eq:stable_mfld} and the commutativity of the Anosov action, we directly deduce that the foliation into stable leaves is invariant under the Anosov action, i.e. for all $a\in \mathbb A$, $\tau(a)(W^s(m))=W^s(\tau(a)(m))$. This also implies that picking a different $A_0\in \W$ would give the same foliation. We can define the \emph{weak stable manifolds}
\begin{equation}\label{eq:weak_stable_subfoliation}
W^{ws}(m)=\bigcup_{a\in \mathbb{A}} W^{s}(\tau(a)(m)) = \bigcup_{y\in W^s(m)}\tau(\mathbb A)y.
\end{equation}
They are immersed submanifolds tangent to the neutral and stable directions, i.e. $T_x(W^{ws}(m)) = E_0(x)\oplus E_s(x)$. By construction the weak stable manifolds provide again a H\"older foliation of $\mc M$ with $C^\infty$-leaves of dimension $n_{ws}=n_s+\kappa$. Precisely the same way one can define the \emph{unstable manifolds} $W^{u}(m)$ and the \emph{weak unstable manifolds} $W^{wu}(m)$ and they provide foliations with the same properties. We can also fix a symbol for the orbit foliation $W^0(m):= \tau(\mathbb{A})m$. Note that despite the fact that all foliations have $C^\infty$-leaves, none of these dynamical foliations is expected to be a $C^\infty$-foliation (or even a $C^1$ foliation) in general --- except $W^0$. See \cite{BFL92} for an example of what is expected to happen when one assumes smoothness of such foliations.

In order to discuss the disintegration of measures along foliations let us introduce product neighbourhoods. We consider given $F$ and $G$ two continuous foliations with smooth leaves and assume they are complementary (i.e $TM= TF\oplus TG$). For $\delta>0$ we denote by $B_F(m, \delta)\subset F(m)$ the ball of radius $\delta$ around $m$ inside the leaf $F(m)$. Then for any $m\in \mc M$ there is a $\delta>0$, a neighbourhood $U$ called \emph{product neighbourhood} (see \cite[Theorem 3.2]{PS70}) such that the following map is a homeomorphism
\begin{equation}\label{eq:product_nbhd}
  P:\Abb{B_F(m,\delta)\times B_G(m,\delta)}{\quad\quad U}{{\quad\quad\quad\quad(x,y)}}{ G_\mathrm{loc}(x)\cap F_\mathrm{loc}(y)}.
\end{equation}
Given such a product neighbourhood $U$, we can introduce the Rokhlin disintegration of measures along $F$ in $U$.
\begin{prop}[Rokhlin's theorem \cite{Rok49}]\label{prop:roklin}
For any Borel probability measure $\mu$ on $U$ there exists a measure $\hat \mu$ on $B_G(m,\delta)$ and a measurable family of probability measures
$\mu_y$ on $F_\mathrm{loc}(y)$, called conditional measures, so that for $f:U\to \C$ in $L^1(\mu)$,
\begin{equation}
 \label{eq:rokhlin}
\int_U f d\mu  = \int_{B_G(m,\delta)} \Big(\int_{F_\mathrm{loc}(y)}  f(x)  d\mu_y(x)\Big)d\hat\mu(y).
\end{equation}
The $\mu_y$ are unique $\mu$-almost surely.
\end{prop}
The $\mu_y$ are called the conditional measures on the leaves $F_\mathrm{loc}(y)$.
Note that by \eqref{eq:rokhlin} one has that $\hat \mu$ is the pushforward of $\mu$ under the projection $U\cong B_F(m,\delta)\times B_G(m,\delta)\to B_G(m,\delta)$. Furthermore by the proof of Rokhlin's theorem (see for example \S5.2 in \cite{VO16}) one gets a description of the conditional measures $\mu_y$. Let us therefore introduce the $F$-tubes
\[
 \mc T_F(y,\varepsilon) := P(B_F(m,\delta)\times B_G(y,\varepsilon)) \subset U.
\]
Then for $\hat \mu$ almost all $y\in B_G(m,\delta)$ the limit
\[
 \lim_{\varepsilon\to 0} \frac{\mathds 1_{\mc T_F(y,\varepsilon)}\mu}{\mu(\mc T_F(y,\varepsilon))}
\]
exists as a weak limit of probability measures on $U$. Obviously the limit is a probability measure supported on $F_\mathrm{loc}(y)$. It coincides with the conditional measures $\mu_y$ (for the points $y$ where the limit may not exist the measures $\mu_y$ can be chosen arbitrarily as they are a $\hat \mu$-zero set).
\begin{Def}\label{def:abscont}
Let $F$ be a continuous foliation on $\mc M$ with $C^1$ leaves. We say that a measure $\mu$ has an absolutely continuous disintegration if $\mu$ can be disintegrated in all product neighbourhoods (with an arbitrarily chosen local smooth transversal foliation $G$ of complementary dimension)
such that all the conditional measures $\mu_{y}$ are absolutely continuous with respect to the Riemannian volume measures on the local leaves $F_\mathrm{loc}(y)$. Being absolutely continuous does not depend on the choice of Lebesgue measures.

We call the foliation itself $F$ is \emph{absolutely continuous} if the Riemannian volume measure $v_g$ on $\mc M$ has an absolutely continuous disintegration
\end{Def}

If $G$ is a foliation with $C^1$ leaves, the Riemannian volume form obtained by restriction of the ambient metric to the leaves will be denoted $L^G$. More concisely, we will write $L^\ast:=L^{W^\ast}$ for $\ast \in \{ u, s, 0, wu, ws\}$.\\

Note that if the foliation is $C^k$ with $k>1$ then, by Fubini's theorem, the foliation is absolutely continuous and the conditional measures have $C^{k-1}$ densities. It is worthwhile to note that if the foliation is not smooth anymore but only the leaves are, then absolute continuity does not hold in general. There are indeed examples of H\"older foliations with smooth leaves that are \emph{not} absolutely continuous (see e.g. \cite[Section 7.4]{Pes04}).
Thankfully, the stable and unstable foliations of Anosov actions are absolutely continuous. This follows from the very general statement \cite[Theorem 7.1]{Pes04}, which concerns (un)stable foliations of partially hyperbolic diffeomorphisms. However, we need a more precise statement.

Given a manifold $N$ with a locally finite atlas of charts, one can measure $C^k$ regularity of functions using the usual $C^k$ norms in restriction to the charts, to provide $C^k(N)$ with a norm $\|\cdot\|_{C^k,{\rm charts}}$. On the other hand, if $N$ is endowed with a Riemannian metric $g$, one can use the Levi-Civita connection $\nabla$ to define covariant norms, $\|\cdot\|_{C^k,\nabla}$. Recall that whenever $g$ and $g^{-1}$ are $C^\infty$ bounded in each chart, with constants that do not depend on the chart, $\|\cdot\|_{C^k,{\rm charts}}$ and $\|\cdot\|_{C^k,\nabla}$ are equivalent. One can find such charts provided the Riemann tensor of $g$ and all its covariant derivatives are bounded. In particular, if $G$ is a continuous foliation with $C^\infty$ leaves in a smooth compact manifold $(N,g)$, the restriction of $g$ to each leaf $G(m)$ must satisfy such estimates, with constants independent of $m$; it thus makes sense to talk of $C^k$ norms of functions on leaves $G(m)$ without specifying exactly how they were computed.
\begin{Def}\label{def:unif-smooth-foliation} Let $G$ be a foliation with smooth leaves, $U$ an open set. Let $f:U\to \C$ a measurable map. We will say, that $f$ is uniformly smooth if for almost all $y\in U$ $f_{|G_{{\rm loc}}(y)}\in C^\infty(G_{{\rm loc}}(y))$ and if for all $k\in \N$
\[
\mathrm{esssup}_{y\in U} \| f_{|G_{{\rm loc}}(y)} \|_{C^k(G_{{\rm loc}}(y))} < \infty.
\]
\end{Def}

\begin{prop}\label{prop:smooth_disintegration_Lebesgue}
Let $X$ be a smooth\footnote{If we assume that $X$ acts on $C^{1+\alpha}$, with $\alpha>0$, we would probably obtain $C^\alpha$ density; it is a classical observation that dynamical foliations of \emph{only} $C^1$ hyperbolic flows need not have absolutely continuous foliations.} Anosov action, and let $W^s,W^u$ be the associated stable and unstable foliations. Then, $W^s$ and $W^u$ are absolutely continuous in the sense of Definition \ref{def:abscont}. Moreover if $U\subset \mc M$ is a product neighbourhood around $m\in U$ of $W^{s/u}$ and an arbitrary smooth transversal complementary foliation $G$, and $v_g$ the Riemannian volume measure on $U$ then there is a continuous function $\delta_{W^{s/u}}:U\to \R^+$ such that
\begin{equation}\label{eq:smooth_decomposition_lebesgue}
 \int_U f dv_g = \int_{G_\mathrm{loc}(m)} \left(\int_{W^{s/u}_\mathrm{loc}(y)} f(z) \delta_{W^{s/u}}(z) dL^{s/u}_y(z)\right) dL^{G}_m(y).
\end{equation}
Furthermore $\delta_{W^{s/u}}$ is uniformly smooth along the leaves $W^{s/u}_\mathrm{loc}(y)$. 
\end{prop}

This smoothness seems to be folklore among dynamical systems specialists, but as the statement is not written down explicitly and is important for our further analysis, we explain how it can be deduced from existing results in the literature. Note that in the statement above, the projected measure $\hat{\mu}$ of Proposition \ref{prop:roklin} has also been identified; it is absolutely continuous.
\begin{proof}
In order to simplify the notation we restrict ourselves to the case of the stable foliation $W^s$. We follow the standard approach to express the density function $\delta_{W^s}$ by holonomies and their Jacobians.

\begin{figure}[h]
\centering
\def\svgwidth{0.4\linewidth}
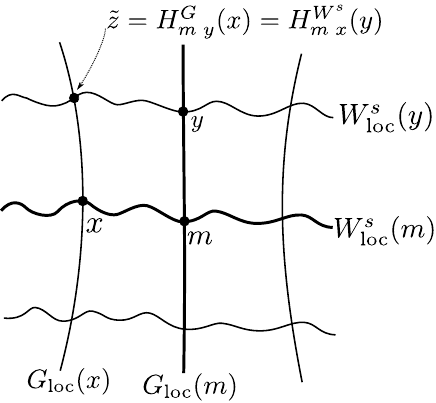
\caption{\label{fig:holonomy-map} In a product neighbourhood of two complementary transversal foliations (e.g. $W^s$ and $G$ in this sketch) one can define the holonomies along the respective foliations. In the proof of Proposition \ref{prop:smooth_disintegration_Lebesgue} a crucial point is that one can first use a Fubini theorem w.r.t. the smooth transversal foliation $G$ and then transform this into a disintegration along $W^s$ by the use of holonomies and their Jacobians.}
\end{figure}

Let us consider around a point $m\in \mc M$ a local $C^\infty$-foliation $G$ that is transversal to $W^s$ and has complementary dimension $n_{wu}$. Let $U$ be a product neighbourhood of these transversal foliations. Now for any $x_1, x_2\in W^s_\mathrm{loc}(m)$ we define the following \emph{holonomy map} (cf. Figure~\ref{fig:holonomy-map}) along the leaves of the stable foliation
\[
 H^{W^s}_{x_1,x_2}: \Abb{G_\mathrm{loc}(x_1)}{G_\mathrm{loc}(x_2)}{ z}{W^s_\mathrm{loc}(z)\cap G_\mathrm{loc}(x_2).}
\]
As the stable foliation is not smooth in general, the holonomy maps are neither. But we have
\begin{prop}[{See e.g. \cite[Theorem 7.1]{Pes04}}]\label{holonomymap}
 The holonomy maps of the stable (and unstable) foliation are absolutely continuous, i.e. there is a measurable function $j^{W^s}_{x_1,x_2}$ on $G_{\mathrm{loc}}(x_1)$ called the Jacobian of the holonomy map such that
 \[
  (H^{W^s}_{x_1, x_2})_*\big( j^ {W^s}_{x_1,x_2}L^{G}_{x_1}\big) = L^{G}_{x_2}
 \]
\end{prop}
In the same manner one can introduce the holonomies along the foliation $G$, $H^G_{y_1,y_2}: W^s_\mathrm{loc}(y_1) \to W^s_\mathrm{loc}(y_2)$ and their Jacobians $j^G_{y_1,y_2}$. As the foliation $G$ is smooth and the leaves $W^s$ are smooth, these holonomy maps are in fact diffeomorphisms and their Jacobians are always defined via the differential.

With the absolute continuity of $H^{W^s}_{x_1, x_2}$ one can prove that $W^s$ is absolutely continuous and give an expression for the conditional densities. First, as $G$ is a smooth foliation we use Fubini's Theorem and write
\[
 \int_U f dv_g = \int_{W^s_\mathrm{loc}(m)}\left(\int_{G_\mathrm{loc}(x)}f(z)\delta_{G}(z) dL^{G}_x(z)\right)dL^s_m(x).
\]
Here, $\delta_{G}\in C^\infty(U)$ is a smooth density. Using the absolute continuity of the homeomorphism $H^{W^s}_{m,x}:G_\mathrm{loc}(m)\to G_\mathrm{loc}(x)$ we can transform the integral over $G_\mathrm{loc}(x)$ into an integral over $G_\mathrm{loc}(m)$
\[
\int_U f dv_g= \int_{W^s_\mathrm{loc}(m)}\left(\int_{G_\mathrm{loc}(m)}f(H^{W^s}_{m,x}(y))\delta_{G}(H^{W^s}_{m,x}(y)) j_{m,x}^{W^s}(y)dL^{G}_m(y)\right)dL^s_m(x).
\]
Next, we use Fubini's theorem to change the order of integration, and change variable $\tilde{z}:= H^{W^s}_{m,x}(y)=H^G_{m,y}(x)$ (cf. Figure~\ref{fig:holonomy-map}), so we can transform the integral over $W^s_\mathrm{loc}(m)$ into an integral over $W^s_\mathrm{loc}(y)$
\begin{equation}\label{eq:disintegration}
 \int_U f dv_g = \int_{G_\mathrm{loc}(m)} \Big( \int_{W^s_\mathrm{loc}(y)} f(\tilde z)\delta_{G}(\tilde z) j^{W^s}_{m, H^G_{y,m}(\tilde z)}(y)j^G_{y,m}(\tilde z)dL^s_y(\tilde z)\Big) dL^G_m(y)
\end{equation}
making appear the Jacobian $j^G_{y,m}$ of the holonomy map $H^G_{y,m}$. This proves the absolute continuity of the stable foliation and shows that the conditional densities on $W^s_\mathrm{loc}(y)$ are given by $\delta_{G}(\tilde z) j^{W^s}_{m,H^G_{y,m}(\tilde z)}(y)j^G_{y,m}(\tilde z)$. As $G$ was a smooth foliation $\delta_G$ and $j^G_{y,m}$ are smooth so it only remains to show that $h(y,\tilde{z}):=j^{W^s}_{m,H^G_{y,m}(\tilde{z})}(y)$ is a smooth function in $\tilde{z}\in W^s_\mathrm{loc}(y)$ depending continuously on $y\in G_\mathrm{loc}(m)$. However by \cite[Remark 7.2]{Pes04} there is an explicit formula for the Jacobian for partially hyperbolic diffeomorphisms. In order to shorten the notation we introduce $\Phi:= \varphi_1^{{A_0}}$ and we can express the Jacobian by \cite[(7.3)]{Pes04} as\footnote{Be aware that Pesin uses the inverted diffeomorphism $\Phi^{-1}$ in his formula but the numerator and denominator in his formula are also interchanged so that the formula agrees with the one that we use here.}
\[
 h(y,\tilde{z})= \prod_{k=0}^\infty\frac{\Big|\mathrm{Jac}\Big((d_{\Phi^k(y)}\Phi)_{\big|T_{\Phi^k(y)}\big(\Phi^k(G_\mathrm{loc}(y))\big)}\Big)\Big|}{{\Big|\mathrm{Jac}\Big((d_{\Phi^k(\tilde{z})}\Phi)_{\big|T_{\Phi^k(\tilde{z})}\big(\Phi^k(G_\mathrm{loc}(\tilde z))\big)}\Big)\Big|}}.
\]
In order to analyze the regularity of this infinite product we use a classical argument --- see for example the proof of \cite[Lemma 5.5]{delallave92}. Consider the expressions $\det((d_{\Phi^k(y)}\Phi)_{|T_{\Phi^k(y)}\Phi^k(G_\mathrm{loc}(m))})$ as functions on the Grassmannians: Let $\pi:\mathcal G\to\mc M$ be the Grassmannian bundle of $n_{wu}$-dimensional subspaces in $T\mc M$, which we endow with some auxiliary metric $\tilde{g}$. The diffeomorphism $\Phi$ has a natural lift $\tilde{\Phi}$ to $\mathcal{G}$. Furthermore we can define $\mathcal J: \mathcal G\ni (x,V) \mapsto |\det((d_x\Phi)_{|V})|\in \R_{>0}$ which is a smooth function.

The foliation $G$ defines a section $S:m \mapsto T_m G(m)\in \mc G$, so that we can write
\begin{equation}\label{eq:logj}
\log(h(y,\tilde{z})) = \sum_{k=0}^\infty \log \mc{J}(\tilde\Phi^ k(S(y))) -\log \mc J(\tilde\Phi^k(S(\tilde{z}))).
\end{equation}
The set  
\[
\Lambda:= \{ (x, E_u(x)\oplus E_0(x))\ |\ x\in\M\}\subset \mathcal{G}
\]
is $\tilde{\Phi}$ invariant. It is a partially hyperbolic set, in the sense that $T\mathcal{G}$ decomposes for $\xi\in \Lambda$ as
\[
T_\xi \mathcal{G} = E^\mathcal{G}_0(\xi) \oplus E^\mathcal{G}_u(\xi) \oplus E^\mathcal{G}_s(\xi).
\]
Here $E^\mathcal{G}_0(\xi)$ is given by the direction of the action and $E^\mathcal{G}_s(\xi)$ is tangent to the local stable manifold
\[
W^s(\xi) = \{ \xi'\in \Lambda \, |\,  x'=\pi(\xi')\in W^s_{{\rm loc}}(\pi(\xi)),\ \xi' \text{ close to } \xi\}.
\] 
Similarly $E^\mathcal{G}_u$ is tangent to the local unstable manifold
\[
W^u(\xi) = \{ \xi' \in\Lambda\, |\, x'=\pi(\xi')\in W^u_{{\rm loc}}(\pi(\xi)),\ \xi'\in \Lambda\}.
\]
The main observation that we will need is that, since $\tilde{\Phi}$ is smooth and contracting on $W^s_{{\rm loc}}$, we have estimates\footnote{This can be proved in local coordinates using Fa\`a di Bruno formula for example; for a more intrinsic way of defining the $C^k$ norm on maps between manifolds, one can consult \cite{Wittmann-2019} for example}
\begin{equation}\label{eq:smooth-contracting-iteration}
\| \tilde{\Phi}^N \|_{C^k(W^s_{{\rm loc}}(\xi))} < C_k e^{- \nu N},
\end{equation}
with constants uniform in $\xi\in\Lambda$. \\

We observe that since $\tilde{z}\in W^s(y)$, we have $S(\tilde{z}),S(y)\in W^s(E_u(y)\oplus E_0(y))$. In particular, this implies
\begin{equation}\label{eq:Grassm_hyperbolicity}
d_\mathcal G\Big(\tilde \Phi^k(S(y)),\tilde\Phi^k(S(\tilde{z})) \Big) \leq Ce^{-\nu k} d_{\mc G}\Big(S(y), S(\tilde{z})\Big).
\end{equation}
Here, the constants $C,\nu$ are global constants on the manifold.

Now the compactness of $\mathcal G$ implies that $\mc J$ is uniformly Lipshitz and thus the series in \eqref{eq:logj} converges absolutely which implies that $h(\cdot,\cdot)$ is a continuous function (in both variables). We now show that it is even differentiable w.r.t. $\tilde{z}$: let us therefore take a smooth curve $\gamma:(-\varepsilon,\varepsilon) \to W^s_{{\rm loc}}(y)$, $\gamma(0)=\tilde{z}$ and $\|\gamma'(t)\|\leq 1$ then
\begin{equation}
 \label{eq:deriv_of_j}
 \frac{d}{dt}_{|t=0} \log h(y,\gamma(t)) = -\sum_{k=0}^\infty d(\log \mc{J})\big[\frac{d}{dt}_{|t=0} \tilde\Phi^k(S(\gamma(t)))\big].
\end{equation}
Using \eqref{eq:smooth-contracting-iteration} we obtain the estimate
\[
 \left\|\frac{d}{dt}_{|t=0} \tilde\Phi^k(S(\gamma(t)))\right\|_{T_{S(x)}\mc G} \leq C'e^{-\nu k}
\]
and this time the uniform boundedness of $d\log \mc J$ ensures the absolute convergence of the right hand side of \eqref{eq:deriv_of_j} which implies that $\frac{d}{dt}_{|t=0} \log h(y,\gamma(t))$ exists and its value depends continuously on $y$.

To estimate derivatives of order $k$ at $\tilde{z}$ it suffices to estimate the derivative at $0$ of $h(y, \gamma(t))$, for maps $\gamma$ defined near $0\in \R^k$, with $\gamma(0) = \tilde{z}$. The computation will be the same as above, using \eqref{eq:smooth-contracting-iteration} again.
\end{proof}

Note that the proof of Proposition~\ref{prop:smooth_disintegration_Lebesgue} strongly depends on the fact that there is an exponential contraction
along the stable, resp. unstable manifold and it would fail when working directly on the weak-(un)stable foliation. Nevertheless thanks to the fact that the neutral foliation is a smooth foliation we can establish the same result as Proposition~\ref{prop:smooth_disintegration_Lebesgue} for $W^{ws}$. Let us start by showing there is a continuous density function $\delta_{W^{ws}}$ for the weak-stable foliations\footnote{Again the case of weak unstable foliation follows exactly the same way but we only focus on the weak stable case to simplify the notation.} and give an explicit expression in terms of  $\delta_{W^{s}}$ and some further quantities which we introduce now: by \eqref{eq:weak_stable_subfoliation} we have
\begin{equation}\label{eq:ws_subfoliation}
W^{ws/wu}_\mathrm{loc}(m) = \bigcup_{y\in W_\mathrm{loc}^{s/u}(m)}W^0_\mathrm{loc}(m)=\bigcup_{x\in W_\mathrm{loc}^{0}(m)}W^{s/u}_\mathrm{loc}(x),
\end{equation}
By the fact that $W^{s/u}(\tau(a)m) = \tau(a)(W^{s/u}(m))$ and the smoothness of the Anosov action both partitions of the leaf $W_\mathrm{loc}^{ws/wu}(m)$ are smooth foliations of $W_\mathrm{loc}^{ws/wu}(m)$ and by Fubini there are strictly positive, smooth  functions $\delta^{W^{ws/wu}_\mathrm{loc}(m)}_{W^0}\in C^\infty(W^{ws/wu}_\mathrm{loc}(m))$ and $\delta^{W^{ws/wu}_\mathrm{loc}(m)}_{W^{s/u}}\in C^\infty(W^{ws/wu}_\mathrm{loc}(m))$ such that
\begin{align}
\int_{W^{ws/wu}_\mathrm{loc}(m)} & f(y)dL^{ws/wu}_m(y)\nonumber \\
								&= \int_{W^{s/u}_\mathrm{loc}(m)} \left(\int_{W^0_\mathrm{loc}(y')} f(z) \delta^{W^{ws/wu}_\mathrm{loc}(m)}_{W^0}(z) dL^0_{y'}(z)\right) dL^{s/u}_m(y') \label{eq:delta_ws_0} \\
							&=\int_{W^{0}_\mathrm{loc}(m)} \left(\int_{W^{s/u}_\mathrm{loc}(y')}f(z) \delta^{W^{ws/wu}_\mathrm{loc}(m)}_{W^{s/u}}(z) dL^{s/u}_{y'}(z)\right) dL^0_m(y'). \label{eq:delta_ws_s}
\end{align}
Now, in the proof of Proposition~\ref{prop:smooth_disintegration_Lebesgue} we chose a transversal complementary smooth foliation $G$ such that $G_\mathrm{loc}(m) = W_\mathrm{loc}^{wu}(m)$.
With \eqref{eq:smooth_decomposition_lebesgue} and \eqref{eq:delta_ws_0} we obtain
\begin{align*}
\int_U f \,dv_g &= \int_{W^{wu}_{\mathrm{loc}}(m)} \left[\int_{W^s_{\mathrm{loc}}(y)} f(z) \delta_{W^s}(z) dL_y^s(z)\right] dL_m^{wu}(y)\\
				&= \int_{W^u_{\mathrm{loc}}(m)} \left(\int_{W^0_{\mathrm{loc}}(x)} \left[\int_{W^s_{\mathrm{loc}}(y)} f(z) \delta_{W^s}(z) dL_y^s(z)\right] \delta_{W^0}^{W^{wu}(m)}(y) dL_x^0(y)\right) dL_m^u(x). 
\end{align*}
We deduce using \eqref{eq:delta_ws_s} that
\begin{equation}\label{eq:smooth_decomposition_lebesgue_weak}
 \int_U f \,dv_g = \int_{W^u_\mathrm{loc}(m)} \left(\int_{W^{ws}_\mathrm{loc}(y)} f(z) \delta_{W^{ws}}(z) dL^{ws}_y(z)\right) dL^u_m(y).
\end{equation}
with
\[
\delta_{W^{ws}}(z) = \frac{\delta_{W^s}(z)\delta^{W^{wu}_\mathrm{loc}(m)}_{W^0}\big(\mathrm{pr}_{W^{ws}_\mathrm{loc}(y)\to W^0_\mathrm{loc}(y)}(z)\big)}{\delta^{W^{ws}_\mathrm{loc}(y)}_{W^s}(z)}
\]
Here $\mathrm{pr}_{W^{ws}_\mathrm{loc}(y)\to W^0_\mathrm{loc}(y)}$ is the projection along the smooth subfoliation \eqref{eq:ws_subfoliation} of $W^{ws}_\mathrm{loc}(y)$.

In order to obtain an analogue to Proposition~\ref{prop:smooth_disintegration_Lebesgue} it remains to analyze the regularity of $\delta_{W^{ws}}$ along the leaves $W^{ws}_\mathrm{loc}(y)$.
Note that for fixed $y$ by the smoothness of the subfoliations \eqref{eq:ws_subfoliation} of  $W^{ws}_\mathrm{loc}(y)$ we conclude that the functions $\delta^{W^{ws}_\mathrm{loc}(y)}_{W^s}(z)$ and $\delta^{W^{wu}_\mathrm{loc}(m)}_{W^0}\big(\mathrm{pr}_{W^{ws}_\mathrm{loc}(y)\to W^0_\mathrm{loc}(y)}(z)\big)$ are smooth on $W^{ws}_\mathrm{loc}(y)$. By the H\"older continuity of the weak stable foliation their $C^k_{W^ {ws}_\mathrm{loc}(y)}$-norms vary continuously on $y\in W^u_\mathrm{loc}(m)$.
However, for $\delta_{W^s}(z)$ we  only know so far that this density is smooth along $W^s_\mathrm{loc}(y)$. Using the smoothness of the Anosov action we can improve this further: for any $z \in W^s_\mathrm{loc}(y)\subset U$ and any $a\in V\subset \mathbb A$ where $V$ is a neighbourhood of the identity such that $\tau(V)z\subset U$ we get the following equivariance property which can be derived from a straightforward calculations using the $\mathbb A$ invariance of the weak-(un)stable foliations as well as several occurrences of the transformation formula:
\[
 \delta_{W^s}(\tau(a)z) = \frac{|\det(d\tau(a))|(z) }{|\det(d\tau(a)_{|E^{wu}})|(y)\cdot|\det(d\tau(a)_{|E^s})(z)|}\delta_{W^s}(z).
\]
All the Jacobians appearing here are understood with respect to the respective Riemannian volume measures.
As the Jacobians depend smoothly on $a$ this shows that $\delta_{W^s}$ has also bounded derivatives of any order into the direction of the $\mathbb A$-orbits.
Summarizing we have shown:
\begin{prop}\label{prop:smooth_disintegration_Lebesgue_weak_stable}
Precisely the same statement as Proposition~\ref{prop:smooth_disintegration_Lebesgue} holds when replacing the (un)stable foliation $W^{s/u}$ by the weak (un)stable foliation $W^{ws/wu}$.
\end{prop}

As a consequence of Proposition~\ref{prop:smooth_disintegration_Lebesgue_weak_stable} we can prove the following crucial result which is a slightly more general version of \cite[Prop 6]{Wei17} for Anosov actions. It connects classical regularity of functions into the directions of a dynamical Hölder foliations with its microlocal regularity i.e. the wavefront set:
\begin{lemma}\label{lemma:regularityunstable}
Let $\tau$ be an $\R^\kappa$-Anosov action, and consider its weak-unstable foliation. Let $f$ be a measurable function on $\mc M$, which is uniformly smooth along the leaves of $W^{wu}$. Then $\WF(f)\subset E^\ast_u$.
\end{lemma}
\begin{proof}
We pick a point $p\in \M$, $\xi \in T^\ast_p \M$, such that $\xi \notin E^\ast_u$ and $S(x)=(x-p).\xi$ defined in local coordinates near $p$. Let $G$ be a transverse foliation to $W^{wu}$ near $p$, and we can assume for example that $G(p)=W^{s}(p)$. Then, using Proposition \ref{prop:smooth_disintegration_Lebesgue_weak_stable}, for each $\chi\in C^\infty(\M)$ supported in a small neighbourhood of $p$
\[
\Big|\int \chi e^{\frac{i}{h}S}f dv_g\Big| \leq \int_{W^{s}(p)} \left| \int_{W^{wu}_{\rm loc}(y)} \chi(x) e^{\frac{i}{h}S(x)}f(x) \delta_{W^{wu}(y)}(x)dL^{wu}_y(x)\right|dL^s_p(y).
\]
 Since each $W^{wu}_{\rm loc}(y)$ is a smooth manifold, and $f$ restricted to this manifold is smooth --- for Lebesgue a.e. $y$ --- we can integrate by parts in the variable $x$. Here, it is crucial that $\delta_{W^{wu}(y)}(x)$ is smooth in $x$. We deduce (since the estimate on $f$ are locally uniform) that this integral is $\mathcal{O}(h^\infty)$, as soon as $dS_{|W^{wu}_{loc}(p)}$ does not vanish.
But the condition that $\xi \notin E^\ast_u$ ensures this close enough to $p$, since $E^\ast_u$ is exactly the set of covectors that vanish on $E^u \oplus \R X$. So for $\chi$ supported close to $p$ one gets the desired result. This implies that $\xi\notin \WF(f)$ by the usual definition of wavefront set \cite[Section 8.1]{Hoe03}.
\end{proof}

\subsection{Invariant measures via spectral theory}\label{sec:spec_theory}

In this section we state the results about the equilibrium measure for general Anosov actions
as they have been obtained in \cite{BGHW20} and we also recall the essential constructions on
which our analysis will be based.

The existence was obtained through the theory of \emph{Ruelle-Taylor resonances}, which are defined as a joint spectrum of the family of vector fields $X_{A}$ for $A\in \a$ acting on certain functional spaces. More precisely, we say that $\la\in \a^*_\C$ is a Ruelle-Taylor resonance for $\tau$ if and only if there exists $u\in \mathcal{D}'(\mc{M})$ non-zero with ${\rm WF}(u)\subset E_u^*$ and
\begin{equation}\label{def:RTresonance}
\forall A\in \a, \quad  -X_{A}u=\lambda(A)u.
\end{equation}
The choice of using $-X_A$ and not  $X_A$ is justified by the fact that the flow $e^{-X_{A}}$ acting in ``backward'' time by pull-back on distributions (for example on a Dirac mass) moves the support forward by $e^{X_A}$ and is then more convenient.

The definition of the spectrum involves the unstable bundle, and depends only on the choice of positive Weyl chamber. In the particular case of algebraic Anosov actions, the different Weyl chambers are permuted by the so-called Weyl group and the Weyl group also acts on the spectra as is explained in  \cite{HWW21}.

Given a general Anosov action $X$, we choose vectors $A_1,\dots,A_\kappa$ in the Weyl chamber $\W$, which form a basis of $\a$. The dual basis in $\a^*$ is denoted $(e_j)_j$, and set $X_j:=X_{A_j}$, and we use $dv_g$ the smooth Riemannian probability measure on $\M$. If we further pick a non-negative function $\psi_j\in C_c^\infty(\R^+)$ satisfying $\int_{0}^\infty  \psi_j(t)dt = 1$ then we can consider the operator
\begin{equation}\label{eq:Eq-R}
R := \prod_{j=1}^\kappa \int e^{-t_j X_j} \psi_j(t_j) dt_j.
\end{equation}
This operator appeared in a parametrix construction in a Taylor complex generated by the Anosov action and this parametrix was the central ingredient for establishing the existence of the Ruelle-Taylor resonances in \cite{BGHW20}. For the purpose of this paper we will not need to introduce the Taylor complexes and spectrum but we will only focus on the results needed for our present work.
In section 4.1 of \cite{BGHW20}, we construct a function $G\in C^\infty(T^*\M)$, called \emph{escape function for any $A\in\W$}  satisfying the properties of \cite[Definition 4.1]{BGHW20}: in particular, there is $R_0>0$, $c_X>0$ and a conic neighborhood $\Gamma_{E_0^*}$ of $E_0^*$ in $T^*\M$ such that
\begin{equation}\label{escapefct}
\left\{\begin{array}{l}
G(x,\xi)=m(x,\xi)\log(1+f(x,\xi)),\\
 f\in C^\infty(T^*\M,\R^+) \textrm{ positive and homogeneous of degree } 1 \textrm{ for }|\xi|>R_0,\\
  m\in C^\infty(T^*\M,[-1/2,8]) \textrm{ homogeneous of degree } 0 \textrm{ for }|\xi|>R_0,\\
  m\leq -1/4 \textrm{ in an arbitrarily small conic neighborhood } \Gamma_u\textrm{ of } E_u^* \textrm{ in }\{|\xi|>R_0\}
  \\
  m\geq 1/2 \textrm{ outside an arbitrarily small conic neighborhood }\Gamma_u'\textrm{ of }
 \Gamma_u \textrm{ in }\{|\xi|>R_0\}.\\
(\bigcup_{t\in [0,1]} e^{tX_A^H}(x,\xi))\cap  \{|\xi|<R_0\}=\emptyset  \Rightarrow   
m(e^{X_{A}^H}(x,\xi))-m(x,\xi)\leq 0 \\
(\bigcup_{t\in [0,1]} e^{tX_A^H}(x,\xi))\cap ( \Gamma_{E_0^*}\cup \{|\xi|<R_0\})=\emptyset \Rightarrow   G(e^{X_{A}^H}(x,\xi))-G(x,\xi)\leq -c_X
\end{array}\right.
\end{equation}
where $X_A^H$ is the Hamilton vector field of the principal symbol $p:=\xi(X_A(x))$ of the operator $-iX_A$ (we note that its flow $e^{tX_A^H}$ is the symplectic lift of $\varphi_t^{A}$).

After fixing a quantization procedure
$\Op$ mapping symbols on $T^*\M$ to operators acting on $C^\infty(\M)$ (as in \cite{Zwo12}), we consider the pseudo-differential operator
$\Op(e^{NG})$ with variable order and we define the Hilbert space
\[
\mc{H}^{NG}:=\Op(e^{NG})^{-1}L^2(\M).
\]
where $\Op(e^{NG})$ can be made invertible by choosing appropriately $G$.
For later, we will also need a semi-classical parameter $h\in (0,1]$, to consider a semiclassical quantization $\Op_h$ and to define
\begin{equation}\label{HNG_h}
\mc{H}^{NG}_h:=(\Op_h(e^{NG}))^{-1}L^2(\M).
\end{equation}
The spaces $\mc{H}^{NG}_h$ for different values of $h$ are the same topological vector spaces but the norms are different.
For more details on the construction of the anisotropic spaces and the used microlocal techniques we refer to \cite[Section 4.1 and Appendix A]{BGHW20}.
\begin{prop}[{\cite[Lemma 4.14, Lemma 5.1 and Lemma 5.2]{BGHW20}}]\label{limiteR^k}
For $N>0$ large enough, the operator $R$ of \eqref{eq:Eq-R} is a bounded operator on $\mc{H}^{NG}$ with essential spectrum  contained in the disk $D(0,1/2)$. The only
eigenvalue $s$ with $|s|\geq1$ is $s=1$ and this eigenvalue has finite multiplicity and no Jordan blocks. Finally, if $\Pi$ denotes the finite rank spectral projector of $R$ at $s=1$, then the following convergence holds in $\mc{L}(\mc{H}^{NG})$
\[
\lim_{k\to +\infty}R^k=\Pi.
\]
Moreover $\Pi$ only depends on the choice of positive Weyl chamber.
\end{prop}
We note that the same results hold on the spaces $\mc{H}_h^{NG}$ for all $h\in (0,1)$. We would like to emphasize that the anisotropic Sobolev spaces and the escape functions are rather auxiliary objects. Neither the Ruelle-Taylor resonances nor the resonant states depend on their choice. The following technical Lemma (see \cite[Proof of Lemma 4.13]{BGHW20}) indicates how the flexibility in the  choice of the function $G$ can be used:
\begin{lemma}\label{HNG-WF}
Let $G$ satisfy \eqref{escapefct}. A distribution $u\in \mathcal{D}'(\M)$ having $\WF(u)\subset E_u^*$ satisfies $u\in \mc{H}^{NG}$ for some $N>0$. Conversely, if for any $G$ satisfying \eqref{escapefct} there is $N_0$ such that $u\in \mc{H}^{NG}$ for all $N\geq N_0$, then $\WF(u)\subset E_u^*$.
\end{lemma}

If $F\subset T^*\M$ is a conical closed set in $T^*\M$, we denote by $\mathcal{D}'_F(\M):=\{ u\in \mathcal{D}'(\M)\, |\, \WF(u)\subset F\}$.

The spectral projector $\Pi$ satisfies $R\Pi=\Pi=\Pi R$, and by \cite[Lemma 5.3]{BGHW20} its Schwartz kernel is independent of $N,G$ and has the form
\begin{equation}\label{Pisum}
\Pi=\sum_{j=1}^rv_j\otimes \omega_j^*
\end{equation}
with $v_j\in \mc{H}^{NG}\cap \mathcal{D}'_{E_u^*}(\M)$ and $w_j^*\in (\mc{H}^{NG})^*\cap \mathcal{D}'_{E_u^*}(\M)$;
moreover if $N>0$ is large enough we have by \cite[Lemma 5.3]{BGHW20}
\begin{equation}\label{RanPi}
\ran \Pi=\{ u\in \mathcal{D}'_{E_u^*}(\M)\,|\,\forall A\in \a,\,  X_{A}u=0 \}=\{ u\in \mc{H}^{NG}\,|\, \forall A\in \a,\, X_{A}u=0\}.
\end{equation}

The relation of $\Pi$ with the equilibrium measure is explained by \cite[Proposition 5.4]{BGHW20} as follows:
\begin{prop}\label{prop:PhysMeasures}~
\begin{enumerate}
\item For each $v\in (\mathcal{H}^{NG})^\ast$, the map\footnote{an equivalent, shorter notation to express the definition of $\mu_v$ in \eqref{eq:def_muv} is to use the adjoint $\Pi^* $ of the projector $\Pi$ and write $\mu_v = \Pi^*v$.}
\begin{equation}\label{eq:def_muv}
\mu_v : u\in C^\infty(\M)\mapsto \cjg\Pi u,v\cjd
\end{equation}
extends to $C^0(\M)$ as a signed Radon measure. It is $\mathbb{A}$-invariant, in the sense that $\mu_v(X_A u)=0$ for all $A\in \a$ and $u\in C^\infty(\M)$. Also, $\WF(\mu_v)\subset E^\ast_s$. 
\item We have the equality
\[
{\rm span}\{\mu_v \,|\, v\in\mathcal{D}'(\M),\ \WF(v)\cap E^\ast_u = \emptyset\}= \Pi^*(C^\infty(\M)). 
\]
For $N$ sufficiently large, it is a finite dimensional subspace of $(\mc{H}^{NG})^*$ whose dimension equals the space of joint resonant states. Furthermore it is precisely spanned by the $\mathbb{A}$-invariant measures $\mu$ with $\WF(\mu)\subset E_s^*$.
\item \label{it:ac} Let $v_1,v_2\in C^\infty(\mc M, \R^+)$ with $v_1\leq Cv_2$ for some $C>0$. Then $\mu_{v_1}$ is absolutely continuous with bounded density with respect to $\mu_{v_2}$. In particular any $\mu_v$ is absolutely continuous with respect to $\mu_{\bf 1}$.
\item\label{it:average} For each open proper subcone $\mc{C}\subset \W$ in the positive Weyl chamber and each dual element $e_1\in \a^*$ so that $e_1>0$ on $\mc{C}$ a slightly larger open cone, $v$ a finite Radon measure with $\WF(v)\cap E^\ast_u = \emptyset$, and $u\in C^\infty(\M)$,
\begin{equation}\label{Birkhoff_sum}
\mu_v(u)= \lim_{T\to \infty} \frac{1}{|\mc{C}_T|} \int_{A\in \mc{C}_T,}\cjg \varphi^A_{-1}u,v\cjd dA
\end{equation}
where $dA$ is the Lebesgue-Haar measure on $\a$ and $\mc{C}_T:=\{A\in \mc{C}\,|\, e_1(A) \leq T \}$. In particular, $\mu_{\bf 1}$ is the equilibrium measure.
\end{enumerate}
\end{prop}

\begin{rem}
Item (4) only holds if we start with a measure and not a general distribution, because we have to replace integration over a whole cone by integration against a smooth cutoff, and control the error.
\end{rem}

\begin{proof}
According to \eqref{Pisum}, we have
\[
\mu_v(u) = \sum_{j=1}^r \langle v, v_j\rangle \langle u, \omega_j^\ast \rangle. 
\]
To establish (1) it suffices thus to prove that $u\mapsto \langle u, \omega_j^\ast\rangle$ is a Radon measure. This and item (2), (3) are contained in \cite[Proposition 5.4]{BGHW20}.

Now, (4) was only proven for $v\in C^\infty(\M)$ in \cite{BGHW20}. According to Proposition \ref{limiteR^k}, we obtain that for $\psi$ chosen suitably, 
\begin{equation}\label{limiteRkPi}
\langle R^k u, v\rangle \to  \mu_v(u).
\end{equation}
We now have to relate the quantity in the LHS to cone averages, which we do following the same strategy as in \cite[Proposition 5.4]{BGHW20}.

There, we replaced the functions $\prod_{j=1}^\kappa \psi_j(t_j)$ by $\psi_\sigma(t)=\prod_{j=1}^\kappa \psi_j(t_j-\sigma_j)$ for $\sigma\in \R^\kappa\simeq \a$ small in the definition of $R^k$, and call $R_\sigma^k$ the resulting operator. Fixing one direction $A_1\in \mc{C}$, $e_1\in \a^*$ so that $e_1(A_1)=1$, and taking a
transverse hypersurface $\Sigma=e_1^{-1}(\{1\})$ to $\mc{C}$, we obtain coordinates $(t_1,\dots ,t_\kappa)$ with $t_1=e_1(A)$ on $\a$ and $\bar{t}=(t_2,\dots,t_\kappa)$ some linear coordinates on $\Sigma$ associated with a basis $A_2,\dots A_{\kappa}$ of $\ker e_1$.
We proved in \cite[Lemma 5.5]{BGHW20} that if $\omega\in C_c^\infty((0,1))$ satisfies $\int \omega=1$ and $q\in C_c^\infty(\Sigma\cap \mc{C})$ satisfies $\int q=1$, then for each $v'\in C^\infty(\M)$, one has 
\[
\begin{gathered}
 \Big|\frac{1}{N}\sum_{k=1}^N\int_{\R^{\kappa-1}}\omega(\tfrac{k}{N})\cjg R_{\sigma(\bar{t})}^ku,v'\cjd q(\bar{t})d\bar{t} -\frac{1}{N^\kappa}\int_{0}^N\int_{\R^{\kappa-1}}\cjg e^{-\sum_{j=1}^\kappa t_j X_{A_j}}u,v'\cjd (\tfrac{t_1}{N})^{1-\kappa}\omega(\tfrac{t_1}{N}) q(\tfrac{\bar{t}}{t_1})d\bar{t}dt_1\Big|\\
\leq \eps(N) \sup_{t_1,\bar{t}}|\cjg e^{-\sum_{j=1}^\kappa t_j X_{A_j}}u,v'\cjd|\leq \eps(N)\|u\|_{C^0(\M)}\|v'\|_{(C^0(\M))^*}
\end{gathered}
\]
where $\sigma(\bar{t})=(1,\bar{t})$ in the coordinates $t_1,\dots,t_\kappa$ and $\eps(N)\to 0$ as $N\to \infty$. By density, we can extend the inequality above to Radon measures, so it applies to $v$. We obtain, using \eqref{limiteRkPi} and letting $N\to \infty$,
\[
\cjg \Pi u, v\cjd=\lim_{N\to \infty}\frac{1}{N^\kappa}\int_{0}^N\int_{\R^{\kappa-1}}\cjg e^{-\sum_{j=1}^nt_j X_{A_j}}u,v\cjd (\tfrac{t_1}{N})^{1-\kappa}\omega(\tfrac{t_1}{N}) q(\tfrac{\bar{t}}{t_1})d\bar{t}dt_1.
\]
By finally letting
$\omega(t_1)q(\bar{t}/t_1)$ be arbitrarly close to $\kappa t_1^{\kappa-1}\textbf{1}_{[0,1]}(t_1)\textbf{1}_{\Sigma\cap \mc{C}}(\bar{t}/t_1)/|\Sigma\cap\mathcal{C}|$ in $L^1$, we obtain the desired limit.
\end{proof}
\section{Characterization of the  invariant measures}\label{sec:basin}
In this section we will characterise the invariant measures $\mu_v$ obtained by the Ruelle-Taylor spectral theory
in terms of their fine dynamical properties: we will first study the disintegration of $\mu_v$ with respect to the stable foliation. This will show that the measures $\mu_v$ are precisely the Sinai-Ruelle-Bowen measures (see Definition~\ref{def:SRB}). Next we will study the ergodic decompositions of the measures $\mu_v$ as well as their basins of attraction and relate them to the physical measure (cf Definition~\ref{def:physical_measure}).
\subsection{SRB measures}\label{sec:SRB}
Let us now discuss absolute continuity of measures along $W^s$:
\begin{Def}\label{def:SRB} Let $\mu$ be a Radon measure. If $\mu$ is $\mathbb{A}$-invariant and has absolutely continuous conditionals with respect to $W^s$, we say that $\mu$ is \emph{SRB}.\end{Def}

Just as in the rank $1$ case, we can prove the following:
\begin{theorem}\label{thm:absolutecontinuity}
The set of SRB measures is exactly the set of positive Radon measures $\mu_v$ in $\ran\Pi^\ast$. Additionally, the Lebesgue density of the $W^s$ conditionals of $\mu_v$ is uniformly smooth along the $W^s$ leaves and strictly positive $\mu_v$-almost everywhere. 
\end{theorem}
\begin{proof} Let us assume the following:
\begin{lemma}\label{lemma:absolutecontinuitymu1}
The measure $\mu_{\bf 1}$ is SRB.\footnote{A priori, a different choice of measure $v_g$ yields another $\mu_\mathbf{1}'$, but according to item 3) of Proposition \ref{prop:PhysMeasures}, we have $\mu_\mathbf{1} << \mu_\mathbf{1}' << \mu_\mathbf{1}$.}
\end{lemma}
It is not hard to see that if $\mu \leq C \nu$, then $\nu$-a.e- $\mu_y \leq C \nu_y$. In particular, we deduce that actually every positive Radon measure in $\ran \Pi^\ast$ has absolutely continuous disintegration which shows that the positive Radon measures in $\ran\Pi^\ast$ are SRB measures.  In \cite{Ledrappier-Young85}, we can find the following result (which follows from Theorem A and Corollaries 6.14 and 6.2 therein):
\begin{lemma}\label{lemma:Ledrappier-Young}
If $\mu$ is invariant by $\varphi^A_1$ for some $A\in\W$, and has absolutely continuous disintegration with respect to $W^s$, then for $\mu$ a.e. $y$, its density $\rho_y$ is non vanishing on $W^s_{{\rm loc}}(y)$ and satisfies
\[
\frac{\rho_y(x)}{\rho_y(z)} = \frac{\prod_{j=1}^\infty |\det(d_{\varphi_j^{A}(x)}\varphi^{A}_{-1}|_{E_s})|}{\prod_{j=1}^\infty |\det(d_{\varphi_j^{A}(z)}\varphi^{A}_{-1}|_{E_s})|}.
\]
\end{lemma}
The formula on the right of the equation defines a smooth function, as one can prove following the same strategy as the one employed in the proof of Proposition \ref{prop:smooth_disintegration_Lebesgue}. This proves the announced statement regarding the densities of the $\mu_v$'s. Finally, consider $\mu$ an SRB measure. According to Lemma \ref{lemma:Ledrappier-Young}, its densities must be uniformly smooth. Using the same arguments as in the proof of Lemma \ref{lemma:regularityunstable}, we deduce that $\WF(\mu)\subset E^\ast_s$, and thus $\mu\in \ran\Pi^\ast$.\end{proof} 

It remains to prove Lemma \ref{lemma:absolutecontinuitymu1}. In general, if $\mu_n \rightharpoonup \mu$, and the $\mu_n$ have conditionals $\mu^n_y$, it is not guaranteed that $\mu^n_y \rightharpoonup \mu_y$ a.e. $y$. This is behind the subtlety of this proof.

\begin{proof}[Proof of Lemma \ref{lemma:absolutecontinuitymu1}] We follow the proof of \cite[Theorem 6.3.1]{LSYoung95}. Call $\nu_0:=v_g$ and $\nu_A=(\varphi^{A}_{-1})_*\nu_0$ for $A\in \W$. We also define $\nu^T:=\frac{1}{|\mc{C}_T|}\int_{\mc{C}_T}\nu_A dA$ and $\nu^\infty:=\mu_{{\bf 1}}$ so that $\nu_T \rightharpoonup \nu^\infty$. 

Let $x_0\in \M$ and consider a small neighborhood $U_{x_0}$  and fix a smooth foliation $G$ transversal to $W^s_{\rm loc}$ such that the disintegration along stable leaves can be done by Rokhlin's theorem (Proposition \ref{prop:roklin}). We also assume that $\nu^\infty(\partial U_{x_0})= 0$. For $\nu^\infty$ a.e. $y\in U_{x_0}$ we denote $\nu^\infty_y$ the conditional on $W^s_{\rm loc}(y)$. Let $F_\eps(x)= \cup_{z\in B_G(x,\eps)}W^{s}_{\rm loc}(z)$: this is a tube of radius $\eps>0$ around $W^{s}_{\rm loc}(x_0)$ (here $B_G(x,\eps)\subset G(x)$ is the ball of radius $\eps$ in $G(x)$). Then, for $\nu^\infty$ a.e $x$,
\[
\nu^\infty_{x}=\lim_{\eps\to 0}\frac{\textbf{1}_{F_\eps(x)}\nu^\infty}{\nu^\infty(F_\eps(x))}.
\]
Fixing $x\in U_{x_0}$ such that the convergence holds, we can take the limit excluding those (potentially countably many) $\eps$'s for which $\nu^\infty(\partial F_\eps) \neq 0$. This ensures that $\nu^T(F_\epsilon)\to \nu^\infty(F_\epsilon)$ according to the Portmanteau theorem.

Certainly, the measure $\nu^T$ has a simple expression, since
\[
j_T(z):=\frac{d\nu^T}{d\nu_0}(z) = \frac{1}{|\mc{C}_T|} \int_{\mc{C}_T} \det(d\varphi^A_1(z)) dA. 
\]
Now, we observe that for $z\in W^s_{\rm loc}(y)$, with constants independent of $A\in\mathcal{C}$, 
\begin{align*}
\left|\log\frac{\det(d\varphi^A_1(z))}{\det(d\varphi^A_1(y))}\right| &= \left|\int_0^1 \big[\mathrm{div}(X_A)(e^{tX_A}(z)) -  \mathrm{div}(X_A)(e^{tX_A}(y))\big] dt \right| \\
				&\leq C \int_0^1 d_g( e^{tX_A}(z),  e^{tX_A}(y)) dt \leq C d_g(z,y). 
\end{align*}
This implies
\[
\frac{j_T(z)}{j_T(y)} \leq e^{C d_g(z,y)}. 
\]
Now let $V\subset W^{s}_{\rm loc}(x)$ be a small open ball and consider $\mc{V}_\eps:=\cup_{y\in V}G(y)\cap F_\eps(x)$. Taking a slightly large or smaller $V$, and again restricting the $\eps$'s, we can ensure that $\nu^T(\mc{V}_\eps)\to \nu^\infty(\mc{V}_\eps)$. We use the smooth disintegration of Lebesgue measure (Proposition \ref{prop:smooth_disintegration_Lebesgue}) to write
\begin{align*}
\nu^T(\mc{V}_\epsilon) &= \int_{B_G(x,\eps)} \int_{H^G_{x,y}(V)} j_T(z)\delta_{W^s} dL^s_y(z) dL^G_x(y)\\
						&\leq C \int_{B_G(x,\eps)} \int_{H^G_{x,y}(V)} j_T(y)\delta_{W^s} dL^s_y(z) dL^G_x(y)\\
						&\leq C \int_{B_G(x,\eps)} j_T(y)dL^G_x(y) \int_{V}\delta_{W^s} dL^s_x(z).
\end{align*}
In the last line, we used that $\delta_{W^s}$ are bounded from above and away from zero (because they are continuous an non vanishing, see Proposition \ref{prop:smooth_disintegration_Lebesgue}). On the other hand,
\begin{align*}
\nu^T(F_\epsilon) &= \int_{B_G(x,\eps)} \int_{W^s_{\rm loc}(y)} j_T(z)\delta_{W^s} dL^s_y(z) dL^G_x(y)\\
			&\geq \frac{1}{C} \int_{B_G(x,\eps)} \int_{W^s_{\rm loc}(y)} j_T(y)\delta_{W^s} dL^s_y(z) dL^G_x(y)\\
			&\geq \frac{1}{C} \int_{B_G(x,\eps)} j_T(y)dL^G_x(y) \int_{W^s_{\rm loc}(x)}\delta_{W^s} dL^s_x(z).
\end{align*}
We deduce that for some constant $C>0$ independent of $V$ and $T$, 
\[
\frac{\nu^T(\mc{V}_\epsilon)}{\nu^T(F_\epsilon)} \leq C|V|,
\]
where $|.|$ is the \emph{normalized} Lebesgue measure on $W^s_{\rm loc}(x)$. Letting $T\to+\infty$ and then $\eps\to 0$ (with the conditions mentioned above on $\eps$) we get the same bound for $\nu^\infty$, which proves that $\nu^\infty=\mu_{\bf 1}$ has absolutely continuous density. 
\end{proof}

\subsection{Physical measures}\label{sec:physical}

Next we want to study the ergodic decomposition and the basins of the invariant measures $\mu_v$ that we obtained by the Ruelle-Taylor spectral theory.

Let us first introduce some notation:
\begin{Def}\label{def:physical_measure}
For an invariant measure $\mu$ define the \emph{basin} (of attraction) to be those points $x\in \mc M$ such that for any $f\in C^0(\mc M)$ and any proper open subcone $\mc C \subset \mc W$ we have
\begin{equation}\label{eq:Cone-averaging-ergodic-SRB-measure}
\mu(f) =\lim_{T\to \infty} \frac{1}{|\mc{C}_T|} \int_{A\in \mc{C}_T}f(\varphi^A_{-1}(x)) dA.
\end{equation}
We say that an invariant measure $\mu$ is a \emph{physical measure} for $\tau$ if the basin has positive Lebesgue measure.
\end{Def}

We show that the $\mu_v$ are linear combinations of physical measures and give various other different characterisations:
\begin{theorem}\label{thm:physical}
Let $\tau$ be a smooth, locally free, $\R^\kappa$ Anosov action with generating map $X$, then we have:
\begin{enumerate}[1)]
\item The linear span over $\C$ of the physical measures is $\ran \Pi^\ast$.
\item The equilibrium measure \eqref{eq:Cone-averaging-physical-measure} is a positive linear combination of physical measures.
\item The union of the basins of the physical measures has full Lebesgue measure in $\M$.
\item An ergodic Radon probability  measure $\mu$ is a physical measure if and only if it is invariant by $\varphi_t^{A}$ for all $A\in \a$ and has wavefront set
$\WF(\mu)\subset E_s^*$.
\end{enumerate}
\end{theorem}

Observe that the physical measures are really intrinsic to the action and do not depend on the choice of auxiliary metric. The equilibrium measure is some combination of these physical measure determined by a choice of volume measure, it does not have an intrinsic dynamical meaning.

We furthermore prove that if the Anosov action is transitive then the equilibrium measure is the unique SRB and physical measure (see Corollary~\ref{cor:unique}) and if it is positively transitive the equilibrium measure has full support (see Proposition~\ref{prop:full_support}). Theorem~\ref{thm:physical} together with these two additional result then gives Theorem~\ref{thm:SRB_intro} stated in the introduction. Indeed the implication $(1)\Longrightarrow (2)$ from Theorem \ref{thm:SRB_intro} is exactly $3)$ of Theorem \ref{thm:physical}, the implication $(2)\Longrightarrow (1)$ of Theorem \ref{thm:SRB_intro}  follows from dominated convergence, the equivalence $(4)\iff (2)$ in Theorem \ref{thm:SRB_intro} follows from $4)$ of Theorem  \ref{thm:physical}, the equivalence $(2)\iff (3)$ in Theorem \ref{thm:SRB_intro} is a consequence of $1)$ of Theorem \ref{thm:physical} (with the uniqueness of physical measure by transitivity in Proposition~\ref{prop:full_support}),  and of Theorem \ref{thm:absolutecontinuity}.

We will first study the ergodic decomposition of $\mu_{\bf 1}$ and identify the basins of attractions. We will use some arguments in the spirit of \cite[Proposition 2.3.2]{BKL02} to obtain the:
\begin{lemma}\label{lem:basinsFj}
Let $X$ be an Anosov action, and let $\mu=\mu_{\bf 1}$ be the equilibrium measure. There exist disjoint measurable sets $F_1,\ldots,F_r$ such that $\mu(F_i)=v_g(F_i)\neq 0$ for all $i$ and $\mu(\cup_iF_i)=v_g(\cup_iF_i)=1$. Furthermore the ergodic components of $\mu$ are given by $\mu^i:=\frac{\mathbf{1}_{F_i}}{\mu(F_i)}\cdot \mu$ and $F_i$ is the basin of $\mu^i$. In particular each $\mu^i$ is a physical measure. Finally the $\mu^i$'s form a basis of the space $\Pi^*(C^\infty(\M))$.
\end{lemma}
This Lemma actually implies Theorem \ref{thm:physical}. Indeed, it gives directly item 2) and 3) and also implies $\ran \Pi^\ast \subset \mathrm{span(physical)}$. Next we observe that the basins of different physical measures must be pairwise disjoint. Since the union of the $F_j$'s has full Lebesgue measure, there cannot be an additional physical measure. This proves item 1). Finally item 4) follows from 1). Indeed, let $\mu$ be an ergodic measure. If it satisfies the wavefront set condition, it must be in $\ran \Pi^\ast$, so it is absolutely continuous with respect to $\mu_{\bf 1}$. Since it is ergodic, it must be an ergodic component of $\mu_{\bf 1}$ and Lemma \ref{lem:basinsFj} applies. On the other hand, if it is physical, item 1) applies.
\begin{proof}
For $\mc{C}\subset \W$ a proper open subcone, $e_1\in \a^*$ with $e_1>0$ on a slightly larger cone $\mc{C}$,
and $f\in C^0(\M)$, we define
\[
\Omega(f,\mc{C}):=\left\{ x\in \M\ \middle|\ f_-(x):=\lim_{T\to +\infty}\frac{1}{|\mc{C}_T|}\int_{A\in \mc{C}_T} f(\varphi^A_{-1}(x))dA \text{ exists} \ \right\},
\]
where $\mc{C}_T=\mc{C}\cap \{e_1(A)\in (0,T)\}$.
It follows from the ergodic Birkhoff Theorem for actions (see \cite[Theorem 3]{Bewley}) that for all such $\mc{C}$ and $f$, and every invariant Borel measure $\nu$, provided $\partial\mc{C}$ has zero Lebesgue measure, $\Omega(f,\mc{C})$ has full $\nu$-measure. We observe that if $\mc{C}_{1}, \mc{C}_{2}$ are disjoint open subcones, $\Omega(f,\mc{C}_1)\cap \Omega(f, \mc{C}_2) \subset \Omega(f, \mc{C}_1\cup \mc{C}_2)$. By dichotomy, we can find a sequence $(\mc{C}_j)_{j\geq 0}$ of proper open subcones whose boundary has zero Lebesgue measure satisfying the following. For every proper open subcone $\mc{C}$, there exists $I\subset \N$ such that $\cup_I \mc{C}_j\subset \mc{C}$ is a disjoint union, and $\mc{C}\setminus(\cup_I \mc{C}_j)$ has Lebesgue measure zero. Then we obtain that $\cap_I \Omega(f,\mc{C}_j) \subset \Omega(f,\mc{C})$. In particular, this gives
\[
\bigcap_{\mc{C}} \Omega(f,\mc{C}) = \bigcap_j \Omega(f,\mc{C}_j),
\]
has full $\nu$-measure. Using the fact that $C^0(\M)$ is separable, we can further improve this, by saying that
\[
\Omega := \bigcap_{f,\mc{C}} \Omega(f,\mc{C}),
\]
also has full $\nu$ measure for every invariant Borel measure $\nu$. Additionally, we observe that if $x\in\Omega$, then the weak unstable manifold satisfies $W^{wu}(x)\subset \Omega$.

More generally, if a measurable set $F$ is a union of full weak unstable manifolds, we will say that $F$ is \emph{unstable-invariant}. Notice that $\Omega$ is unstable-invariant. According to Lemma \ref{lemma:regularityunstable}, this implies  $\WF(\mathbf{1}_{F}) \subset E^\ast_u $, and thus $\mathbf{1}_{F} \in \mathcal{H}^{NG}$ for $N$ large enough. In particular, since $X_A \textbf{1}_{F} = 0$ for all $A\in\a$, $\textbf{1}_{F}$ belongs to the finite dimensional space $\ran \Pi$ by \eqref{RanPi}. This means that $\textbf{1}_{F}=\Pi\textbf{1}_{F}$ in the distribution sense, thus \emph{Lebesgue} almost everywhere since $\textbf{1}_{F}$ is $L^1$.

Since for $v\in C^\infty$, the identity $\mu_v(u)=\langle\Pi u,v\rangle$ extends from smooth functions $u$ to elements $u\in \mc{H}^{NG}$, we have for each unstable-invariant set $F$,
\begin{equation}\label{eq:mu_lebesgue}
\mu_v(F)= \int_{\M}(\Pi\textbf{1}_{F})v \, dv_g= \int_{F}v \, dv_g.
\end{equation}
For $v={\bf 1}$, this gives $\mu(F) = v_g(F)$. In particular, since $\mu(\Omega)=1$, $\Omega$ has full Lebesgue measure (even if the Lebesgue measure is not invariant by the action). Since each such $\mathbf{1}_F$ is an element of the space of  resonant states at $\lambda=0$
\[
H:=\mathrm{Span}\left\{ \mathbf{1}_F\ \middle|\ F\text{ is unstable invariant } \right\}/\sim \text{ is a subspace of } \ran \Pi.
\]
(here, $\sim$ is the equivalence relation of being equal Lebesgue a.e. or, equivalently $\mu$ a.e.). Accordingly, we can find pairwise disjoint unstable invariant sets $\tilde{F}_1,\dots, \tilde{F}_r$ with $r\leq \rank \Pi$, so that the $[\mathbf{1}_{\tilde{F}_j}]$ form a basis of $H$. Since $\Omega$ has full-measure, we can assume that $\cup \tilde{F}_j = \Omega$.

If $f$ is a continuous function and $\mc{C}\subset \W$ is an open proper subcone, we observe that the sets $\{ x\in \Omega\ |\ f_-(x) \leq r\}$ are unstable-invariant for each $r\in \R$, so they are finite unions of $\tilde{F}_j$'s, up to Lebesgue-null sets and by \eqref{eq:mu_lebesgue} also up to $\mu$-null sets. This implies that $f_-(x)$ is constant on each $\tilde{F}_j$, $\mu$-a.e. and we denote $f_{-,j}$ that value. By Lebesgue theorem and the invariance of $\mu$ by $\varphi^{A}_{-1}$,
\begin{equation}\label{eq:muFj}
\mu(\tilde{F}_j)f_{-,j}=\int_{\tilde{F}_j}f_-(x)d\mu(x)= \lim_{T\to \infty}\int_{\tilde{F}_j}\frac{1}{|\mc{C}_T|}\int_{\mc{C}_T} f(\varphi^A_{-1}(x))dA\, d\mu(x)=\int_{\tilde{F}_j}f\,d\mu.
\end{equation}
Thus, if we define $\mu^j := \textbf{1}_{\tilde{F}_j} \mu/\mu(\tilde{F}_j)$ we get $f_{-,j}=\mu^j(f)$. Thus for arbitrary $f\in C^0(\M)$ and an arbitrary proper subcone $\mc C\subset \W$ we have shown that for $\mu$ a.e. $x\in \tilde{F}_j$
we have $f_-(x)=\mu^j(f)$. Using as above that $C^0(\M)$ is separable and that we can approximate an arbitrary cone by a countable number of cones, we deduce that the basin $F_j$ of $\mu^j$ differs from $\tilde{F}_j$ by a $\mu$ nullset or equivalently a Lebesgue nullset. We have thus seen that the $\mathbf 1_{F_j}\sim \mathbf{1}_{\tilde{F}_j}$ form a basis of $H$ with dynamical significance. 

Now we prove that in fact $H=\ran\Pi$. The same argument as in \eqref{eq:muFj} can be done for $\mu_v$, so we deduce that for $v\in C^\infty$,
\begin{equation}\label{eq:ergodic-decomposition-mu}
\int f d\mu_v = \sum_j f_-(F_j) \mu_v(F_j) = \sum_j \mu^j(f) \int_{F_j} v dv_g.
\end{equation}
Let $\pi:L^1(\M,\mu)\to L^1(\M,\mu)$ be the projector onto the set of $X_A$ invariant functions (for all $A\in \a$) along the closed subspace generated by coboundaries $\{\varphi^{A}_1f - f\ |\ f\in L^1(\mu),\ A\in \a\}$.
By the ergodic theorem  of \cite{Bewley}, $\pi$ is a continuous operator and for $\mu$ almost all $x\in \M$
\[ \pi f(x)=\lim_{T\to \infty}\frac{1}{|\mc{C}_T|}\int_{\mc{C}_T} f(\varphi^A_{-1}(x))dA.\]
We have just proved that $\pi$ maps $C^0(\M)$ to functions constant on the $F_j$'s. In particular, by density of continuous functions in $L^1(\mu)$, we deduce that the image of $\pi$ only contains functions constant on the $F_j$'s. This proves that the $F_j$, or more precisely, the $\mu^j$ are the ergodic components of $\mu$, and that Equation \eqref{eq:ergodic-decomposition-mu} is the ergodic decomposition (in the sense of Theorem 4.2.6 of \cite{HaKaHandbook1}).
One consequence of the above is that $\mu$ has at most $\rank \Pi$ ergodic components. However, according to Proposition \ref{prop:PhysMeasures} in $\{\mu_v\, |\, v\in C^\infty(\M)\}$, we can find $\rank \Pi$ linearly independent probability measures, absolutely continuous with respect to $\mu$, and invariant under the action. This implies that the number of ergodic components is at least $\rank \Pi$. We deduce that $H = \ran \Pi$, and that the $\mathbf{1}_{F_j}$ form a basis of $\ran \Pi$.

It remains to show that the $\mu^j$'s form a basis of $\Pi^*(C^\infty)$. Since they have pairwise disjoint basin, the
$\mu^j$ are linearly independent and they span a space of the same dimension as $\Pi^*(C^\infty)$. It thus remains to prove that all $\mu^j$ lie in $\ran\Pi^*$. We can refine \eqref{eq:ergodic-decomposition-mu}, since the LHS of \eqref{eq:ergodic-decomposition-mu} is equal to $\langle \Pi f, v\rangle$. Since the $\mathbf{1}_{F_j}$ form a basis of $\ran\Pi$, we can find some $\omega_j^\ast \in \ran \Pi^\ast$ so that
\begin{equation}\label{Pi0explicit}
\Pi = \sum_{j=1}^r \mathbf{1}_{F_j} \otimes \omega_j^*,
\end{equation}
Now, we can write $\mu_v$ in two different ways:
\[
\mu_v(u) = \sum_j w_j^\ast(u) \int_{F_j} v dv_g =  \sum_j \mu^j(u) \int_{F_j} v dv_g,
\]
so that $w_j^\ast = \mu^j$. 
\end{proof}
Observe that we had originally no information on the wavefront set of $\mu^j$, because it is defined as the product of $\mu_{\bf 1}$ with $\WF(\mu_{\bf 1})\subset E^\ast_s$ and $\mathbf{1}_{F_j}$ with $\WF(\mathbf{1}_{F_j})\subset E^\ast_u$. In order to prove the uniqueness of physical measures for transitive actions, we will prove the following statement, reminiscent of Smale's spectral decomposition for Axiom A flows:
\begin{lemma}\label{lem:open_basin}
Let $\mu^j$ be a physical measure. The support of $\mu^j$ is a union of full weak stable leaves. The smallest $\mathbb{A}$-invariant open neighbourhood $U_j$ of the support of $\mu^j$ is given by 
\begin{equation}\label{eq:def-U_j}
U_j :=\bigcup_{x\in\, \supp(\mu^j)} W^{u}(x).
\end{equation}
The basin $F_j$ of $\mu^j$ is contained and has full Lebesgue measure in $U_j$. Finally, $\supp(\mu^j)$ admits an arbitrarily small open neighbourhood that is $-\W$ stable. Finally, we obtain the decomposition
\[
\Pi= \sum_j \mathbf{1}_{U_j} \otimes \mu^j.
\]
\end{lemma}
We see that $\supp(\mu^j)$ is an ``attractor'' for $-\W$. Let us discuss some consequences before going to the proof.
\begin{cor}\label{cor:unique}
It the Anosov action $\tau$ is transitive then there is a unique physical measure associated with $\W$.
\end{cor}
\begin{proof}
Assume that there are two physical measures $\mu^1$ and $\mu^2$ and denote by $U_1, U_2$ the open sets provided by Lemma \ref{lem:open_basin}. By transitivity there is $A\in \a$ such that $U_1\cap \varphi^A_1(U_2)\neq \emptyset$. Then we deduce that Lebesgue a.e.
 $x\in U_1\cap\varphi^A_1(U_2)$ lies in $F_1$ and (by absolute continuity of $\varphi^A_1$ w.r.t. Lebesgue) Lebesgue a.e.
 $x\in U_1\cap\varphi^A_1(U_2)$ lies in $\varphi_1^A(F_2)$. But as the basins are flow invariant this is not possible except if $F_1=F_2$.
\end{proof}

\begin{Def}\label{def:pos_trans}
 We call an Anosov action \emph{positively transitive} with respect to $\W$ if there is a proper subcone
 $\mc C\subset \mc W$ such that for any two open sets $U,V\in\mc M$ there is
 $A\in \mc C$ such that $\varphi_1^A(U)\cap V\neq \emptyset$.
\end{Def}
If the Anosov action is positively transitive then it is obviously transitive and
we know that there is a unique SRB measure.
\begin{prop}\label{prop:full_support}
 If the Anosov action is positively transitive then the SRB measure $\mu$ has full support.
\end{prop}
\begin{proof}
If the action is positively transitive with respect to $\W$, it must also be with respect to $-\W$. In particular, if we take any open neighbourhood of $\supp(\mu)$, if it is $-\W$-invariant, it must be dense. According to Lemma \ref{lem:open_basin}, $\supp(\mu)$ has an arbitrarily small open neighbourhood, so $\supp(\mu)$ itself must be dense; since it is closed, $\supp(\mu)=\M$.
\end{proof}

\begin{proof}[Proof of Lemma~\ref{lem:open_basin}]
According to Theorem \ref{thm:physical}, the $\mu^j$'s are SRB, so we can apply Theorem \ref{thm:absolutecontinuity}. From this we deduce two facts. First, we observe that the support of $\mu^j$ must be stable-invariant. Indeed, consider $x_0$ in the support of $\mu^j$, a local transverse foliation $G$, with $G_{\rm loc}(x_0) = W^{wu}_{\rm loc}(x_0)$ and a product neighbourhood $V$. Consider $y\in W^s_{\rm loc}(x_0)$ and assume that $y\notin\mathrm{supp}(\mu^j)$. Then, we can find $y\in V' \subset V$, of the form $V' = \cup_{z\in B_{W^s(x_0)}(y,\eps)} B_G(z,\eps)$, so that $\mu^j(V')=0$. Then, since the densities of the conditionnals of $\mu^j$ with respect to $W^s$ are uniformly bounded above \emph{and} below, we find that
\begin{align*}
0=\int_{V'} d\mu^j &= \int_{W^{wu}_{\rm loc}(x_0)} \left(\int_{W^s_{\rm loc}(x)\cap V'} d\mu^j_x(z)\right) d\hat\mu^j(x)\\
				&\geq \frac{1}{C}\int_{B_{W^{wu}}(x_0,\eps/2)} \left(\int_{W^s_{\rm loc}(x)} d\mu^j_x(z)\right) d\hat\mu^j(x) \\
				&\geq \frac{1}{C}\mu^j(\cup_{x\in B_{W^{wu}}(x_0,\eps/2)} W^s_{\rm loc}(x)) \geq 0.
\end{align*}
This is in contradiction with the fact that $x_0$ is in the support of $\mu^j$. We have thus shown that $\supp \mu^j$ is stable invariant. As it is trivially invariant in the flow direction we deduce that the support of $\mu^j$ is a union of weak stable leaves as claimed in the Lemma.

The second consequence of $\mu^j$ being SRB is the fact that the product neighbourhoods $V$ of a point $x_0\in \supp \mu^j$ from above fullfill $v_g(F_j\cap V) = v_g(V)$ which can be seen as follows: we start from the observation $\mathbf{1}_{F_j} \mu^j = \mu^j$, so that with the notations above, for $\hat\mu$ a.e. $y$, $\mathbf{1}_{F_j\cap W^{s}_{\rm loc}(y)}\mu^j_y= \mu^ j_y$. This means that for $\hat\mu^j$ a.e. $y$, $F_j(y):=W^s_{\rm loc}(y)\cap F_j$ has full Lebesgue measure in $W^s_{\rm loc}(y)$. As the $F_j$ are invariant in the weak-unstable directions we can use the holonomy along $W^{wu}_{\rm loc}$ from one $W^s_{\rm loc}$ leaf to another, which is absolutely continuous (Proposition \ref{holonomymap}), and deduce that \emph{every} $F_j(y)$ has full Lebesgue measure. Using Proposition \ref{prop:smooth_disintegration_Lebesgue} we deduce that $F_j\cap V$ has full Lebesgue measure in $V$.

By definition, $U_j$ is $\mathbb{A}$-invariant, and $W^u$-saturated. Let $U_j^0$ be a sufficiently small neighbourhood of the support of $\mu^j$ that can be covered by product neighbourhoods as above. Then $F\cap U_j^0$ has full Lebesgue measure in $U_j^0$. We make two observations concerning $U_j^0$: thanks to the local product structure, $U_j^0 \subset U_j$ and additionally, for any point $x$ in $U_j$, there exists $A\in\W$ such that $\varphi^A_{-1}(x)$ is in $U_j^0$. Both observations together with the $\mathbb A$-invariance of $U_j$ imply that
\begin{equation}\label{eq:attracting-neighbourhood}
U_j = \bigcup_{A\in\W} \varphi^A_1(U_j^0), 
\end{equation}
so that $U_j$ is open. We observe that for $x\in F_j$ we must also have $\varphi^A_{-1}(x)\in U_j^0$ for some $A\in\W$.
Thus by \eqref{eq:attracting-neighbourhood}, we obtain $F_j\subset U_j$.
Since $F_j$ is invariant, and since $F_j\cap U_0^j$ has full Lebesgue measure in $U_j^0$ it must also have full measure in $U_j$, according to
\eqref{eq:attracting-neighbourhood}.
\end{proof}

We close this section with a last result on the support of the physical measures.
\begin{lemma}\label{lem:ac_measures_have_basin2}
Let $\nu$ be an $\mathbb{A}$-invariant SRB Radon probability measure on $\M$. Then $\supp(\nu)\subset \M$ is connected if and only if $\nu$ is ergodic. Then $\nu$ is physical.
\end{lemma}
\begin{proof}
According to our results so far, $\nu$ being SRB must thus be a linear combination of physical measures $\mu^j$. According to Lemma \ref{lem:open_basin}, $\supp(\mu^j)\subset U_j$ and $U_j\cap U_k =\emptyset$ if $j\neq k$, so that $\supp(\nu)$ can only be connected if $\nu$ is physical. On the other hand, we know that physical measures are ergodic. It now suffices to prove that the support of each physical measure is connected. For this we observe that if $C$ is a connected component of $\supp(\mu^j)$, then $\mathbf{1}_C$ is $\mathbb{A}$-invariant, and non zero in $L^1(\mu^j)$, so that it must be of full $\mu^j$ measure. However, each connected component of $\supp(\mu^j)$ must have positive $\mu^j$ measure, so we are done.
\end{proof}

\section{A Bowen type formula for the SRB measure and Guillemin trace formula}\label{sec:bowen}

In this section we show that the SRB measure $\mu$ can be expressed in terms of the periodic orbits of the flow. We obtain thus a generalized Bowen formula. Before we can state our result, we have to recall some basic facts regarding the structure of periodic orbits of Anosov actions. We recall the classical Lemma
\begin{lemma}\label{lem:torus}
Let $\tau$ be an Anosov action with positive Weyl chamber $\W$. Let $x\in \M$, and $A\in \W$, such that $\varphi^{A}_1(x)= x$. Then there exists a lattice $L\subset \mathbb{A}$, such that for all $A' \in L$, $\varphi^{A'}_1(x) = x$, and $T: = \tau(\mathbb{A})x\simeq \mathbb{A}/L$ is an embedded torus in $\M$. We denote $L=L(T)$.
\end{lemma}

\begin{proof}
Note that the set $L=\{A'\in \a\, |\, \varphi_1^{A'}(x)=x\}\subset \a\cong \R^n$ is a discrete abelian subgroup and the Anosov action provides an injective immersion of $\R^\kappa/L$ on the orbit through $x$. Let us denote this orbit by $Y:=\{\varphi_1^{A'}(x)\,|\, A'\in \a\}$. It is enough to show that $Y$ is a closed submanifold (then it has to be compact because $\mc{M}$ is compact and hence the lattice $L$ necessarily has full rank). Let $y=\varphi_1^{A'}(x)\in Y$, then one has $\varphi_1^A(y)=\varphi_1^{A'}(\varphi_{1}^A(x))=y$, i.e. $Y$ is comprised only of fixed points of $\varphi^A_1$, and thus so is $\overline{Y}$ by continuity. However, since $A$ is transversely hyperbolic, we deduce by the implicit function theorem that for each fixed point $y$ of $\varphi^A_1$ there is an open set $U\owns y$ such that for $y' \in U$, if $\varphi^A_1(y')=y'$, then $y'$ is in the local orbit of $y$ under the action. This proves that $\overline{Y}=Y$.
\end{proof}

We stress that there may be periodic orbits $\{\varphi_t^{A_0}(x) \,  |\,  t\in [0,1]\}$ of the flow in direction $A_0\in \a$ which are not contained in an invariant torus. However this can only happen if the orbit is periodic with respect to a direction $A_0$ which is \emph{not} transversely hyperbolic and thus in no positive Weyl chamber. In any case, we denote by $\mathcal{T}$ the set of invariant tori which are precisely the compact orbits of the Anosov $\mathbb A$ action. According to the closing lemma (see \cite[Theorem 2.4]{Katok-Spatzier96}), the lattice points of the periodic tori are locally discrete in the sense that for each compact set $K\subset \W$,
\begin{equation}\label{finitenessperiodictori}
 K \bigcap \Big(\bigcup_{T\in \mc{T}} L(T) \Big)\text{ is finite.}
\end{equation}
Recall from the introduction that after fixing a Lebesgue measure on $\mathbb A$ we can pushforward the measure on $\mathbb A/L(T)$ to $T$ and obtain a natural measure $\lambda_T$ on $T$. It thus makes sense to integrate over periodic  orbits.

Given $T\in \mathcal{T}$, $x\in T$, and $A\in L(T)\cap \W$, the map $\varphi^A_1$ is hyperbolic transversal to $T$ by definition of being an Anosov action. In particular, if we set $\mathcal{P}_A(x) := d_x(\varphi^A_{-1})_{|E_u(x) \oplus E_s(x)}$, we find that
\[
|\det (1-\mathcal{P}_A(x))| \neq 0
\]
and that it does not depend on $x\in T$. We denote then by $\mathcal{P}_A$ some choice of $\mathcal{P}_A(x)$.

Let us recall that in order to prove Theorem~\ref{thm:formuledebowenintro} we need to fix $A_1\in \W$ in an open proper subcone $\mc{C}\subset\W$ and a one form $\eta\in \a^*$ positive in a small neighbourhood of $\mc C$ and we defined for $0\leq a < b$, $\mc{C}_{a,b}:=\{A\in \mc{C}\,|\, \eta(A)\in [a,b]\}$ and denoted by $|\mc C_{a,b}|$ its volume. We then want to show for $f\in C^0(\mc M)$
\[
\mu(f)=\lim_{N\to \infty}\frac{1}{|\mc{C}_{aN,bN}|} \sum_{T\in \mc{T}}
\sum_{A\in \mc{C}_{aN,bN}\cap L(T)}\frac{\int_{T}f\,d\lambda_T}{|\det(1-\mc{P}_{A})|}.
\]

In the rank 1 situation, one way to prove such a formula is to consider the flat trace of the resolvent $(X-s)^{-1}$, relating it with the periodic orbits on the one hand via the Guillemin trace formula, and with the SRB measure on the other hand, using the spectral theory of $X$ on some anisotropic space.

In our case, the proof will be heuristically similar. However, it will be complicated by the fact that we do not have a resolvent at our disposal in the multiflow situation. Thankfully, we can work around this. In the paper \cite{BGHW20}, we introduced some averaged propagators (see \eqref{eq:Eq-R} for the case $\lambda=0$)
\[
R(\lambda) := \prod_{j=1}^\kappa \int_{\R^\kappa} e^{- t_j(X_j+\lambda_j)} \psi_j(t_j) dt.
\]
where $X_j=X_{A_j}$ for some basis $(A_j)_j\in \mc{W}$ of $\a$.
By definition, $R(\lambda)$ commutes with the action. We proved that given $N>0$, for $G$ well chosen, $R(\lambda)$ is quasi-compact on $\mathcal{H}^{NG}$ for all $\lambda$'s with $\Re \lambda_j > - N$, $j=1,\dots,\kappa$. Then we proved that given $\lambda_0$, if $\lambda$ is close enough to $\lambda_0$, $\lambda$ is a Ruelle-Taylor resonance of $X$ if and only if $\lambda$ is in the joint spectrum of the family $(X_1,\dots,X_\kappa)$ acting on $\ker (R(\lambda_0)- 1)$ (see \cite[Proposition 4.17]{BGHW20}). For this reason, the study of the Ruelle-Taylor resonances (and in particular $0$, the leading resonance) can be done using the averaged propagators $R(\lambda)$. More generally, we will take functions $\psi \in C^\infty_c(\W)$, with $\int \psi = 1$, and consider
\[
R_\psi(\lambda) := \int_{\W} e^{-X_A - \lambda(A)} \psi(A) dA.
\]
What will replace the propagator in our arguments will thus be the so-called \emph{shifted resolvent} $T^\lambda_{\psi, f}(s)$, defined for $f\in C^\infty(\mc M)$,
\[
T^\lambda_{\psi, f}(s) := f R_\psi(\lambda) (R_\psi(\lambda) - s)^{-1}.
\]
We will show that $T^\lambda_{\psi,f}(s)$ admits a flat trace, and express this flat trace in terms of the orbits. Practically, one would rather consider the resolvent $(R_\psi(\lambda) - s)^{-1}$ of $R_\psi(\lambda)$ than the shifted resolvent, but this operator would not satisfy the right wavefront set condition to define the flat trace.

\subsection{Guillemin trace formula}

To start with, we need to extend the Guillemin trace formula (see eg. \cite[p313]{GS77})
to the case of Anosov actions. We write $n:=\dim\M$ and $\kappa$ the rank of the action $\tau:\mathbb{A}\to {\rm Diffeo}(\M)$.
We will follow the proof in rank $1$ by Dyatlov-Zworski \cite{DZ16a}.
Recall that the flat trace is a regularized trace for certain operators that are not trace class. The conormal to the diagonal $\Delta$ of $\M\times \M$ is given by
\[
N^*\Delta = \{ (x,x,\xi,-\xi)\ |\ x\in \M,\ \xi\in T^\ast_x \M\} \subset T^*(\M\times \M).
\]
If $P:C^\infty(\M)\to \mathcal{D}'(\M)$ is a continuous linear operator, one can consider its Schwartz kernel $\mc{K}_P\in \mathcal{D}'(\M\times \M)$, and assuming that $\WF(\mc{K}_P)\cap N^*\Delta=\emptyset$ we can set
\[
\Tr^\flat (P):= \cjg \iota_{\Delta}^* \mc{K}_P,1\cjd_{C^{-\infty},C^\infty}
\]
where $\iota_\Delta: x\in \M\mapsto (x,x)\in \Delta\subset \M\times\M$ is the inclusion. Here the pull-back is well-defined thanks to the wavefront condition, see \cite[Theorem 8.2.4]{Hoe03}.

\begin{prop}[Guillemin Trace formula]\label{prop:Guillemin}
Let $\tau : \mathbb{A} \to {\rm Diffeo}(\M)$ be an Anosov action with Weyl chamber $\W$. Then the map
\[
f \in C^\infty(\M), \psi\in C^\infty_c(\W) \mapsto \Tr^\flat\left( f \int_\W e^{-X_A}\psi(A)dA \right),
\]
is well defined, and extends as a Radon measure on $\M\times\W$. 
\end{prop}

\begin{proof}[Proof of Proposition~\ref{prop:Guillemin}]

The proof is divided into three steps. The first step consists in checking the wavefront set condition necessary to define the flat trace. Next, we need to make a local explicit computation to obtain the formula. Finally, we need to obtain some estimates to extend the formula to non-compactly supported functions.

For the first two parts of the proof, we can assume to be working with $f(x)\psi(A)$ with $\psi\in C_c^\infty(\mc{W})$, using the density of product functions in functions on $\M\times \W$. We introduce the notation
\[
R_\psi:= R_\psi(0)= \int_{\W} \psi(A)e^{-X_A} dA.
\]
\textbf{First step:} Let us show that $R_\psi$ has a well defined flat trace, which means that its Schwartz kernel $\mc{K}_{R_\psi}$ satisfies
\begin{equation}\label{conditionN^*Delta}
\WF(\mc{K}_{R_\psi})\cap N^*\Delta=\emptyset
\end{equation}
if $\Delta\subset \M\times \M$ is the diagonal. First, we consider $e^{-X}$ as an operator 
$C^\infty(\M)\otimes C_c^\infty(\W)\to C^\infty(\M)$ by $f\otimes\psi\mapsto R_\psi f$ and we consider its Schwartz kernel
$\mc{K}_{e^{-X}}\in  \mc{D}'(\W\times \M\times \M)$.
Using the formula for the wavefront set of a pushforward, we obtain
\[
\WF (\mc{K}_{R_\psi})\subset \{ (x,\eta,x',\eta')\in T^*(\M\times\M)\,|\,\exists A\in \supp(\psi), (A,0,x,\eta,x',\eta')\in \WF (\mc{K}_{e^{-X}})\}.
\]
Since $\mc{K}_{e^{-X}}(A,x,x')=\delta_{x=\varphi^A_1(x')}$, one has, by \cite[Theorem 8.2.4]{Hoe03},
\begin{equation}
\begin{split}
\WF(\mc{K}_{e^{-X}})\subset \big\{(A,-\eta(X_\bullet(\varphi^ A_1(x'))),  \varphi_1^{A}(x'),\eta,x', & -d\varphi_1^{A}(x')^T\eta  ) \in T^*(\W\times\M\times \M)  \\
& \,\big| \, A\in \W,x'\in \M, \eta\in T^*_{\varphi_1^{A}(x')}\M\setminus \{0\}\big\}
\end{split}.
\end{equation}
Thus a point of $\WF(\mc{K}_{R_\psi})$ belongs to $N^*\Delta$ if and only if there exists $x'\in\M$ and $A\in \supp \psi\subset \W$ such that $\varphi_1^{A}(x')=x'$, $\eta(X_{A'}(\varphi_1^A(x')))=0$ for all $A'\in \a$ and $\eta=d\varphi_1^{A}(x')^T\eta\not=0$. Note that $\eta(X_{A'}(\varphi_1^A(x')))=0$ implies that $\eta\in E_u^*\oplus E_s^*$, where $d\varphi_1^A(x')^T$ has no eigenvalues of modulus $1$ by normal hyperbolicity.  This shows that \eqref{conditionN^*Delta} holds.

For the \textbf{second step} we start with the following Lemma:
\begin{lemma}\label{lemma:Guillemin-local}
Let $x_0\in \M$ and $A_0\in \W$ such that $\varphi^{A_0}_1(x_0)=x_0$. There is a neighborhood $U$ of $x_0$ and $\eps>0$ such that if $B(A_0,\eps):=\{A\in \a\,| \, |A-A_0|<\eps\}$ one has $\varphi^{A}_1(x_0)\in U$ for all $|A|<\eps$ and for each $h\in C_c^\infty(U\times B(A_0,\eps))$,
\[
\Tr^\flat\Big(\int h(x,A) e^{-X_A}\, dA\Big)=\frac{1}{|\det(1-\mc{P}_{A_0}(x_0))|}\int_{|A|<\eps} h(\varphi^{A}_1(x_0),A_0)dA.
\]
\end{lemma}
\begin{proof} We follow the argument in \cite[Lemma B.1]{DZ16a}. Take an arbitrary basis $A_1,\ldots,A_\kappa$ of $\a$ and take $\phi: U\to B_{\R^n}(0,\eps)$ some diffeomorphism so that, if $y=\phi(x)$
\[
\begin{gathered}
\phi(x_0)=0,\quad  \forall x\in U, \, d\phi(x) X_{A_i}=\pl_{y_i}\\
d\phi(x_0)(E_u(x_0)\oplus E_s(x_0)) =  {\rm span}\{\partial _{y_{\kappa+1}},\ldots,\partial_{y_n}\}.
 \end{gathered}
 \]
Let $F:B_{\rr^{n-\kappa}}(0,\eps)\to B_{\rr^\kappa}(0,\eps_1)$ and $G:B_{\rr^{n-\kappa}}(0,\eps)\to B_{\rr^{n-\kappa}}(0,\eps_1)$ so that
\[
\phi \circ e^{-X_{A_0}}\circ \phi^{-1}(0,y'')=(F(y''),G(y'')), \quad y''\in \rr^{n-\kappa}, |y''|<\eps.
\]
and $F(0)=0, G(0)=0$.
For $A\in U$ and $(y',y'')\in B_{\rr^n}(0,\eps)$ we thus have, identifying $\a$ with $\rr^\kappa$
\[
\phi \circ e^{-X_A}\circ \phi^{-1}(y',y'')=(-A+A_0+y'+F(y''),G(y'')).
\]
Then for $(z',z'')$ and $(y',y'')$ in $B_{\R^n}(0,\eps)$, we can write, using that $\mc{K}_{e^{-X}}(A,x,x')=\delta_{x=\varphi^{A}_1(x')}$,
\[
\mc{K}_{e^{-X}}(A,\phi^{-1}(z',z''),\phi^{-1}(y',y''))=\delta(G(z'')-y'')\delta(y'+A-A_0-z'-F(z'')).
\]
Taking the flat trace gives
\[
\begin{split}
{\rm Tr}^\flat\Big(\int h e^{-X_A}\, dA\Big)=& \int \Big(h(\phi^{-1}(y',y''),A)\delta(y''-G(y''))\delta(A-A_0-F(y''))\Big)dy'dy''dA\\
=& \int_{B(0,\eps)} h(\phi^{-1}(y',y''), A_0+F(y''))\delta(y''-G(y''))dy'dy''.
\end{split}
\]
Now ${\rm Id}-dG(0)$ is invertible by the normal hyperbolicity of the action (it is conjugated to the Poincar\'e map $d_{x_0}\varphi^{A_0}_{-1}|_{E_u\oplus E_s}$), thus $y''=G(y'')$ has a unique solution $y''=0$ for $|y''|<\eps$ if $\eps>0$ is small enough, and we thus get
\[
{\rm Tr}^\flat\Big(\int h e^{-X_A}\, dA\Big)=\frac{1}{|\det(1-d_{x_0}\varphi^{A_0}_{-1}|_{E_u\oplus E_s})|}\int_{|y'|<\eps} h(\phi^{-1}(y',0), A_0)dy'
\]
this concludes the proof of Lemma \ref{lemma:Guillemin-local}.
\end{proof}

Now, call $T_{x_0}$ the periodic torus containing $x_0$, then if $h(x,A)=\psi(A)f(x)\in C_c^\infty(\W)\otimes C^\infty(\M)$ is such that $f$ is supported in a small neighborhood of  $T_{x_0}$ containing $x_0$, and
\[
(\cup_{T\in \mc{T}\setminus T_{x_0}} (L(T)\times T))\cap (\supp(\psi)\times \supp(f))=\emptyset ,
\]
(a choice of such $\psi, f$ is possible thanks to \eqref{finitenessperiodictori}) we have by a partition of unity and Lemma \ref{lemma:Guillemin-local}
\[\begin{split}
{\rm Tr}^\flat\Big(f\int \psi(A) e^{-X_A}\, dA\Big)=&\sum_{A\in \W\cap L(T_{x_0})}\frac{\psi(A) \int_{T_{x_0}}f\,d\lambda_{T_{x_0}}}{|\det(1-\mc{P}_{A})|}.
\end{split}\]
Consequently, using \eqref{finitenessperiodictori},
since the measure on the orbit is given by the push-forward of the Lebesgue measure on $\mathbb{A}$, the formula is thus established for compactly supported functions.

 The second observation is that given a proper subcone $\mathcal{C}\subset \W$, there exists $C>0$ such that for $T\in \T$, $A\in L(T)\cap \mathcal C$, and $x\in T$,

This complete the proof of Proposition~\ref{prop:Guillemin}.
\end{proof}

We can apply the Guillemin trace formula to our integrated propagators. Given $\psi\in C_c^\infty(\W)$ with $\int \psi=1$ and $\supp(\psi)$ contained in a small neighborhood of an element $A_0\in \W$, and $f\in C^\infty(\M)$, we obtain 
\begin{equation}\label{trafR^k}
\Tr^\flat(fR_\psi(\lambda)^k)=\sum_{T\in \mc{T}}
\sum_{A\in \W\cap L(T)}\frac{\int_{T}f\,d\lambda_T}{|\det(1-\mc{P}_{A})|}e^{-\lambda(A)}\psi^{*k}(A)
\end{equation}
where $\psi^{*k}$ is the $k$-th convolution power of $\psi$ and we used $R_{\psi}(\lambda)^k=R_{\psi^{*k}}(\lambda)$.

\begin{rem}
The formula of Proposition \ref{prop:Guillemin} has a direct extension to the action on vector bundles: instead of functions, let us consider the action on sections of some vector bundle $\mathcal{E}$. For $T\in \T$, and $A\in L(T)$, and $x\in T$, we denote by $M(A,x)$ the holonomy map on $\mathcal{E}_x$. Then for $f\in C^\infty(\mc M\times \mc C)$ the formula becomes
\[
\Tr^\flat_\mathcal{E}\left(\int_\W fe^{ - X_A}  dA \right) = \sum_{T\in \mathcal{T}} \sum_{A\in \mathcal{W}\cap L(T)} \frac{ \int_T f(x, A)\Tr M(A,x) d\lambda_T(x)}{|\det (1- \mathcal{P}_A) |}.
\]
Following the approach for flows in \cite{GLP13, BS20}, this gives a method to get rid of the Poincar\'e factor: We define for $m\in [0,n-\kappa]$ the bundle
\[
\mc{E}_0^m:= \{ \omega \in \Lambda^mT^*\M\,|\, \forall A\in \a, \,\iota_{X_A}\omega=0\}.
\]
and denote by $o(E_s)$ the orientation bundle which is a flat line bundle (see e.g. \cite[Definition 1.4]{BS20}).
The Guillemin trace formula for the bundle $\mc{E}_0^m\otimes o(E_s)$ reads
\[
\Tr^\flat\Big(\int f e^{-{ X}_A }|_{\mc{E}_0^m\otimes o(E_s)}\, dA\Big)=\sum_{T\in \mc{T}}
\sum_{A\in \W\cap L(T)}\frac{ \Tr(\Lambda^m\mc{P}_A){\rm sign}(\det (\mc{P}_A)_{|E_s}) \int_{T}f(x,A)d\lambda_T(x)}{|\det(1-\mc{P}_{A})|}.
\]
Using
\[
\det(1-\mc{P}_{A})=\sum_{m=0}^{n-\kappa}(-1)^m{\rm Tr}(\Lambda^m\mc{P}_A)\text{ and }
{\rm sign}(\det(1-\mc P_A)) = (-1)^{1+\dim E_s}{\rm sign}(\det \mc (P_A)_{|E_s})
\]
we get a new formula where the determinant of the Poincar\'e map disappears
\begin{equation}\label{case_of_forms}
\sum_{m=0}^{n-\kappa}(-1)^{1+m+\dim E_s} \Tr^\flat\Big(\int f e^{-X_A }|_{\mc{E}_0^m\otimes o(E_s)}\, dA\Big)=\sum_{T\in \mc{T}}
\sum_{A\in \W\cap L(T)} \int_{T}f(x,A)d\lambda_T(x).
\end{equation}

Note that this identity also opens a perspective to derive a Bowen formula for the measure of maximal entropy for general Anosov actions introduced in \cite{CaRH21}. In analogy to the case of Anosov flows one expects a leading pole for differential forms whose rank is the dimension of the unstable bundle. This requires however considerable additional work because the fact that leading the Ruelle-Taylor resonance gives rise to a measure is so far only established for the scalar case \cite[Proposition 5.4]{BGHW20}.
\end{rem}

\subsection{Flat trace of the shifted resolvent}

The purpose of this section is to study the shifted resolvent, and its flat trace. That is to say that for $f\in C^\infty(\M),\psi\in C_c^\infty(\mc W)$ and $\lambda\in \a_\C^\ast,s\in \C$, we will consider the function
\begin{equation}\label{eq:def-Zpsif}
Z_{\psi,f}(s,\lambda):={\rm Tr}^\flat(fR_\psi(\lambda)(s-R_\psi (\lambda))^{-1}).
\end{equation}
and prove the:
\begin{prop}\label{thm:Zf}
Let $\tau:\mathbb{A}\to {\rm Diffeo}(\M)$ be an Anosov action with Weyl chamber $\W$. There is $\eps>0$ small such that, for each 
$\psi\in C^\infty_c(\W)$ with support contained in a ball of radius $\eps$ and $\|\psi\|_{L^\infty}\leq 1$, the following holds true.
For $f\in C^\infty(\M)$, the function $Z_{\psi,f}$ of \eqref{eq:def-Zpsif} is well-defined and holomorphic in $(s,\lambda)$ for $\lambda\in \a_\C^*$ with $\Re(\lambda)$ large enough\footnote{Here, large enough means that the real part $\Re(\lambda)\in\mathfrak a^*$ is a sufficiently positive linear functional on the open Weyl chamber $\W\in\a$.} and $s\in \C$ with $|s-1|<1/2$, and it has a meromorphic extension to $B_\C(1,1/2) \times \a_\C^*$.
Moreover, for each compact set $K\subset \a_{\C}^*$ there is $N_K>0$ such that  $Z_{\psi,f}$ is holomorphic in $(\C\setminus B_{\C}(0,N_K))\times K$ and the following identity holds, with the right hand side converging:
\begin{equation}\label{Zfsum}
Z_{f,\psi}(s,\lambda)=\sum_{k=1}^\infty s^{-k}\sum_{T\in \mc{T}}
\sum_{A\in \W\cap L(T)}\frac{\int_{T}f\,d\lambda_T}{|\det(1-\mc{P}_{A})|}e^{-\lambda(A)}\psi^{*k}(A).
\end{equation}

Finally, if we replace $\psi$ by $\psi_{\sigma}:=\psi(\cdot-\sigma)$, then $Z_{f,\psi_\sigma}$ depends continuously on $\sigma$ in a small neighbourhood of $0\in \a^*$. The topology on $Z_{f,\psi_\sigma}$ is given by uniform convergence on compact subsets of the holomorphic regions in $\a_\C^* \times B_\C(1,1/2)$.
\end{prop}

The proof of this theorem will follow the ideas of the proof by Dyatlov and Zworski of the meromorphic extension of the dynamical determinant of Anosov flows \cite{DZ16a}. The idea is to use propagation of singularities, and source/sink estimates to control the wavefront set of the resolvent. We will explain this in detail. If $\Gamma\subset T^*(\M\times \M)\setminus\{0\}$ is a conic closed set, define
\[
C^{-\infty}_{\Gamma}(\M\times \M):=\{u\in C^{-\infty}(\M\times\M)\,|\, \WF(u)\subset \Gamma\}
\]
the space of distributions on $\M\times \M$ with wavefront set included in $\Gamma$. Its topology is defined using sequences in \cite[Definition 8.2.2]{Hoe03}.

To analyze the wavefront set of the resolvent of $R_\psi(\lambda)$, it will be convenient to work with a small semiclassical parameter $h>0$. As mentioned above \eqref{HNG_h}, we use a semiclassical quantization $\Op_h$ and define $\mc{H}_h^{NG}=\Op_h(e^{NG})^{-1}L^2(\M)$. Recall, that, as a vector space, $\mc{H}_h^{NG}$ is equal to $\mc{H}^{NG}$; only the norm is different. We will denote by $\Psi^m_h(\M)$ the space of semiclassical pseudo-differential operators of order $m\in \R$ (see \cite{Zwo12} or \cite[Appendix E]{DZ19}). We recall briefly that $Q\in \Psi^m_h(\M)$ can be written
as $Q=\Op_h(q)+Q'$ with $Q'$ an operator having smooth Schwartz kernel with its $C^k$ norms being $\mc{O}(h^\infty)$ for all $k\in \N$ and $q\in S^{m}(T^*\M)$ a semi-classical symbol of order $m$. We use the notation $\bbar{T}^*\mc{M}$ for the radially compactified cotangent bundle (see \cite[Appendix E]{DZ19}) and recall that the semiclassical wavefront set is the closed subset $\WF_h(Q)\subset \bbar{T}^*\M$ defined as the complement to the set of points where $q$ and its derivatives is equal to $\mc{O}(h^\infty (1+|\xi|)^{-\infty})$ (\cite[Definition E.26]{DZ19}).
We denote by $\Psi^{\rm comp}_h(\M)$ the space of those semiclassical pseudo-differential operators with compact semiclassical wavefront set. Below, we say that a family $\theta\in \R^n\mapsto \mc{K}_\theta\in C^{-\infty}(\M\times \M)$ of Schwartz kernels of operators has wavefront set contained in $\Gamma\subset T^*(\M\times \M)$ locally uniformly in $\tau$ if for each closed conic subset
$\Omega\in T^*(\M\times \M)\setminus\{0\}$ not intersecting $\Gamma$, and each $B,B'\in \Psi^0(\M)$ satisfying $\WF(B)\times\WF(B')\subset \{(x,\xi,x',-\xi')\,|\, (x,\xi,x',\xi')\in\Omega\}$, for each $N\in \N$ and each compact set $K\subset \R^n$, there is $C_{N,K,B,B'}>0$ so that for all 
$\theta\in K$ 
\[ \|B\mc{K}_\theta B'\|_{H^{-N}(\M)\to H^{N}(\M)}\leq C_{N,K,B,B'}.\]
There is a semi-classical version, for which we refer to \cite[Lemma 2.3]{DZ16a}\footnote{Our condition readily implies the condition of \cite[Lemma 2.3]{DZ16a}, for it suffices to take $U,V$ in \cite[Lemma 2.3]{DZ16a} given by the elliptic set of two $B,B'\in \Psi^0_h(\mc{M})$ that are respectively semi-classically ellpitic near $(x,\xi)$ and $(x',\xi')$.} or \cite[Lemma 6.2]{DGRS18} for the parameter dependent version. If $\mc{K}_\theta$ is an $h-$tempered distribution, a point $(x,\xi,x',-\xi')\in T^*(\mc{M}\times \mc{M})$ is not in $\WF_h(\mc{K}_{\theta})$ uniformly in $\theta\in K$ if there are $\theta$-independent neighborhoods $U$ of $(x',\xi')$ in $T^*\mc{M}$ and $V$ of $(x,\xi)$ such that for any
$B,B'\in \Psi_h^{\rm comp}(\mc{M})$ with $\WFh(B)\subset V$ and $\WFh(B')\subset U$, for any $m\geq 0$ there is $C_m$ so that for all $h>0$ small and all $\theta\in K$
\[
\|B \mc{K}_\theta B'\|_{\mc{L}(L^2)} \leq C_mh^m.
\]
For notational simplicity we shall say that the RHS is an $\mathcal{O}(h^\infty)$ uniformly in $\theta\in K.$
\begin{prop}\label{prop:WF-shifted-resolvent}
Let $\mathcal{O}$ be a ball centered in $A_0$ and of radius $\delta>0$ small\footnote{As explained in the proof of Lemma \ref{lem:WFtoprove}, the reason of taking $\mc{O}$ small is that we shall need to use an escape function that is valid for all $A\in \mc{O}$.} with closure contained in $\W$, let   $\psi \in C^\infty_c(\mathcal{O};\R^+)$ and define $\psi_\sigma=\psi(\cdot-\sigma)$ for $\sigma\in \a$ small so that $\psi_\sigma\in C_c^\infty(\mc{O})$.\\
1) There is $c_0>0$, $c_1>0$, $c_2>0$ such that for all $N>0$ 
the operator $T_{\psi_\sigma}(\lambda,s):=R_{\psi_\sigma}(\lambda)(R_{\psi_\sigma}(\lambda)-s)^{-1}:\mc{H}^{NG}\to \mc{H}^{NG}$ is holomorphic in the parameter $(\la,s)$ in the region 
\[
\{ (\la,s)\in \a_\C\times \C\,|\, {\rm Re}(\la(A_0))>-c_1N, |s|>c_2N\} 
\]
and meromorphic in the parameter $(\la,s)$ in the region
\begin{equation}\label{region_meromorphy} 
 \{ (\la,s)\in \a_\C\times \C\,|\, {\rm Re}(\la(A_0))>-c_1N, |s|>e^{-c_0N}\}.
 \end{equation}
2) Locally uniformly in $(\sigma,\lambda,s)$ (where it is defined) the Schwartz kernel
$\mc{K}_{T_{\psi_\sigma}(\lambda,s)}$  of $R_{\psi_\sigma}(\lambda)(R_{\psi_\sigma}(\lambda)-s)^{-1}$ has wavefront set contained in
\[
\WF(\mc{K}_{T_\psi(s)})\subset
\left\{(e^{X^H_A}(x,\xi),(x,\xi))\,\middle|\, (x,\xi)\in E_u^*\oplus E_s^*, \, A\in k\mathcal{O},\, k\in\N\setminus\{0\}\right\}\cup (E_u^*\times E_s^*)
\]
where $X^H_A$is the Hamilton flow of $\xi(X_A(x))$, the principal symbol of $-iX_A$ for $A\in \W$.\\ 
3) There is a cone
$\Gamma\subset T^*(\M\times\M)$ such that $\Gamma\cap N^*\Delta=\emptyset$ and
\[
(\sigma,\lambda,s)\mapsto \mc{K}_{T_{\psi_\sigma}(\lambda,s)}\in C^{-\infty}_\Gamma(\M\times \M)
\]
is continuous where it is defined.
\end{prop}

This is the main technical result; it is similar to \cite[Proposition 3.3]{DZ16a}. Using the tools of \cite[Chapter 8]{Hoe03}, we deduce directly that $Z_{f,\psi}$ is well-defined as a meromorphic function (with the continuous dependence on $\sigma$ for $Z_{f,\psi_\sigma}$). It will remain to obtain Formula \eqref{Zfsum} to prove Proposition \ref{thm:Zf}.

To prove Proposition \ref{prop:WF-shifted-resolvent}, we will rely on a wavefront set estimate for a parametrix -- much as \cite{DZ16a}. This is done in the following Lemma, the proof of which is defered to Section \ref{proofLemma3.7}. This Lemma can be viewed as a refinement of \cite[Lemma 4.14]{BGHW20}.
\begin{lemma}\label{lem:WFtoprove}
Let $\mathcal{O}$ as in Proposition \ref{prop:WF-shifted-resolvent}. 
There exists $c_0,c_1>0$ and $Q\in \Psi_{h}^{\rm comp}(\M)$ such $R_\psi(\lambda)(1-Q)-s$ has a bounded inverse on $\mc{H}_h^{NG}$ for $|s|>e^{-c_0N}$ and $\Re \lambda(A_0) > - c_1 N$, $|\Im \lambda|< h^{-1/2}$ for all $N>0$. In that region, its inverse 
\begin{equation}\label{defTspiQ}
T_{\psi_\sigma}^Q(\lambda, s):=(R_{\psi_\sigma}(\lambda)(1-Q)-s)^{-1}
\end{equation} is an analytic family of bounded operator in $(\lambda,s)$ and its Schwartz kernel $\mc{K}_{T_{\psi_\sigma}^Q(\lambda,s)}$ satisfies uniformly in $(\sigma,\lambda,s)$
\begin{equation}\label{WFtoprove}
\WF_h(\mc{K}_{T_{\psi_\sigma}^Q(s)})\cap T^*(\M\times \M)\subset N^*\Delta\cup \Omega_+(\mathcal{O})
\end{equation}
\begin{equation*}
\Omega_+(\mathcal{O}):= \bigcup_{k\geq 1} \Omega_+^k(\mathcal{O}), \quad \Omega_+^k(\mathcal{O}):=\left\{(e^{X^H_A}(x,\xi),(x,-\xi))\ \middle|\ (x,\xi)\in E_u^*\oplus E_s^*, \, A\in k \mathcal{O} \right\}.
\end{equation*}
\end{lemma}

\begin{proof}[Proof of Proposition \ref{prop:WF-shifted-resolvent}]
Using the operator $T_{\psi_\sigma}^Q(\lambda,s)$ of \eqref{defTspiQ}, we write 
\[(R_{\psi_\sigma}(\lambda)-s)T_{\psi_\sigma}^Q(\lambda,s)=1+R_{\psi_\sigma}(\lambda)QT_{\psi_\sigma}^Q(\lambda,s)\] and $T_{\psi_\sigma}^Q(\lambda,s)(R_{\psi_\sigma}(\lambda)-s)=1+T_{\psi_\sigma}^Q(\lambda,s)R_{\psi_\sigma}(\lambda)Q$ for $|s|>e^{-c_0N}$ on
$\mc{H}_h^{NG}$. Denote by $Q':=R_{\psi_\sigma}(\lambda)Q\in \Psi_h^{\rm comp}(\mc{M})$ which is compact on $\mc{H}_h^{NG}$.
Since $\|T_{\psi_\sigma}^Q(\lambda,s)\|_{\mc{H}_h^{NG}}=\mc{O}(|s|^{-1})$ uniformly in 
$h>0$ and locally uniformly in $\lambda,\sigma$  when $|s|\to \infty$, we see that $Q'T_{\psi_\sigma}^Q(\lambda,s)$ and $T_{\psi_\sigma}^Q(\lambda,s)Q'$ are invertible for some $s$.
We can then use the multivariable Fredholm analytic theorem of Proposition \ref{prop:Fredholm-several-variables}: this shows that $(1-Q'T_{\psi_\sigma}^Q(\lambda,s))^{-1}$ extends meromorphically in $(\lambda,s)$ in the region \eqref{region_meromorphy} on the space $\mc{H}_h^{NG}$ and
\[(R_{\psi_\sigma}(\lambda)-s)^{-1}=T_{\psi_\sigma}^Q(\lambda,s)(1-Q'T_{\psi_\sigma}^Q(\lambda,s))^{-1}\]
is a meromorphic extension of $(R_{\psi_\sigma}(\lambda)-s)^{-1}\in \mc{L}(\mc{H}_h^{NG})$ in \eqref{region_meromorphy} and shows 1).

This argument also implies that for $|s|\gg 1$ large
\[ (R_{\psi_\sigma}(\lambda)-s)^{-1}=T_{\psi_\sigma}^Q(\lambda,s)+ T_{\psi_\sigma}^Q(\lambda,s)Q'T_{\psi_\sigma}^Q(\lambda,s)+T_{\psi_\sigma}^Q(\lambda,s)Q'(R_{\psi_\sigma}(\lambda)-s)^{-1}Q'T_{\psi_\sigma}^Q(\lambda,s).\]
By Lemma~\ref{lem:WFtoprove}  and since $R_{\psi_\sigma}(\lambda)$ is $h$-independent, we have
\[\begin{gathered}
\WFh(T_{\psi_{\sigma}}^Q(\lambda,s)Q'T_{\psi_{\sigma}}^{Q}(\lambda,s))\cap T^*(\M\times \M)\subset \Upsilon,\\
\WFh(T_{\psi_{\sigma}}^Q(\lambda,s)Q'(R_{\psi_\sigma}(\lambda)-s)^{-1}Q'T_{\psi_{\sigma}}^Q(\lambda,s))\cap T^*(\M\times \M)\subset  \Upsilon,\\
\Upsilon:=\{(x,\xi,x',-\xi')\in T^*(\M\times \M)\, | \, (x,\xi)\in E_u^*\oplus E_s^*,\, \exists k,k'\in \N, \exists A\in k\mc{O},\\
 \exists A'\in k'\mc{O},
e^{-X_A^H}(x,\xi)\in
\WFh(Q'), e^{X_{A'}^H}(x',\xi')\in
\WFh(Q')\}
\end{gathered}\]
and this holds locally uniformly with respect to $(\sigma,\lambda,s)$.
Since $\WFh(R_{\psi_\sigma}(\lambda))\subset \Omega_+^1$, this shows that
\[\WFh(R_{\psi_\sigma}(\lambda)(R_{\psi_\sigma}(\lambda)-s)^{-1})\cap T^*(\M\times \M)\subset \Omega_+(\mc{O})\cup \Upsilon.\]
Now, using that $\WFh(Q')$ is a compact set, we observe by hyperbolicity of the action that if $L\gg 1$ is large, then the set
\[\{(x,\xi)\in T^*\M\,|\, |\xi|>L\}\cap \bigcup_{k\geq 0}\bigcup_{A\in k\mc{O}}e^{\pm X_A^H}(\WFh(Q')\cap (E_u^*\oplus E_s^*))\]
is contained in $\{|\xi|>L\}\cap \mc{C}^L_{\pm}$ where $\mc{C}^L_{+}$ is a small conic neighborhood of $E_u^*$ and $\mc{C}^L_{-}$ is a small conic neighborhood of $E_s^*$, with the size of the cone sections going to $0$ as $L\to +\infty$. Since $T_{\psi_\sigma}(\lambda,s)=R_{\psi_\sigma}(\lambda)(R_{\psi_\sigma}(\lambda)-s)^{-1}$ is independent of $h$, we have
$\WF(\mc{K}_{T_{\psi_\sigma}(s)})=\WFh(\mc{K}_{T_{\psi_\sigma}(s)})\cap T^*(\M\times \M)\setminus \{0\}$, thus by taking $L\to \infty$ we obtain the desired statement, and this holds locally uniformly with respect to $(\sigma,\lambda,s)$. In particular, since $\Omega_+(\mc{O})\cup (E_u^*\times E_s^*)$ is disjoint from $N^*\Delta$, we have proved 2) and 3) by choosing a cone $\Gamma$ containing $\Omega_+(\mc{O})\cup (E_u^*\times E_s^*)$.
\end{proof}

By the wavefront estimates from Proposition~\ref{prop:WF-shifted-resolvent} it follows that the flat trace
$Z_{f,\psi_\sigma}(s,\lambda)$ defined by \eqref{eq:def-Zpsif} is well defined and depends meromorphically on $(\lambda,s)$ in \eqref{region_meromorphy} for each $N>0$, i.e. in $\a_\C \times (\C\setminus\{0\})$: indeed, this is a consequence of the fact that the flat trace ${\rm Tr}^\flat$ is a continuous linear form on the space $C^{-\infty}_\Gamma(\M\times \M)$ equipped with its natural topology given by \cite[Definition 8.2.2]{Hoe03}, provided $\Gamma\cap N^*\Delta=\emptyset$ (the meromorphicity in $\la$ can be seen by using the Cauchy characterization of holomorphic functions using contour integrals).
To finish the proof of Proposition~\ref{thm:Zf}, it suffices to prove the expansion \eqref{Zfsum} for some open set of $s,\lambda$: indeed, both sides of \eqref{Zfsum} are analytic with respect to $\lambda,s$ in some open set and the equality will hold by analytic continuation. For $\psi\in C_c^\infty(\W,\R^+)$ with support near a given $A_0\in\W$ as above and $|\sigma|<\eps$ small, since $R_{\psi_\sigma}(\lambda)$ is bounded on $\mc{H}^{NG}$, for $|s|\gg 1$ we have as a converging series
\begin{equation}\label{series_resolvante}
R_{\psi_\sigma}(\lambda)(s-R_{\psi_\sigma}(\lambda))^{-1}=\sum_{k=1}^{\infty} s^{-k}R_{\psi_\sigma}(\lambda)^k.\end{equation}
Formally, if we take the flat trace of the above identity multiplied by $f$ on both sides, we obtain using \eqref{trafR^k} the desired identity \eqref{Zfsum}.
\begin{lemma}\label{Zorbitesperiod}
Let $\psi\in C_c^\infty(\W,\R^+)$, $|\sigma|<\eps$ and $\psi_\sigma=\psi(\cdot-\sigma)$ as above. Then for each $\lambda\in \a_\C^*$, if $|s|\gg 1$ is large enough, we obtain
\begin{equation*}
Z_{f,\psi_\sigma}(s,\lambda)=\sum_{k=1}^\infty s^{-k}\sum_{T\in \mc{T}}
\sum_{A\in \W\cap L(T)}\frac{\int_{T}f\,d\lambda_T}{|\det(1-\mc{P}_{A})|}e^{-\lambda(A)}\psi_\sigma^{*k}(A).
\end{equation*}
Moreover the left hand side extends meromorphically in $(\lambda,s)\in \a_\C\times (\C\setminus \{0\})$.
\end{lemma}
\begin{proof} First, as in the rank $1$ case, we need an exponential estimate on the number of periodic orbits in the region $\{A\in\W\,|\, |A|\leq L\}$ as $L\to \infty$, which will ensure the convergence of the RHS of \eqref{Zfsum} when $|s|\gg 1$ is large. This is the content of Lemma \ref{estimate_periodic}.

We then follow the argument of \cite{DZ16a} in the rank $1$ case. By \cite[Section 2.4]{DZ16a}, there is a family of operators $E_\eps$ with smooth integral kernels $E_\eps(x,y)=C_\eps(x)F(d_g(x,y)/\eps)$ approximating the Identity as a bounded operator $H^{\eps_0}(\M)\to L^2(\M)$ as $\eps\to 0$ for any fixed small $\eps_0>0$, with $C_\eps(x)=\mc{O}(\eps^{-n})$ where $n:=\dim\M$ and $F\in C_c^\infty(\R, [0,1])$ satisfies $F(r)=1$ near $r=0$. Furthermore for each $A:C^\infty(\M)\to \mc{D}'(\M)$ with $\WF(A)\cap N^*\Delta=\emptyset$  (see \cite[Lemma 2.8]{DZ16a})
\begin{equation}\label{approx_trace}
\lim_{\eps\to 0}{\rm Tr}(E_\eps AE_\eps)={\rm Tr}^\flat(A).
\end{equation}
Moreover, the proof of \cite[Lemma 4.1]{DZ16a} yields that there is $C>0$ such that for each $A\in\W$ (with $\|\cdot\|_{\rm Tr}$ the trace norm) $\|E_\eps e^{-X_A}E_\eps\|_{{\rm Tr}}\leq Ce^{C|A|}\eps^{-n-\delta_0}$ for any $\delta_0>0$ fixed small. Since $\supp(\psi_{\sigma}^{*k})\subset \{A\in \a\,|\, |A-kA_0|<\delta k\}$ for some small $\delta>0$ (depending on $\supp(\psi)$) and $A_0\in\W$ fixed, this implies that there is $C>0$ such that for all $\eps>0$ and $k\in\N$
\begin{equation}\label{normetrace}
\|E_\eps R_{\psi_\sigma}^k(\lambda)E_\eps\|_{{\rm Tr}}\leq Ce^{C\langle \lambda \rangle k}\eps^{-n-\delta_0}.
\end{equation}
This proof also gives for some uniform $C>0$
\[\begin{split}
|{\rm Tr}(E_\eps R_{\psi_\sigma}^k(\lambda)E_\eps)|\leq &  \int_{\W}e^{-{\rm Re}\lambda(A)}\psi_{\sigma}^{*k}(A)|{\rm Tr}(E_\eps e^{-X_A}E_\eps)| dA\\
\leq & \int_{\W}\psi_{\sigma}^{*k}(A)e^{-{\rm Re}\, \lambda(A)}\int_{\M\times \M} E_\eps(x,y)E_\eps(\varphi^A_{-1}(y),x)dv_g(x)dv_g(y)dA\\
\leq & C\eps^{-2n}\int_{\W\times\M\times\M} \psi_{\sigma}^{*k}(A)e^{-{\rm Re}\, \lambda(A)} {\bf 1}_{\{d_g(x,y)<c_1\eps,d_g(x,\varphi^A_{-1}(y))<c_1\eps\}}dv_g(x)dv_g(y)dA\\
\leq & C\eps^{-n}\int_{\W\times \M}\psi_{\sigma}^{*k}(A) e^{-{\rm Re}\,\lambda(A)} {\bf 1}_{\{d_g(y,\varphi^A_{-1}(y))<2c_1\eps\}}dv_g(y)dA.
\end{split}\]
We used Lidskii's theorem in the second line.
Using Lemma \ref{estimate_periodic} and the support property of $\psi_\sigma^{*k}$, this gives
\begin{equation}\label{normetrace2}
|{\rm Tr}(E_\eps R^k_{\psi_\sigma}(\lambda)E_\eps)|\leq Ce^{C\langle\lambda\rangle k}.
\end{equation}
Using respectively \eqref{series_resolvante}, \eqref{normetrace} and \eqref{normetrace2} we can thus write for $\lambda$ fixed and $|s|$ large enough
\[\begin{split}
Z_{f,\psi_\sigma}(s,\lambda)=& \lim_{\eps\to 0}{\rm Tr}(f\sum_{k=1}^{\infty} s^{-k}E_\eps R_{\psi_\sigma}(\lambda)^kE_\eps)=\lim_{\eps\to 0}\sum_{k=1}^{\infty} s^{-k}{\rm Tr}(fE_\eps R^k_{\psi_\sigma}(\lambda)E_\eps)\\
=& \sum_{k=1}^{\infty} s^{-k}\lim_{\eps\to 0}{\rm Tr}(fE_\eps R^k_{\psi_\sigma}(\lambda)E_\eps)= \sum_{k=1}^{\infty} s^{-k}{\rm Tr}^\flat(fR^k_{\psi_\sigma}(\lambda))
\end{split}\]
where we used \eqref{approx_trace} for the last identity. The formula \eqref{trafR^k} concludes the proof since the meromorphy in $(\lambda,s)$ was discussed before the Lemma.
\end{proof}

\subsection{Proof of Lemma \ref{lem:WFtoprove}}\label{proofLemma3.7}

1) We start with the proof of existence of $T^Q_\psi$. In that aim, we will show that there is $c_1>0,c_0>0$ such that, for all $h>0$ small and all $N>0$,
$R_{\psi}(\la)$ is quasi-compact on $\mc{H}_{h}^{NG}$  for $|s|>e^{-c_0N}$ and $\Re \lambda(A_0) > - c_1 N$, $|\Im \lambda|< h^{-1/2}$.
We pick $\Gamma_{E_0^*}\subset T^*\M$, a conic neighborhood of $E_0^*$, and then $G\in C^\infty(T^*\M)$ an escape function for all $A\in \mc{O}$ compatible with $c_X>0$ and $\Gamma_{E_0^*}\subset T^*\M$ in the sense of \cite[Definition 4.1]{BGHW20}. 
This is possible by \cite[Lemma 4.4]{BGHW20}; the proof amounts to take an escape function for  $A_0\in\mathcal{O}$, and show it remains adapted to all $A$'s sufficiently close to $A_0$. Notice that this is the first reason for assuming that $\mathcal{O}$ is small enough.
By the properties \eqref{escapefct} of the escape function, there is $c_X>0$ such that for $A\in\mathcal{O}$, and some $r>0$ large enough,
\[
\bigcup_{t\in [0,1]} e^{tX_A^H}(x,\xi)\cap  \Gamma_{E_0^*}=\emptyset,\  |\xi|>r \Longrightarrow   G(e^{X_{A}^H}(x,\xi))-G(x,\xi)\leq -c_X.
\]
As it will be used several times below, notice that for $\psi \in C^\infty_c(\mathcal{O})$,
\[
\supp(\psi^{*k})\subset \supp(\psi)+\dots+\supp(\psi) \subset k\mathcal{O}.
\]
This is contained in a ball of radius $\delta k$ centered at $kA_0$ for some small $\delta>0$. We imitate now the proof of \cite[Lemma 4.5]{BGHW20} but with a semiclassical quantization. The idea here is that in the direction of the flow $E^\ast_0$, averaging the flow is regularizing, and in the transverse direction, we can use the escape function to obtain some compactness. Let $\Gamma_0:=\Gamma_{E_0^*}\cap\{|\xi|\geq 1\}$ and choose $P\in \Psi_h^0(\M)$ which satisfies
\[
\WF_h(P)\subset \left\{(x,\xi)\in \bbar{T}^*\M\ \middle|\  \forall t\in [0,1], \forall A\in \mc{O}, \ e^{tX_A^H}(x,\xi)\notin \Gamma_0\right\}
\]
and the semi-classical principal symbol $\sigma(P)$ of $P$ satisfies $0\leq \sigma(P)\leq 1$. The operator $P$ is microlocalizing away from the neutral direction $E^\ast_0$. We also pick $\Gamma'_0$ a neighbourhood of $\Gamma_0$ which is conic for $|\xi|\geq 1$ and contained in $|\xi|>1/2$, and assume that
\[
\WF_h(1-P) \subset \Gamma'_0
\]
(i.e. $P$ is microlocally equal to $1$ outside $\Gamma_0$.) Note that we can chose $\Gamma_0'$ such that $T^*\M\setminus \Gamma_0$ is for $|\xi|>1$ an arbitrary small cone around $E_u^*\oplus E_s^*$. Setting for $A\in \mathcal{O}$, $B_A:=e^{X_A}\Op_h(e^{NG})e^{-X_A}$, we have
\[
\Op_h(e^{NG})e^{-X_A}P\Op_h(e^{NG})^{-1}=e^{-X_A}B_AP\Op_h(e^{NG})^{-1}.
\]
The semi-classical principal symbol of $B_AP\Op_h(e^{NG})^{-1}\in \Psi_h^0(\M)$ is (by Egorov's Lemma and the composition rule for peudo-differential operators)
\[
\sigma(B_AP\Op_h(e^{NG})^{-1})=e^{N(G\circ e^{X^H_A}-G)}\sigma(P) \mod hS^{-1}.
\]
Using the properties \eqref{escapefct} of the escape function $G$, we find for $r>0$ large enough and some $c_X>0$
\[
\sup_{|\xi|\geq r}|\sigma(B_AP\Op_h(e^{NG})^{-1})(x,\xi)|\leq
e^{-c_XN}.
\]
Next, we introduce $Q_0\in \Psi^{\rm comp}_h(\M)$ so that $\WF_h(Q_0)\subset \{|\xi|\leq 2r\}$ and $\WF_h(1-Q_0)\subset \{|\xi|\geq r\}$ with $0\leq \sigma(Q_0)\leq 1$. We also let $C_{\mc{O}}:=\max_{A\in\mc{O}}\|e^{-X_A}\|_{\mc{L}(L^2(\M))}$. The previous estimate implies that for all $h>0$ small and $A\in \mc{O}$,
\begin{equation}\label{firstboundeXA}
\|e^{-X_A}P(1-Q_0)\|_{\mc{L}(\mc{H}_h^{NG})}\leq C_{\mc{O}}e^{-c_XN}+\mc{O}(h).
\end{equation}
We used $[Q_0,\Op_h(e^{NG})^{-1}]\Op_h(e^{NG})\in h\Psi^{-1}_h(\M)$.
We get (using $\int \psi= 1$)
\[
\|R_\psi(\lambda) P(1-Q_0)\|_{\mc{L}(\mc{H}_h^{NG})} \leq C_{\mc{O}}e^{-c_XN}\sup_{A\in\mathcal{O}}e^{-\Re \lambda(A)}+\mc{O}(h).
\]
Now, we observe that the operator $e^{-X_A} PQ_0$ is compact and smoothing (i.e. mapping $\mc{D}'(\mc{M})\to C^\infty(\mc{M})$), thus $R_\psi(\lambda) PQ_0$ is also compact and smoothing.
We have thus shown that there is $c_0>0$ and $c_1>0$ such that for all $m>0$, all $h>0$ small enough
and ${\rm Re} \lambda(A_0)>-c_1N$ and $|{\rm Im}(\lambda)|<h^{-1/2}$
\begin{equation}\label{boundsR^0R^1}
\begin{gathered}
R_{\psi}(\lambda)= R_{\psi}^0(\lambda)+R_{\psi}^1(\lambda) + R_\psi(\lambda) Q, \\
R_{\psi}^0(\lambda):=R_{\psi}(\lambda)P(1-Q_0), \quad
 R_{\psi}^1(\lambda):=R_{\psi}(\lambda)(1-P), \quad Q:= PQ_0 ,\\
\|R_{\psi}^0(\lambda)\|_{\mc{L}(\mc{H}_h^{NG})}\leq e^{-c_0N} , \quad R_\psi(\lambda) Q \textrm{ compact on }\mc{H}_h^{NG}
\end{gathered}
\end{equation}
To prove that $R_{\psi}(\la)$ is quasi-compact with $R_{\psi}(\lambda)(1-Q)-s$ is invertible on $\mc{H}_{h}^{NG}$ for $|s|>e^{-c_0N}$, ${\rm Re}(\la)(A_0)>-c_1N$, $|{\rm Im}(\la(A_0)|<h^{-1/2}$ and $h>0$ small enough, we have to show that the norm of $R_{\psi}^1(\lambda)$ is small.
This term is microlocalized in the region of ellipticity of the action. We use the smoothing\footnote{Here the smoothing  also happens in the semi-classical parameter in the sense that $R_\psi(\la)(1-P)$ will be 
$\mc{O}(h^m)$ for all $m>0$ due to the microlocalization in the ellpitic region.} effect in the direction of the flow: for each $A_{j}\in \W$, by integration by parts we obtain that
\[R_\psi(\lambda) (X_{A_j}+\lambda(A_j))=\int_{\W}(\pl_{A_j}\psi (A))e^{-X_A-\lambda(A)}dA=R_{\pl_{A_j}\psi}(\lambda).\]
This implies that 
\[ R_\psi(\lambda)\Delta_{\mathbb{A}}^0(\lambda)=R_{\Delta_{\mathbb{A}}^0(0)\psi}(\lambda)\]
where, for $A_1,\dots,A_\kappa\in \W$ a local basis of $\a$, we have set
\[
\Delta_{\mathbb{A}}^0(\lambda) :=-\sum_{j=1}^\kappa (\partial_{A_j}+\lambda(A_j))^2,\quad \Delta_{\mathbb{A}}(\lambda) := \tau_\ast \Delta_{\mathbb{A}}^0(\lambda)=-\sum_{j=1}^\kappa (X_{A_j}+\lambda(A_j))^2.
\]
(The first one acts on $\mathbb{A}$ while the second acts on $\M$.) Since we assumed that $|\Im \lambda|<h^{-1/2}$, $h^2\Delta_{\mathbb{A}}(\lambda)$ is elliptic on the wavefront set of $1-P$, uniformly in $\lambda$. We can thus find uniformly in $\lambda$ for each $m\in\N$ a parametrix $S(\lambda)\in \Psi^{-2m}_h(\M)$ so that
\[
(h^2\Delta_\mathbb{A} (\lambda))^mS(\lambda)(1-P)-(1-P) \in h^\infty \Psi_h^{-\infty}(\M)
\]
(actually, $S(\lambda)$ is a holomorphic function of $h\lambda$). We thus deduce that
\[
R_{\psi}(\lambda)(1-P)-h^{2m} \int_{\W} [\Delta_{\mathbb{A}}^0(0)^m\psi_\sigma](A)e^{-X_A-\lambda(A)}S(\lambda)(1-P)  \, dA \in h^\infty\Psi_h^{-\infty}(\M).
\]
(here the bound on the remainder does not depend on $\psi$ since $\| \psi\|_{L^1} =1$.) In particular, for each $m\geq 0$ there is $C_m>0$ depending continuously on $\| D^{2m}\psi\|_{L^1(\mathcal{O})}$ and $N$ such that for all $\la$ with ${\rm Re}(\la(A_0)>-c_1N$, $|\Im \lambda|<h^{-1/2}$ and all $h>0$ small
\begin{equation}\label{R(1-P)}
\|R_{\psi}(\lambda)(1-P)\|_{\mc{L}(\mc{H}_h^{NG})}\leq C_mh^{2m}.
\end{equation}
The same argument (using the operator $\Delta_{\mathbb{A}}(\lambda)$ acting on the other side) shows that for the
 same range or parameters 
 \begin{equation}\label{R(1-P)2}
\|(1-P)R_{\psi}(\lambda)\|_{\mc{L}(\mc{H}_h^{NG})}\leq C_mh^{2m}.
\end{equation}
Now, we observe that the operator $e^{-X_A} PQ_0$ is compact and smoothing.
We thus conclude that there is $c_0>0$ and $c_1>0$ such that for all $m>0$, all $h>0$ small enough
and ${\rm Re} \lambda(A_0)>-c_1N$ and $|{\rm Im}(\lambda)|<h^{-1/2}$
\begin{equation}\label{boundR^1}
 \|R_{\psi}^1(\lambda)\|_{\mc{L}(\mc{H}_h^{NG})}\leq C_mh^m.
\end{equation}
This shows that $R_{\psi}(\lambda)(1-Q)-s$ is invertible on $\mc{H}_{h}^{NG}$ for $|s|>e^{-c_0N}$ and $h>0$ small enough, locally uniformly in $\psi$ and $\lambda$. We call $T_\psi^Q(\lambda,s)=(R_{\psi}(\lambda)(1-Q)-s)^{-1}$ its inverse.\\

2) The second step is to prove the announced property of its wavefront set. We will assume that
$\psi_\sigma=\psi(\cdot-\sigma)$ is a family of $C_c^\infty(\mc{O})$ functions for some parameter $\sigma\in \a^*$ small. We thus have a family $T_{\psi_\sigma}^Q(\lambda,s)$ of operators depending on $(\lambda,s,\sigma)$ and we will assume it lives in a small compact set $K$ where $T_{\psi_\sigma}^Q(\lambda,s)$ is well-defined.

Let us start with the elliptic region. We observe that
\begin{equation}\label{eq:TQpsi}
\begin{split}
T^Q_{\psi_\sigma}(\lambda,s)= & - \frac{1}{s}\left(1 - \frac{R_{\psi_\sigma}(\lambda)(1-Q)}{s} \right)^{-1} \\
 = &-\frac{1}{s}  - \frac{R_{\psi_\sigma}(\lambda)(1-Q)}{s^2} -\frac{R_{\psi_\sigma}(\lambda)(1-Q)}{s^3}T^Q_{\psi_\sigma}(\lambda,s)R_{\psi_\sigma}(\lambda)(1-Q).
\end{split}
\end{equation}
Applying respectively  $(1-P)$ on the right and then on the left (for $P$ chosen as above with a $\Gamma_0'$ coming arbitrary close to $E_u^*\oplus E_s^*$) and using that $\WFh(1-Q)\subset \{|\xi|\geq r\}$, we obtain using \eqref{R(1-P)} and \eqref{R(1-P)2} that there is $r'\in (0,r)$ such that uniformly in $(\lambda,s,\sigma)\in K,$
\begin{equation}\label{eq:WFTQ}
\begin{split}
\WF_h(&\mc{K}_{T_{\psi_\sigma}^Q(\lambda,s)})\  \cap\  T^\ast (\M\times\M ) \subset \\
	&N^\ast \Delta\ \cup  \left\{(x,\xi,x',\xi')\ \middle|\ (x,\xi) \in E_u^*\oplus E_s^* \text{ and }(x',\xi') \in E_u^*\oplus E_s^*,\ \text{and }|\xi|,\ |\xi'|\geq  r' \right\}.
\end{split}
\end{equation}
Indeed, first one has ${\rm WF}_h(1/s)=N^\ast \Delta$, then,  using Egorov lemma, 
\[{\rm WF}_h(R_{\psi_\sigma}(\lambda)(1-Q))\subset \{ (x,\xi,x',\xi')\,\, | \,\, |\xi'|\geq r,  |\xi|\geq r'\}\]
for some $r'>0$ such that $e^{-X_A^H}(\{|\xi|\geq r\})\subset \{\xi\geq r'\}$ for all $A\in \supp(\psi)$. The result follows from composition of semiclassical wavefront sets.

It remains to consider the intersection of the wavefront set with $(E_u^\ast \oplus E_s^\ast)^2$. Expanding the formula for $T^Q_{\psi_\sigma}(\lambda,s)$, we get a convergent sum in $\mathcal{L}(\mathcal{H}^{NG}_h)$
\[
T^Q_{\psi_\sigma}(\lambda,s)= - \frac{1}{s} \sum_{k\geq 0} \left(\frac{R_{\psi_\sigma}(\lambda)(1-Q)}{s} \right)^k .
\]
This is the inspiration for the wavefront set statement. However, since the terms in the sum are not increasingly smoothing, only smaller and smaller, we cannot directly obtain the desired statement.

Let us rewrite exactly what we need to prove. For each $B,B' \in \Psi_h^{\rm \comp}(\mc{M})$  microsupported near $(E^\ast_u \oplus E^\ast_s)\cap \{ |\xi|\geq  r'\}$, so that
$\WFh(B')\cap \Omega_B=\emptyset$ with
\[
\Omega_B :=\left\{ (x,\xi)\ \middle|\ e^{X^H_A}(x,\xi) \in \WF_h(B),\ A\in \cup_{k\geq 0} k\mathcal{O}\right\},
\]
it suffices to prove that uniformly for $(\lambda,s,\sigma)\in K$
\begin{equation}\label{esttoprove}
\|B T^Q_{\psi_\sigma}(\lambda,s) B'\|_{\mc{L}(L^2)} = \mathcal{O}(h^\infty)
\end{equation}
Notice that in \eqref{esttoprove} one can use $\mc{H}_h^{NG}$ norms as well  since $\|B\|_{\mc{L}(\mc{H}^{NG}_h,L^2)}+\|B\|_{\mc{L}(L^2,\mc{H}^{NG}_h)}$ (and same for $B'$) is bounded by some $h^{-CN}$ for some uniform $C>0$, as $B,B$ have compact semiclassical wavefront sets.
Note also that the conormal region $N^*\Delta$ is covered because we now include $k=0$ in the definition of $\Omega_B$.

It will be convenient below to use that for $|s|>e^{-c_0N}$, if $u=T_{\psi_\sigma}^Q(\lambda,s)f$ one has $u=(-f+R_{\psi_\sigma}^0(\lambda)u+R_{\psi_\sigma}^1(\lambda)u)/s$ (recall \eqref{boundsR^0R^1}). Thus for each $B\in \Psi_h^{0}(\mc{M})$ with $|\sigma(B)|\leq 1$, we have uniformly for $(\lambda,s,\sigma)\in K$
\begin{equation}\label{estimateB0}
\|Bu\|_{\mc{H}_h^{NG}}\leq |s|^{-1}(\|Bf\|_{\mc{H}_h^{NG}}+\|BR_{\psi_\sigma}^0(\lambda)u\|_{\mc{H}_h^{NG}})+\mc{O}(h^\infty).\end{equation}
Here, for bounding $\|BR_{\psi_\sigma}^1(\lambda)u\|_{\mc{H}_h^{NG}}$,
 we have used the bound \eqref{boundR^1} and the fact that $T_{\psi_\sigma}^Q(\lambda,s)$ has norm uniformly bounded with respect to $(\lambda,s,\sigma,h)\in K\times (0,h_0)$ for $h_0>0$ fixed small.

Our strategy to prove \eqref{esttoprove} decomposes into 3 steps. First, we shall estimate the norm $\|B_1u\|_{\mc{H}_h^{NG}}$ of the microlocalized term $B_1u$ for $B_1={\rm Op}_h(b_1)$ with $b_1$ a well-chosen function that is $=1$ near $E_s^*\cap \pl \bbar{T}^*\mc{M}$ (which is a source for the symplectic lift $e^{X_A^H}$ of the flow $e^{X_A}$) in terms of $\|B_1f\|_{\mc{H}_{h}^{NG}}$. 
This is reminiscent of the radial source estimate in \cite{DZ16a} and is a consequence of the smallness of the propagator 
$e^{-X_A}$ on the anisotropic space $\mc{H}_h^{NG}$, given in \eqref{firstboundeXA} and \eqref{boundsR^0R^1}. Second, we use the propagation of regularity (a consequence of Egorov lemma) to control the microlocal norm of $u$ in the regions of $\bbar{T}^*\mc{M}$ 
which can be accessed by the flow $e^{X_{kA}^H}$ (for $k\geq 1$ and $A\in \mc{O}$) from $\supp(b_1)$ where we already have microlocal 
regularity. This will show \eqref{esttoprove} when ${\rm WF}_h(B)\cap E_u^*=\emptyset$.  There remains to analyse the case when ${\rm WF}_h(B)$ is localized in a small ball around a point in $E_u^*$: 
the idea is that if ${\rm WF}_h(B')$  is compact, there is $k\geq 1$ large enough so that $e^{-X_{kA}^H}({\rm WF}_h(B))$ (for $A\in \mc{O}$) is going to be contained in a small neighbourhood of $\xi=0$ by the property of $e^{X_A^H}$ on $E_u^*$, and thus can not intersect ${\rm WF}_h(B')$.

First we will need a so-called \emph{source estimate} close to $E_s^*\cap \pl\bbar{T}^*\M$ in order
to control the wavefront set in that region.

\textbf{(i) The source estimate.}
Let $V_s\subset \bbar{T}^*\M$ be a small neighborhood of $L:=E_s^*\cap \pl\bbar{T}^*\M$ so that $e^{-X^H_A}(V_s)\subset V_s$ for all $A\in \mc{O}$, $m(x,\xi)$ is the function of \eqref{escapefct}, which is chosen to be constant in $V_s$ and finally $P=1$ microlocally on $V_s$.
Now using the fact that $L$ is a repulsor (a source) for the flow $e^{X_A^H}$ with $A\in\mc{O} $, there is $b_1,b_2\in S^0(T^*\M)$ with $b_1=1$ near $L$, $\supp(b_2)\subset \{b_1=1\}$ so that
$b_2\circ e^{-X_A^H}=1$ on ${\rm supp}(b_1)$ for all $A\in \mc{O}$. See Figure \ref{fig:phase_space}. Notice that for $v\in \mathcal{H}^{NG}_{h}$ and $B_{j}=\Op_h(b_{j})$ for $j=1,2$, one has
\[
\| B_2 v\|_{\mc{H}_h^{NG}} \leq \| B_1 B_2 v\|_{\mc{H}_h^{NG}} + \mathcal{O}(h^\infty\| v\|_{\mc{H}_h^{NG}}).
\]
By Egorov's Lemma, $e^{X_A}B_1e^{-X_A}$ satisfies  ${\rm WF}_h(e^{X_A}B_1e^{X_A})\subset e^{-X_A^H}({\rm WF}_h(B_1))$ which is contained in the set $\{b_2=1\}$ if $A\in \mc{O}$, thus one obtains $B_1 R_{\psi_\sigma}^0(\lambda)-B_1 R_{\psi_\sigma}^0(\lambda) B_2\in h^\infty \Psi_h^{-\infty}(\M)$. The remainder here depends continuously on $\sigma,\lambda$ as before. Thus, using \eqref{boundsR^0R^1} we get for all $v\in C^\infty(\M)$
\begin{align*}
 \|B_1 R_{\psi}^0(\lambda)v\|_{\mc{H}_h^{NG}}	&\leq  \|B_1 R_{\psi}^0(\lambda) B_2 v\|_{\mc{H}_h^{NG}} +\mc{O}(h^\infty\|v\|_{\mc{H}_h^{NG}}),\\
		&\leq e^{-c_0N}\|B_2 v\|_{\mc{H}_h^{NG}}+ \mc{O}(h^\infty\|v\|_{\mc{H}_h^{NG}}),\\
		&\leq e^{-c_0N}\|B_1 v\|_{\mc{H}_h^{NG}}+ \mc{O}(h^\infty\|v\|_{\mc{H}_h^{NG}}).
\end{align*}
Combining with \eqref{estimateB0}, and assuming $e^{-c_0N}|s|^{-1}\leq 1/2$, this yields that there is $C>0$ such that for all $f \in C^\infty(\M)$
\begin{equation}\label{radialsource}
\|B_1u\|_{\mc{H}_h^{NG}}=\| B_1 T_{\psi_\sigma}^Q(\lambda,s) f\|_{\mc{H}_h^{NG}} \leq C \| B_1 f \|_{\mc{H}^{NG}_h} + \mathcal{O}(h^\infty \|f\|_{\mc{H}_h^{NG}})
\end{equation}
uniformly for $(\lambda,s,\sigma)\in K$. This estimate is somehow related to \cite[Proposition 2.6]{DZ16a} but is slightly different (and actually simpler) in that we work here with propagators $e^{-X_A}$ rather than the vector field generating the flow.\\

\begin{figure}
\includegraphics[scale=0.3]{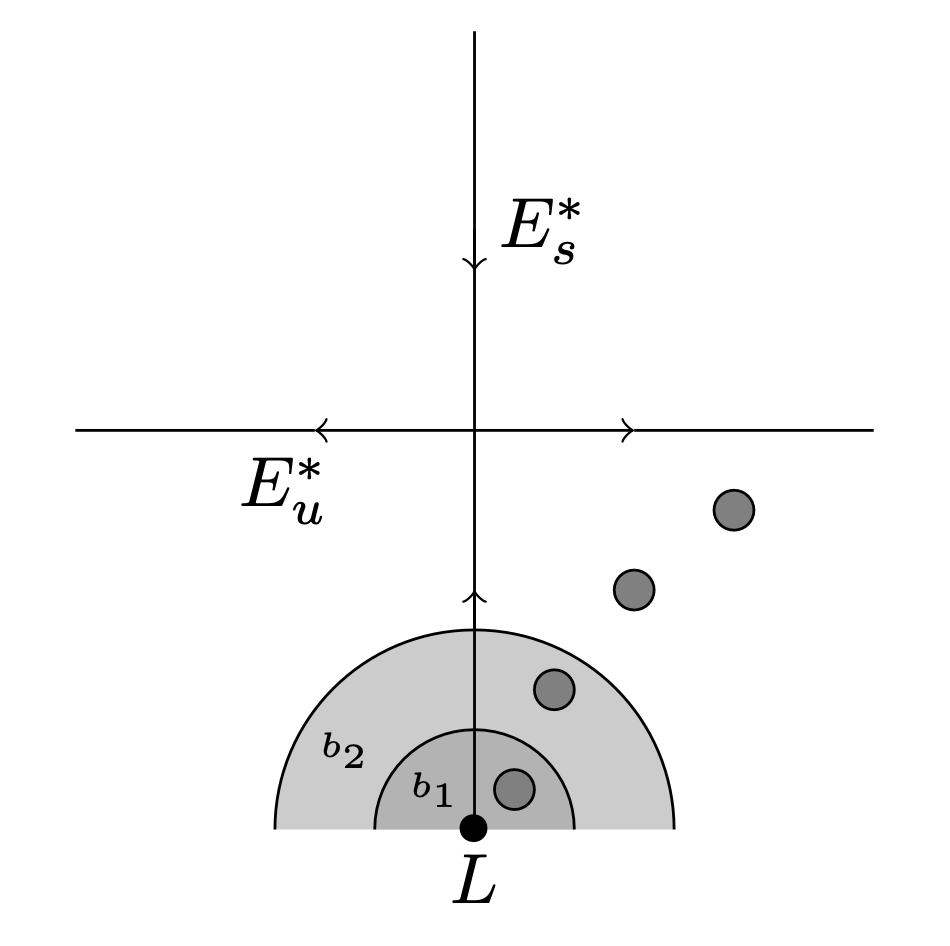}
\caption{The support of $b_1$ and $b_2$ near $L$ in the compactified $\bbar{T}^*\mc{M}$, as defined in (i). The darker disks represent $e^{-X_A^H}({\rm WF}_h(B))$ for $A\in \ell \mc{O}$ for $\ell:=0,\dots,k$, with the last one $\ell=k$ being contained in the region $\{b_1=1\}$ as explained in (ii).}
\label{fig:phase_space}
\end{figure}

(ii) \textbf{Propagation estimate outside $E_u^*$.} We will next show how to use the source estimate to obtain information on the wavefront set of $\mc{K}_{T_{\psi_\sigma}(\lambda,s)}$ outside $E_u^*$.

Assume that $B\in \Psi_h^{\rm comp}(\M)$ satisfies $\WFh(B)\cap E_u^*=\emptyset$, $0\leq \sigma(B)\leq 1$ and $\WFh(B)$
contained in a small neighborhood of $E_u^*\oplus E_s^*$ and $B'\in \Psi_h^{\rm comp}(\mc{M})$ satisfying $\WFh(B')\cap \Omega_B=\emptyset$. Since $L$ is an repulsor, there is $k\in \N$ large enough such that for all $A\in k\mc{O}$, $e^{-X_A^H}(\WFh(B))\subset V_s$ and $\WFh(B')\cap V_s=\emptyset$ for some small enough neighborhood $V_s$ of $L$, invariant by $e^{-X_A^H}$ as in (i). We will use $B_1$ as in (i) for this set $V_s$ and we can assume, up to taking $k$ even larger, that $e^{-X_A^H}(\WFh(B))\subset \supp(b_1)$ for all $A\in k\mc{O}$. See Figure \ref{fig:phase_space}.

Iterating \eqref{estimateB0}, we obtain for each $k\geq 1$ and each $f\in C^\infty(\mc{M})$
(with $u=T_{\psi_\sigma}^Q(\lambda,s) f$)
\begin{equation}\label{iterating}
\|Bu\|_{\mc{H}_h^{NG}}\leq (\sum_{j=0}^{k-1}
|s|^{-j-1}\|B(R_{\psi_\sigma}^0(\lambda))^jf\|_{\mc{H}_h^{NG}}+|s|^{-k}\|B(R_{\psi_\sigma}^0(\lambda))^ku\|_{\mc{H}_h^{NG}})+\mc{O}(h^\infty\|f\|_{\mc{H}_h^{NG}}).
\end{equation}
 By Egorov lemma (just as above), locally uniformly in $(\lambda,s,\sigma)$
\[\WFh(B(R_{\psi_\sigma}^0(\lambda))^jf)\subset \WFh(B)\cap \bigcup_{A\in j\mc{O}} e^{X_A^H}(\WFh(f)).\]
Again by Egorov, we have
$B(R_{\psi_\sigma}^0(\lambda))^k-B(R_{\psi_\sigma}^0(\lambda))^kB_1\in h^\infty \Psi_h^\infty(\M)$ thus by \eqref{radialsource}
\[ \|BT_{\psi_\sigma}^Q(\lambda,s) f\|_{\mc{H}_h^{NG}}\leq C(\sum_{j=0}^{k-1}
|s|^{-j-1}\|B(R_{\psi_\sigma}^0(\lambda))^jf\|_{\mc{H}_h^{NG}}+|s|^{-k}\|B_1f\|_{\mc{H}_h^{NG}})+\mc{O}(h^\infty\|f\|_{\mc{H}_h^{NG}}).\]
with $C>0$ and the remainder being locally uniform in $(\lambda, s,\sigma)$.
Applying this with $B'f$ instead of $f$ and using $B_1B'\in h^\infty\Psi_h^{-\infty}(\mc{M})$ and
$B(R_{\psi_\sigma}^0(\lambda))^jB'\in h^\infty\Psi_h^{-\infty}(\mc{M})$ for $j\in [0,k]$ under our assumptions on $B,B'$, we obtain
\eqref{esttoprove} uniformly in $(\lambda, s,\sigma)\in K$.\\

\textbf{(iii) Estimate near $E_u^*$.} As we already have the estimate \eqref{eq:WFTQ} on $\WF(\mc K_{T^Q_{\psi_\sigma}})$ there remains to study the case where $\WFh(B)$ and $\WFh(B')$ are intersecting $E_u^*$. By Egorov and the fact that $u:=T_{\psi_\sigma}^Q(\lambda,s)B'f$ has $\WFh(u)$ contained in $\{|\xi|\geq r'\}$ (which can be read off \eqref{eq:TQpsi}), we obtain that
$\WFh(R_{\psi_\sigma}^0(\lambda)u)\subset \cup_{A\in\mc{O}}e^{X_A^H}(\WFh(u))$ is contained
in $\{|\xi|\geq r'/C_0\}$ for some $C_0>0$.
We can assume that $\WFh(B)$ is a small neighborhood of a point $(x,\xi)\in E_u^*\cap \{|\xi|\geq r'\}$, so that for $k\geq 2$ large enough $e^{-X_A^H}(\WFh(B))\in \{|\xi|\leq r'/2C_0\}$ for all $A\in (k-1)\mc{O}$.  We then obtain by Egorov theorem  that $\|B(R_{\psi_\sigma}^0(\lambda))^{k-1}R_{\psi_\sigma}^0(\lambda)u\|_{\mc{H}_h^{NG}}=\mc{O}(h^\infty)$ uniformly in $(\lambda,s,\sigma)\in K$.
The estimate \eqref{iterating} still holds and we deduce that uniformly for $(\lambda,s,\sigma)\in K$
\[\|BT_{\psi_\sigma}^Q(\lambda,s)B'f\|_{\mc{H}_h^{NG}}\leq \sum_{j=0}^{k-1}
|s|^{-j-1}\|B(R_{\psi_\sigma}^0(\lambda))^jB'f\|_{\mc{H}_h^{NG}}+\mc{O}(h^\infty\|f\|_{\mc{H}_h^{NG}}).\]
Assuming that $\WFh(B')\cap \Omega_B=\emptyset$, the right hand side is an $\mc{O}(h^\infty)$ uniformly for $(\lambda,s,\sigma)\in K$ by Egorov theorem again. We conclude that \eqref{WFtoprove} holds.

\subsection{Proof of Theorem \ref{thm:formuledebowenintro}}

Since our analysis is based on the use of the operators $R_\psi(\lambda)$, which are a kind of Laplace transform, it will be convenient to introduce the notation
\[
\hat{\psi}(\lambda) : = \int_{\a} e^{\lambda(A)} \psi(A)dA, \ \lambda\in \a_\C^*.
\]
(formally, $R_\psi(\lambda) = \hat{\psi}(X+\lambda)$). Using the analytic continuation of the function $Z_{f,\psi}$, we can obtain the following result
\begin{prop}\label{prop:raw-convergence-result}
Let $f\in C^\infty(\mc{M})$, and let $\psi\in C_c^\infty(\W;\R^+)$ with $\int_{\W}\psi =1$ and support in a small ball. For real $\lambda \in \a^*\subset \a^*_\C$,
we denote $\delta = \log \hat{\psi}(\lambda)$. We have for some $\epsilon>0$ that as $k\to \infty$
\[
\sum_{T\in \mathcal{T}}\sum_{A\in \W\cap L(T)} \psi^{*k}(A)e^{-\lambda(A)} \frac{ \int_T f\,d\lambda_T }{|\det (1-\mathcal{P}_A)|} = \mu(f)  e^{\delta k} + \mathcal{O}(e^{(\delta-\epsilon)k}).
\]
Additionally, if $\psi = \psi_\sigma$, the bound on the remainder is locally uniform with respect to $\sigma$.
\end{prop}

\begin{proof}
By Proposition \ref{Zorbitesperiod}, the terms we want to estimate are the coefficients $c_k$ in the expansion
\[
Z_{f,\psi}(s,\lambda) = \sum_{k\geq 0} s^{-k} c_k.
\]
Since $Z_{f,\psi}$ is meromorphic in the $s$ variable in $\C^\ast$, according to Cauchy's formula, we have for every $\rho\in \R^+$ such that $Z_{f,\psi}(\cdot, \lambda)$ has no poles on the circle of radius $\rho$
\[
c_k = \frac{1}{2i\pi} \int_{\rho \mathbb{S}^1} Z_{f,\psi}(s,\lambda)s^{k-1} ds + \sum_{\text{poles of modulus }> \rho} \Res( Z_{f,\psi}(s,\lambda)s^{k-1} ds).
\]
Let us assume that $Z_{f,\psi}(\cdot,\lambda)$ has a simple pole $s_0$ with modulus strictly larger than all the other poles, and $s_0$ is real. We denote by $K$ the residue of $Z_{f,\psi}(\cdot,\lambda)$ at $s_0$, and we find for some $\epsilon>0$
\[
c_k = K s_0^{k-1} + \mathcal{O}( (e^{-\epsilon}s_0)^{k-1}).
\]
For this reason, we will be done if we can prove that $e^\delta$ is indeed a real dominating pole of order $1$ (when $\lambda$ is real), with residue
\[
K = \mu(f) e^\delta.
\]
Let us describe the poles of $Z_{f,\psi}(\lambda,\cdot)$. As we have seen before, they are exactly the poles of the resolvent $(s-R_\psi(\lambda))^{-1}$. Let us investigate the structure of these poles. Let $s_0\in \C^\ast$ be a pole of $s\mapsto (s-R_\psi(\lambda))^{-1}$, and denote by $\Pi(\lambda,s_0)$ the corresponding spectral projector, so that near $s_0$,
\[
(s-R_\psi(\lambda))^{-1} = \sum_{j\geq 0} \frac{ (R_\psi(\lambda)-s_0)^j \Pi(\lambda,s_0) }{(s - s_0)^{j+1}} + \text{ holomorphic }.
\]
Since $s_0-R_\psi(\lambda)$ is Fredholm on some suitable anisotropic space, there is a finite dimensional 
generalized eigenspace $E(s_0)$ such that $(s_0-R_\psi(\lambda))^{N_0}=0$ on $E(s_0)$ for some $N_0\geq 1$, and since $R_\psi(\lambda)$ commutes with the Anosov action, we can split $E(s_0)$ into a sum of
joint generalized eigenspaces for the action. We can thus find in $E(s_0)$ a non-zero vector $u$ such that $-X_Au = \la_0(A) u$ for some $\la_0 \in \a_\C^*$ and all $A\in \W$. Since the generalized eigenspace is contained in some suitable anisotropic space, we deduce that $\la_0 \in \Res(X)$ where ${\rm Res}(X)$ is the set of Ruelle-Taylor resonance of \cite{BGHW20} (and introduced in \eqref{def:RTresonance}). It follows that
\[
R_\psi(\lambda) u = \int_{\W} e^{-X_A-\lambda(A)} \psi(A) u \, dA = \hat{\psi}(\lambda-\la_0) u,
\]
so that $s_0 = \hat{\psi}(\lambda-\la_0)$. The converse argument completes the proof of the fact that $\ran \Pi(\lambda, s_0)$ is exactly equal to
\[
\left\{ u\ \middle|\ \exists \la_0 \in \Res(X),\ \hat{\psi}(\lambda-\la_0) = s_0,\ \text{ $u$ is a generalized resonant state at $\la_0$} \right\}.
\]

Now, we know from \cite[Theorem 2]{BGHW20} that for all resonances $\la_0$, and all $A\in \W$, $\Re (\la_0(A))\leq 0$, with equality for $\la_0 = 0$ (if the action is not mixing, there may be other purely imaginary resonances). Now, since we have chosen $\psi$ to be real non-negative, we have
\[
\left|\int_{\W} e^{\la_0(A)} \psi(A) dA\right| \leq \int_{\W} e^{\Re (\la_0(A))} \psi(A) dt.
\]
With equality if and only if $e^{\la_0(A) - \Re (\la_0(A))}$ is constant on the support of $\psi$.
It follows that for $\lambda$ real, and $\la_0 \in \Res(X)$ with $\la_0\neq 0$,
\[
\left|\hat\psi(\lambda-\lambda_0)\right| = \left| \int_{\W} e^{(\la_0-\lambda)(A)} \psi(A) dA \right| < \int_{\W} e^{-\lambda(A)}\psi(A) dA = \hat \psi(\lambda).
\]
This implies that $e^{\delta}=\hat{\psi}(\lambda)$ is indeed a dominating pole. It remains to compute its order and residue. However from \cite[Proposition 5.4]{BGHW20}, we know that there are no Jordan blocks at $0$, so that the $0$-characteristic space is equal to the $0$-eigenspace, and that the corresponding spectral projector $\Pi(\lambda, e^\delta)$ does not depend on the choices -- it was denoted $\Pi(0)$ in \cite{BGHW20}. We deduce that indeed $e^\delta$ is a simple pole, with residue
\[
\Tr  f R_\psi(\lambda) \Pi(0)  = e^\delta \mu(f). \qedhere
\]
\end{proof}

As a corollary of the proof, we find
\begin{cor}\label{cor:poles-Z_1}
The function $Z_{1,\psi}(\cdot, \lambda)$ has simple poles, and
\[
\Res(Z_{1,\psi}(s,\lambda)ds, s_0) = s_0\sum_{\zeta\in\Res(X),\hat{\psi}(\lambda-\zeta) = s_0} \dim\{u\in C^{-\infty}_{E^u_\ast} \ |\ \exists \ell>0,\ (-X-\zeta)^\ell u = 0\ \}.
\]
\end{cor}

\begin{proof}
Near a pole $s_0$ of $(s-R_\psi(\la))^{-1}$, we come back to the formula
\[
R_\psi(\lambda)(s-R_\psi(\lambda))^{-1} = \sum_{j\geq 0} \frac{ R_\psi(\lambda)(R_\psi(\lambda)-s_0)^j \Pi_{s_0}}{(s-s_0)^{j+1}} + \text{ holomorphic around $s_0$.}
\]
Since the trace of nilpotent operators is always $0$, taking the trace eliminates the terms $j>0$, and we deduce that around $s_0$,
\[
Z_{1,\psi}(s,\lambda) = \frac{s_0}{s-s_0} \Tr \Pi_{s_0} + \text{ holomorphic around $s_0$.}
\]
We have already seen that the range of $\Pi_{s_0}$ is the direct sum of the spaces
\[
\{u\in C^{-\infty}_{E^u_\ast} \ |\ \exists \ell>0,\ (-X-\zeta)^\ell u = 0\ \},
\]
for $\hat{\psi}(\lambda-\zeta) = s_0$.
\end{proof}

In the case that $f\neq 1$, the singularities could become more complicated because $f$ could interact with the Jordan blocks of $R$. We can now turn to the proof of Theorem \ref{thm:formuledebowenintro}:
\begin{proof}[Proof of Theorem \ref{thm:formuledebowenintro}] 
We want to prove the convergence for any $f\in C^0(\mc M)$, but it suffices to prove the statement for $f\geq 0$ and $f\in C^\infty(\mc M)$ (first we can write $f = f_1-f_2+if_3-if_4$ with positive continuouse functions $f_i$ and then approximate each of the function by a positive smooth function in $\|\cdot \|_\infty$ norm). We will use some estimates from the proof of \cite[Proposition 5.4]{BGHW20}.
Let $\nu\in \mc{D}'(\W)$ be the measure\footnote{$\mc{D}'(\mc{W})$ denotes the space of distributions on the open cone $\mc{W}$.}
\[\nu=\sum_{T\in \mc{T}}
\sum_{A\in \W\cap L(T)}\frac{\int_{T}f\,d\lambda_T}{|\det(1-\mc{P}_{A})|}\delta_A\]
where $\delta_A$ is the Dirac mass at $A$.
Choose a basis $(A_j)_{j=1}^\kappa$ of $\a$ so that $A_j\in \ker \eta$ for $j\geq 2$ and $\eta(A_1)=1$. We then identify $\a\simeq \R^\kappa$ by identifying the canonical basis of $\R^\kappa$ with the basis $(A_j)_j$. We let $\mc{C}$ be a proper subcone of the Weyl chamber $\mc{W}$ and 
$\Sigma=\mc{C}\cap \{A_1+\sum_{j=2}^\kappa t_jA_j, |\, t_j\in \R\}$ be a hyperplane section of the cone $\mc{C}$.
Choose $r>0$ smaller than the distance of $\Sigma$ to the boundary $\pl \mc{W}$ of the Weyl chamber (where the distance is the Euclidean distance in the chosen coordinates). Next choose $\psi\in C_c^\infty((-r/2,r/2))$ to be non-negative and even with $\int_\R\psi=1$, and for each $\sigma\in \R^\kappa$,
define $\psi_\sigma(t):=\prod_{j=1}^\kappa\psi(t_j-\sigma_j)$. Our assumption implies that $\psi_\sigma$ is supported in $\mc{W}$.
We view $\Sigma$ as an open subset of $\{(1,\bar{t})\,|\, \bar{t}\in \R^{\kappa-1}\}$ and choose $q\in C_c^\infty(\Sigma;\R^+)$ with small support and let $Q=\int_{\R^{\kappa-1}}q(\bar{t})d\bar{t}>0$, and $\omega\in C_c^\infty((0,1);[0,1])$ with $W:=\int_0^1 \omega>0$. We introduced $q$ in order to apply Proposition \ref{prop:raw-convergence-result}, that holds only locally uniformly with respect to $\sigma$.
We consider $\lim_{N\to \infty} \nu(F_N)$ for $\sigma(\theta):=(1,\theta)\in \Sigma$ where
\[
F_N(t):=\frac{1}{N}\sum_{k=1}^{N}\int_{\R^{\kappa-1}}\omega(\tfrac{k}{N}) \psi_{\sigma(\theta)}^{*k}(t)q(\theta)d\theta=\frac{1}{N}\sum_{k\in [\eps N,(1-\eps)N]}\int_{\R^{\kappa-1}}\omega(\tfrac{k}{N}) \psi_{\sigma(\theta)}^{*k}(t)q(\theta)d\theta.
\]
for some $\eps>0$ (we use the compact support property of $\omega$).
\begin{figure}
\includegraphics[scale=0.22]{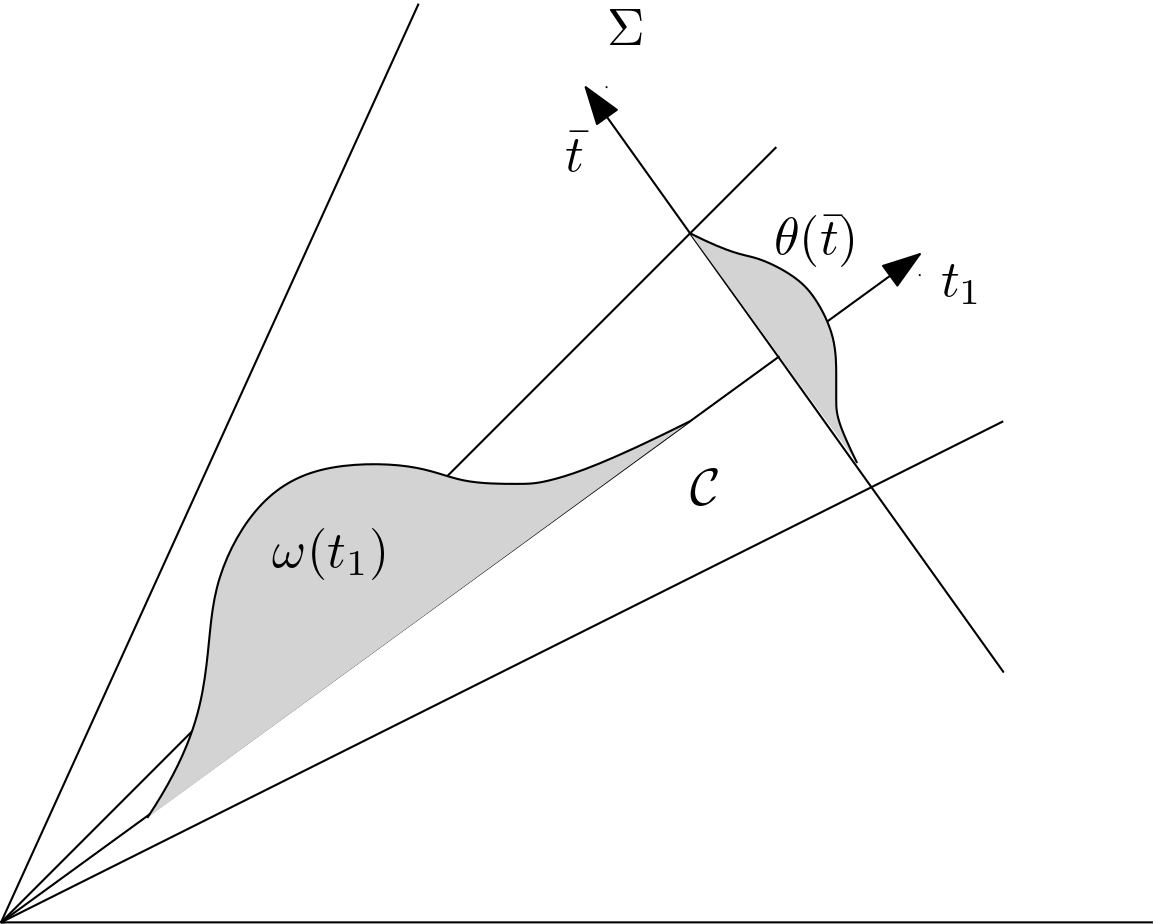}
\caption{Support of the cutoff functions $\omega$ and $\theta$.}
\label{fig:cutoff}
\end{figure}
By applying Proposition \ref{prop:raw-convergence-result} with $\lambda = 0$, we have
\[\begin{split}
F_N(t):=&\frac{1}{N}\sum_{k=1}^{N}(1+\mc{O}(e^{-\eps k}))\omega(\tfrac{k}{N})\int_{\R^{\kappa-1}}q(\theta)d\theta
\end{split}\]
and the sum over $k$ converges as $N\to \infty$ to $\int_0^1\omega$ (Riemann sum), thus
\begin{equation}\label{eq:nuFNlim}
WQ\mu(f)=\lim_{N\to \infty}\nu(F_N).
\end{equation}
In the proof of \cite[Proposition 5.4]{BGHW20}, it is shown that if $h(t):=t_1^{1-\kappa}\omega(t_1)q(\bar{t}/t_1)$ with $t=(t_1,\bar{t})$, then with $G_N(t):=N^\kappa F_N(tN)$
\[ \|G_N(t)- h(t)\|_{L^2(\R^\kappa,dt)}\to 0 \textrm{ as }N\to \infty.\]
The proof of this convergence is based on the following estimates on the Fourier transform $\hat{G}_N$ of $G_N$: for $\delta>0$ small, and all $\ell\in\N$ there is $C_\ell>0,C>0,c_0>0$ so that
\begin{align} 
& \forall \xi, |\xi|>N^{1/2+\delta},\quad  |\hat{G}_N(\xi)|^2 \leq C(1+c_0\big|\frac{\xi}{N}\big|^2)^{-2\eps N},\label{bound1onGN}
\\
& \forall \xi, |\xi|\leq N^{1/2+\delta}, \quad |\hat{G}_N(\xi)|^2 \leq
C_\ell(|\xi|^{-\ell}+N^{-\ell(\frac{1}{2}-\delta)})\label{bound2onGN}
\end{align}
and $G_N(\xi)\to \hat{h}(\xi)$ for each $\xi\in \R^\kappa$.  The estimate above for $|\xi|>N^{1/2+\delta}$ 
follows from the bound $|\hat{\psi}(\xi)|\leq (1+c_0|\xi|^2)^{-1}$ for some $c_0>0$ (proved in \cite[Eq. (5.12)]{BGHW20})
and 
\[ |\hat{G}_N(\xi)|=|\hat{F}_N(\frac{\xi}{N})|\leq \frac{1}{N}\sum_{k\in [\eps,N,(1-\eps)N}|\hat{\psi}(\frac{\xi}{N})|^{k}\leq |\hat{\psi}(\frac{\xi}{N})|^{\eps N}\]
 while the bound for $|\xi|\leq N^{1/2+\delta}$ is proved in \cite[Eq. (5.13)]{BGHW20}.
As a consequence, for each $\ell\gg \kappa$ one can take $N_0$ large enough so that for all $N>N_0$
\[ \begin{split}
\int_{|\xi|>N^{1/2+\delta}} \cjg \xi\cjd^{2\ell} |\hat{G}_N(\xi)-\hat{h}(\xi)|^2d\xi\leq & C_\ell \int_{|\xi|>N^{1/2+\delta}}\cjg \xi\cjd^{-\ell}+C\big(1+c_0\big|\frac{\xi}{N}\big|^2\big)^{-2\eps N} d\xi\\
\leq & C_\ell \Big(N^{-\ell(\frac{1}{2}+\delta)+\kappa}+N^\kappa \int_{|\xi|>N^{-\frac{1}{2}+\delta}}C\big(1+c_0|\xi|^2\big)^{-2\eps N}d\xi\Big)\\
\leq & C_\ell \Big(N^{-\ell(\frac{1}{2}+\delta)+\kappa}+N^\kappa e^{-c_0\eps N^{2\delta}}\Big)
\end{split}\]
The bound follows from the fact that $h\in C_c^\infty(\R^n)$, thus its Fourier transform is Schwartz, and the bound \eqref{bound1onGN}. Now, using \eqref{bound2onGN} (with $3\ell$ instead of $\ell$) and the fact that $\hat{h}$ is Schwartz, we also have
\[ {\bf 1}_{|\xi|\leq N^{1/2+\delta}} \cjg \xi\cjd^{2\ell} |\hat{G}_N(\xi)-\hat{h}(\xi)|^2\leq
C_\ell \cjg \xi\cjd^{-\ell} \]
where $C_\ell>0$ is independent of $N$. 
Using Lebesgue dominated convergence theorem, we conclude that $\|\cjg \xi\cjd^\ell (\hat{G}_N(\xi)-\hat{h}(\xi))\|_{L^2}\to 0$ as $N\to \infty$, thus by Sobolev embedding $\|G_N-h\|_{C^0}\to 0$ as $N\to\infty$. Therefore
\begin{equation}\label{FNh}
\|F_N(t)-N^{-\kappa}h(t/N)\|_{C^0}=o(N^{-\kappa}).
\end{equation}
Choosing $q=1$ on an open set $U\subset \Sigma$ and $\omega(t_1)=t_1^{\kappa - 1}$ in $(\eps,1-\eps)$, we deduce that  for $N$ large enough
\[
F_N(t) \geq \frac{1}{2N^\kappa}{\bf 1}_{[\eps N,N(1-\eps)]}(t_1){\bf 1}_{U}(\bar{t}/t_1).
\]
This implies that if $\mc{C}_{a,b}(U):=\{(t_1,\bar{t})\,|\, t_1\in [a,b], \bar{t}/t_1\in U\}$
\[
\nu(\mc{C}_{\eps N,N(1-\eps)}(U))\leq 2N^\kappa \nu(F_N) \leq 3N^\kappa QW\mu(f)
\]
where we used \eqref{eq:nuFNlim}. Thus we obtain $\nu(\mc{C}_{0,N}(U))=\mc{O}(N^\kappa)$ by letting $\eps\to 0$. This estimate thus also implies by a covering argument that $\nu(\mc{C}\cap \{|A|<N\})=\mc{O}(N^\kappa)$.
Coming back to a general $\omega,q$, using that
$\supp(F_N)\cup \supp(h(\cdot/N))\subset \mc{C}_{0,3N}(U')$ for some open set $U'\subset \W\cap \{t_1=1\}$, and since $\nu(\mc{C}_{0,3N}(U'))=\mc{O}(N^{\kappa})$, we obtain by \eqref{FNh}
\[\lim_{N\to \infty}N^{-\kappa}\nu(h(\cdot/N))=\lim_{N\to \infty} \nu(F_N)=WQ\mu(f).\]
Next, let $U$ be a small open ball in $\Sigma$, we can choose $q_j\in C_c^\infty(\Sigma)$ supported near $U$ for $j=1,2$ so that $q_1\leq \textbf{1}_U\leq q_2$ and $\int q_j=|U|+\mc{O}(\eps)$, and for $0<a<b<1$, choose
$\omega_j\in C_c^\infty((0,1),[0,1])$ such that
$\omega_1\leq t^{\kappa-1}\textbf{1}_{[a,b]}\leq \omega_2$ with $\int_0^1\omega_j(t)dt=\int_a^{b}t^{\kappa-1}dt+\mc{O}(\eps)$. Write now $h_j(t)=t_1^{1-\kappa}\omega_j(t_1)q_j(\bar{t}/t_1)$.
One then has
\[
N^{-\kappa}\nu(h_1(\cdot/N))\leq N^{-\kappa}\nu(\mc{C}_{\delta N,(1-\delta) N}(U))\leq N^{-\kappa}\nu(h_2(\cdot/N))
\]
thus if $V_{a,b}=\int_a^{b}t^{\kappa-1}dt$, we obtain for each $\eps>0$ small
\[
\begin{gathered}
(V_{a,b}-\eps)(|U|-\eps)\mu(f)\leq \liminf_{N\to \infty}
N^{-\kappa}\nu(\mc{C}_{a N,b N}(U))\\
\limsup_{N\to \infty}N^{-\kappa}\nu(\mc{C}_{a N,b N}(U))\leq (V_{a,b}+\eps)(|U|+\eps)\mu(f).
\end{gathered}
\]
We let $\eps\to 0$ and deduce that
\[\lim_{N\to\infty}
N^{-\kappa}\nu(\mc{C}_{a N,b N}(U))=\mu(f)V_{a,b}|U|=\mu(f) |\mc{C}_{a,b}(U)|.\]
Since $N^\kappa |\mc{C}_{a,b}(U)|=|\mc{C}_{aN,bN}(U)|$ and $|\mc{C}_{aN,bN}(U)|^{-1} \nu(\mc{C}_{a N,b N}(U))$ is the RHS of \eqref{eq:bowenformula}, we have proved the desired result in the case of $\mc{C}$ being a cone with section $U$, i.e. a small proper subcone of $\mc{W}$. By compactness (of the cone section), any 
proper subcone of $\mc{W}$ can be decomposed into a finite union of small cones and applying the result for small cones we obtain the general result.
\end{proof}

\subsection{A dynamical zeta function}

We conclude by some comments on links with dynamical zeta functions. In \cite{BT08}, for an operator $K$, Baladi and Tsujii introduce the flat determinant as
\[
{\det}^\flat(1 - z K) := \exp\left( - \sum_{\ell\geq 0} \frac{z^\ell}{\ell} \Tr^\flat K^\ell \right),
\]
provided the RHS makes sense and converges. From the argument of the previous section, we see that $\det^\flat(1- z R_\psi(\lambda))$ is well defined for $|z|$ small enough, and has a holomorphic extension to $z\in \C$. Indeed for $|z|$ small enough,
\[
\frac{ \partial_z \det^\flat(1- z R_\psi(\lambda) )}{\det^\flat ( 1- z R_\psi(\lambda))} = - \frac{1}{z} \sum_{\ell \geq 1} z^\ell \Tr^\flat R_\psi(\lambda)^{\ell} = - \frac{1}{z} Z_{1,\psi}\left(\frac{1}{z}, \lambda\right).
\]
According to Proposition \ref{thm:Zf}, the RHS here is meromorphic for $z$ close to $1$ and $\lambda\in \mathfrak{a}_\C^\ast$. It remains to show it has simple poles, with integer residue, in order to deduce that $\det^\flat(1-zR_\psi(\lambda))$ is holomorphic in $(z,\lambda)$. According to Corollary \ref{cor:poles-Z_1}, near a pole $1/s_0$, $Z_{1,\psi}$ takes the form
\[
-\frac{1}{z} Z_{1,\psi}\left(\frac{1}{z},\lambda\right) = - \frac{ s_0 m }{z( 1/z  - s_0 )} + \text{ holomorphic} =  \frac{m}{z - 1/s_0} + \text{ holomorphic}.
\]
Here $m$ is the dimension of some generalized eigenspace, i.e an integer. This proves that $(\lambda,z)\mapsto \det^\flat(1-zR_\psi(\lambda))$ is holomorphic in $\a_\C^*\times B_\C(1,1/4)$. The parameter $z$ here is auxiliary, so we can fix its value to $1$, and obtain:
\begin{theorem}\label{thm:zeta-function}
Let $X$ be an Anosov action of $\R^\kappa$ with positive Weyl chamber $\W$, and let $\psi\in C^\infty_c(\W,\R^+)$ have $\int \psi =1$ and small enough support. Then
\[
d_\psi(\lambda) := \exp \left( - \sum_{T\in\T}\sum_{A\in L(T)} \frac{\mathrm{vol}(T)}{|\det(1-\mathcal{P}_A)|}e^{-\lambda(A)}\sum_k \frac{\psi^{*k}(A) }{k} \right),
\]
originally defined for $\Re \lambda$ large enough (in the sense of evaluating it in elements of the positive Weyl chamber $\W$), has a holomorphic continuation to $\a_\C^*$. This continuation is a regularized version of the formal product
\[
\prod_{\zeta \in \Res(-X)} (1- \hat{\psi}(\lambda - \zeta) )^{\textup{multiplicity of }\zeta}.
\]
This means that for each $\lambda_0 \in \C^\kappa$, there exist only a finite number of $\zeta$'s in $\Res(X)$ such that $|\hat{\psi}(\lambda_0-\zeta)-1| < 1/2$, and for $\lambda$ close enough to $\lambda_0$,
\[
d_\psi(\lambda) = \prod_{\zeta \in \Res(X)} (1- \hat{\psi}(\lambda - \zeta) )^{\textup{multiplicity of $\zeta$}} \times \textup{ a holomorphic function}.
\]
\end{theorem}

As far as we know, this is the first appearance of a multi-parameter zeta function for actions with a global meromorphic extension. The reader accustomed to the rank 1 case may find the formula a bit surprising and wonder if it is possible to replace the term $F:=\sum \psi^{*k}/k$ by a simpler weight function. Let us first observe that formally, in the rank 1 case, to obtain $F=\mathbf{1}_{\R^+}$, we need to take formally $\hat\psi= 1- e^{-1/s}$, that is $\psi= J_1(2\sqrt{t})/\sqrt{t}$. This is certainly smooth enough for our arguments, but the support assumption is not verified. Let us explain briefly why this is not an easy question. The natural extension of the rank 1 case would be to replace it by
\[
\tilde{F} := {\bf 1}_{\{\lambda_j > 0,\ j= 1\dots \kappa\}},
\]
having taken a basis comprised of elements of $\W$ and $\la=(\la_1,\dots,\la_\kappa)$ in this basis.
However, there seems to be no hope that this can define a globally meromorphic function. Indeed, let us consider the singular set of such a function. Each resonance $\zeta$ would contribute by
\[
\left( \prod \frac{1}{\lambda_j - \zeta_j}\right)^{\text{multiplicity of $\zeta$}}.
\]
For the resulting product to be meromorphic, the singular set has to be (at least) locally closed. It is given by
\[
\left\{ \lambda\in \C^\kappa\ |\ \exists \zeta\in\Res(X),\ \exists j=1\dots \kappa, \lambda_j = \zeta_j \right\}.
\]
For this to be locally closed, we need that for any sequence of resonances $\zeta^\ell$, each coordinate $(\zeta_j^\ell)_\ell$ tends to infinity. For example for the resonances with small real part, this means that the imaginary parts cannot equidistribute in $\R^\kappa$. This would certainly be a surprise to us, in particular as for Weyl chamber flows a Weyl-lower bound on the number of Ruelle Taylor resonances with $\Re(\lambda)=-\rho$ is known \cite[Theorem 1.1]{HWW21}.

\section{Application to lattice point counting}\label{sec:counting}
In this final section we will work out the consequences of the Bowen formula for the SRB measures in terms of the counting-problem of lattice points. We focus in particular on the case of Weyl chamber flows, where we obtain precise estimates for the exponential growth rates.

Choose a proper subcone $\mc C\subset \mc W$ and fix some $\xi\in \a^*$ such that $\xi$ is positive on a small conical neighbourhood of $\mathcal{C}$.
For $0\leq a<b$ we define $\mc C_{a,b}:= \{A\in\mc C,\xi(A) \in [a,b]\}$ and the lattice point counting function
\[
 \mathcal N_{\mc C_{a,b}}:= \sum_{T\in\mc T}\sum_{A\in L(T)\cap \mc C_{a,b}} \textup{vol}(T).
\]

We furthermore introduce a counting function that is additionally weighted by the Jacobians $|\det(1-\mc P_A)|$
\[
 \mathcal N^w_{\mc C_{a,b}}:= \sum_{T\in\mc T}\sum_{A\in L(T)\cap \mc C_{a,b}} \frac{\textup{vol}(T)}{|\det(1-\mc P_A)|}.
\]
Note that by Theorem~\ref{thm:formuledebowenintro} choosing the constant test function $f=1$ we get for any $q>1$
\begin{equation}\label{eq:counting_lim}
 \lim_{n\to\infty} \frac{\mathcal N^w_{\mc C_{q^{n-1},q^{n}}}}{q^{\kappa (n-1)}} = |\mc C_{1,q}|.
\end{equation}
Let us define
\begin{align*}
 m:=&\liminf_{n\to\infty} \frac{\log\left( \inf_{A\in\mc C_{q^{n-1},q^{n}}\cap(\cup_{T\in\mc T} L(T))}|\det (1-\mc P_A)|\right)}{q^n},\\
 M:=&\limsup_{n\to\infty} \frac{\log\left( \sup_{A\in\mc C_{q^{n-1},q^{n}}\cap(\cup_{T\in\mc T} L(T))}|\det (1-\mc P_A)|\right)}{q^n}.\\
\end{align*}
Then for any $\varepsilon>0$ there is $N$ such that for $n>N$ we have
\[
 e^{(m-\varepsilon)q^n}\mathcal N^w_{\mc C_{q^{n-1},q^{n}}} \leq \mathcal N_{\mc C_{q^{n-1},q^{n}}} \leq \mathcal N^w_{\mc C_{q^{n-1},q^{n}}} e^{(M+\varepsilon)q^n}
\]
and taking additionally \eqref{eq:counting_lim} into account we get
\[
 (|\mc C_{1,q}|-\varepsilon)q^{\kappa (n-1)} e^{(m-\varepsilon)q^n} \leq \mathcal N_{\mc C_{q^{n-1},q^{n}}} \leq (|\mc C_{1,q}|+\varepsilon)q^{\kappa (n-1)} e^{(M+\varepsilon)q^n}.
\]
Now, using $\mathcal N_{\mc C_{0,q^n}} = \mathcal N_{\mc C_{0,1}} + \sum_{k=1}^{n} \mathcal N_{\mc C_{q^{k-1},q^{k}}}$ we deduce
\[
 m \leq  \liminf_{n\to\infty} \frac{\log \mc N_{\mc C_{0,q^n}}}{q^n}\leq \limsup_{n\to\infty} \frac{\log \mc N_{\mc C_{0,q^n}}}{q^n}\leq M.
\]
Now let us assume that there is $\eta\in \a^*$ positive on $\W$ such that for any proper subcone $\mc C\subset \W$ there is $\varepsilon>0$ such that $|\det(1-\mc P_A)| = e^{\eta(A)}(1-\mc{O}(e^{-\varepsilon|A|}))|$ for all $A\in \C$.  If we now fix $\eta/\|\eta\|$ then we get the particularly simple expressions $M = \|\eta\|$ and $m=\|\eta\|/q$. As we can choose $q>1$ arbitrary close to $1$ we get the following result, which is precisely  Corollary \ref{cor:comptage}: 
\begin{prop}\label{countingprop}
For an Anosov action for which there is $\eta\in \a^*$ with the above properties, one has for each proper subcone
$\mc C\subset \mc W$
\[
\lim_{R\to\infty}\frac{\log \mc N_{\mc C_{0,R}}}{R} =\|\eta\|.
\]
\end{prop}
The assumption on $|\det(1-\mc P_A)|$ is fulfilled for all standard Anosov actions. For example for Weyl chamber flows, a periodic point $x_0 =\Gamma g_0M \in \Gamma\backslash G/M =\M$ under $A_0\in \W$ implies the existence of $\gamma_0\in \Gamma, m_0\in M$ such that $\gamma_0g_0\exp(A_0)m_0 = g_0$. With these notations we can give an explicit expression of
$|\det(1-\mc P_{A_0})|$
\[\begin{split}
|\det(1-\mc P_{A_0})| &=\prod_{\alpha\in\Delta_+}\left|\det_{\mathfrak{g}_\alpha}\left(1-e^{-\alpha(A_0)}\tu{Ad}(m_0^{-1})\right)\right| \left|\det_{\mathfrak{g}_{-\alpha}}\left(1-e^{\alpha(A_0)}\tu{Ad}(m_0^{-1})\right)\right|\\
 &=e^{2\rho(A_0)}\prod_{\alpha\in\Delta_+}\left|\det_{\mathfrak{g}_\alpha}\left(1-e^{-\alpha(A_0)}\tu{Ad}(m_0^{-1})\right)\right| \left|\det_{\mathfrak{g}_{-\alpha}}\left(\tu{Ad}(m_0^{-1})-e^{-\alpha(A_0)}\right)\right|.\\
 \end{split}
\]
Here $\Delta_+\subset \a^*$ denotes the set of positive roots, $\mathfrak g_{\pm \alpha}$ the corresponding root spaces, $m_\alpha:=\dim \mathfrak g_\alpha$ and we use the usual notation $\rho:= \sum_{\alpha\in\Delta_+}\frac{m_\alpha}{2}\alpha \in \a^*$ for the half sum of positive roots. As the adjoint action of $M$ on $\mathfrak g_\alpha$ is orthogonal and for any proper subcone $\mc C\subset \W$, $\alpha(A)>\varepsilon|A|$ for all $A\in \mc C$ we deduce by the continuity of the determinant
\begin{equation}\label{det1-P_A}
|\det(1-\mc P_A)|=e^{2\rho(A)}(1+\mc O(e^{-\varepsilon |A|}).
\end{equation}

\appendix

\section{Upper bound on the number of periodic orbits}
\label{appendix:bound-number-closed-orbits}

\begin{lemma}\label{estimate_periodic}
Let $n=\dim \M$, let $dA$ be the Haar measure on $\a$, let $\mc{P}:=\cup_{T\in \mc{T}}L(T)$ and $M:=\sup_{A\in\W,|A|=1}\|e^{X_A}\|_{\mc{L}(C^2(\M,\R))}$.

Fix a proper subcone $\mc{C}\subset \W$, then there is $C>0$ such that for all $\ell\geq 0$
\[
\begin{gathered}
\sharp \{ A\in \mc{P}\cap \mc C\ |\  \, |A|\leq \ell\}\leq C\ell^\kappa e^{(n-\kappa)M\ell},\\
\sharp \{ T\in \mc{T}\ |\ \exists A\in L(T) \cap \mc C\text{ with }|A|\leq \ell\}\leq C\ell^\kappa e^{(n-\kappa)M\ell}
\end{gathered}
\]
and for each $\eps>0$ and $\delta>0$,
\begin{equation} \label{volumeestimate}
(v_g\otimes dA) (\{ (x,A)\in \M\times \W\ |\  |A|\in (\delta,\ell), d_g(x,\varphi^{A}_1(x))<\eps\})\leq C\eps^ne^{nM\ell}.
\end{equation}
\end{lemma}

\begin{proof}
First, for $\delta_0>0$ small there is a family of invertible linear maps
$\mc{T}_{x,y}:T_{x}\M\to T_y\M$ depending continuously on $d_g(x,y)\leq \delta_0$
such that $\mc{T}_{x,x}={\rm Id}$ and $\mc{T}_{x,y}$ mapping $E_u(x),E_s(x)$,
and $E_0(x)$ onto $E_u(y),E_s(y)$ and $E_0(y)$.  Then exactly the same proof
as \cite[Lemma A.1]{DZ16a} shows that for each $r>0$ there is
$\delta \in (0,\delta_0)$ and $C>0$ such that if
$d_g(x,\varphi^A_1(x))<\delta$, $A\in\mc C$ with $|A|>r$ and
$v\in E_u(x)\oplus E_s(x)$, then
\begin{equation}\label{bounddphi}
|v|\leq C| (d\varphi^A_1-\mc{T}_{x,\varphi^A_1(x)})v|.
\end{equation}
This should be compared to the fact that for $A\in\mc C$, $\varphi^A_1(x)=x$ we
know that $(d_x\varphi^A_1-\mathrm{Id})_{|E_u(x)\oplus E_s(x)}$ is invertible with a uniform bound on the inverse as we restrict to a proper subcone $\mc C\subset \W$ that is bounded away from the walls of the Weyl chamber.
\eqref{bounddphi} generalizes this invertibility to orbits that are only approximately closed.

Next, we have, by the group property of $e^{X\cdot}$, that there is $C>0,M>0$ such that for all $A\in\W$, $x,x'\in \M$
\begin{equation}\label{boundC2}
\|e^{X_A}\|_{C^2\to C^2} \leq Ce^{M|A|}, \quad d_g(\varphi^A_1(x),\varphi^A_1(x'))\leq Ce^{M|A|}d_g(x,x').
\end{equation}
Next, we show a separation estimate between periodic tori.
\begin{lemma}\label{lem:separation}
Let $r>0$, then there is $C,\delta>0$ such that for all $\eps>0$ small, if $d_g(x,\varphi^A_1(x))\leq \eps$, $d_g(x',\varphi^{A'}_1(x'))\leq \eps$, $A,A'\in\W$ with $|A-A'|\leq \delta$, $d_g(x,x')\leq \delta e^{-M|A|}$, then $|A-A'|\leq C\eps,$ and furthermore there is $A''\in \a$ with $|A''|\leq 1$ such that $d_{g}(x,\varphi^{A''}_1(x'))\leq C\eps$.
\end{lemma}
Letting $\varepsilon\to 0$ we get, as a direct consequence, the following:
\begin{cor}\label{cor:separation}
 If two periodic orbits $\cup_{t\in [0,1]}\varphi^A_t(x)$ and $\cup_{t\in [0,1]}\varphi^{A'}_t(x')$ have minimal distance $\leq \delta e^{-M|A|}$ and nearby period $|A-A'|<\delta$, and $A,A'\in \mathcal{W}$, then $A=A'$ and there is an invariant torus orbit $T$ such that $x,x'\in T$.
\end{cor}

\begin{proof}[Proof of Lemma~\ref{lem:separation}] We follow closely the proof of \cite[Lemma A.2]{DZ16a}. Under our assumptions, $x,x',\varphi^A_1(x)$ and $\varphi^{A'}_1(x')$ are all in a small chart in $\R^n$ and we will frequently identify points in $\M$ and vectors in $T\M$ as elements in $\R^n$ via this chart.
The norm $|\bullet|$ then induces a metric that is equivalent to the Riemannian distance on $\M$.
We will furthermore assume that chart is chosen small enough such that for any $x,x'$ in the chart, the angle of $E_u(x)\oplus E_s(x)$ and $E_0(x')$ is bounded from below.
As $E_u(x)\oplus E_s(x)$ is a slice of the $\mathbb A$ action, there is $A''\in \a$ with $|A''|\leq 1$
such that $\varphi^{A''}_1(x')-x\in E_u(x)\oplus E_s(x)$. We write $x'':=\varphi^{A''}(x')$.
By the boundedness of the angles between $E_u\oplus E_s$ and $E_0$ there is a global $C$ such that $|x-x''|\leq C |x-x'|$.
Then by Taylor expansion there is $C>0$ such that
\[
\begin{gathered}
|\varphi^A_1(x'')-\varphi^A_1(x)-d\varphi^A_1(x)(x''-x)|\leq Ce^{M|A|}|x-x''|^2\leq C\delta |x-x''|,\\
|\varphi^{A'}_1(x'')-\varphi^A_1(x'')-X_{A'-A}(\varphi^A_1(x''))|\leq C|A-A'|^2\leq C\delta |A-A'|
\end{gathered}
\]
thus we obtain
\[
|\varphi^{A'}_1(x'')-\varphi^A_1(x)-d\varphi^A_1(x)(x''-x)-X_{A'-A}(\varphi^A_1(x''))|\leq C\delta(|x-x''|+|A-A'|).\]
Then, using $d_g(x,\varphi^A_1(x))\leq \eps$ and $d_g(x'',\varphi^{A'}_1(x''))\leq C\eps$ for some uniform $C>0$,
\[| (d\varphi^A_1(x)-{\rm Id})(x''-x)+X_{A'-A}(\varphi^A_1(x''))|\leq C\delta (|x-x''|+|A'-A|)+C\eps.
\]
Using that $\mc{T}_{x,y}$ is uniformly continuous in $x,y$, and
$d_g(x,\varphi^A_1(x))\leq \eps$ we get, if $\eps$ is chosen small enough (depending on $\delta$):
\[
|(d\varphi^A_1(x)-\mc{T}_{x,\varphi^A_1(x)})(x''-x)+X_{A'-A}(\varphi^A_1(x''))|\leq C\delta(|x''-x|+|A'-A|)+C\eps.
\]
Finally, using that $(d\varphi^A_1(x)-\mc{T}_{x,\varphi^A_1(x)})(x''-x)\in (E_u\oplus E_s)(\varphi^A_1(x))$, $X_{A'-A}(\varphi^A_1(x''))\in E_0\varphi^A_1(x'')$ together with the lower bound on the angle of these subspaces as well as \eqref{bounddphi}, we conclude that
\[
|A-A'|+|x-x''|\leq C(|d\varphi^A_1(x)-\mc{T}_{x,\varphi^A_1(x)})(x''-x)|+|A'-A|)\leq C\delta(|A'-A|+|x''-x|)+C\eps
\]
which gives the result by choosing $\delta$ small enough.
\end{proof}
We now prove \eqref{volumeestimate}. Let $\ell>0$ be large.
We take a maximal set of points $(x_j, A_j)_j$ in $\M\times \{A\in\W||A|\leq \ell \}$ so that $d_g(x_j,x_k)>\delta e^{-M\ell}/2$ or $|A_j-A_k|>\delta/2$. The number of such points is $\mc{O}(\ell^{\kappa} e^{nM\ell})$ and the polynomial term in $\ell$ can easily be absorbed (by changing $M$) such that we have $\mc{O}(e^{nM\ell})$ points.
One has
\[
\begin{gathered}
Z:=\{ (x,A)\in \M\times \W\, |\, |A|\leq\ell , d_g(x,\varphi^A_1(x))\leq \eps \}\subset \bigcup_j B_j, \\
B_j:=\{(x,A)\in \M\times \W\, |\, |A-A_j|\leq \delta/2, d_g(x,x_j)\leq \delta e^{-ML}/2, d_g(x,\varphi^A_1(x))\leq \eps \}.
\end{gathered}
\]
Now by Lemma \ref{lem:separation}, if $(x',A')\in B_j$, $B_j$ is contained in
an $\eps$-neighborhood of the orbit $\{\varphi^{A''}(x')\, |\, |A''|\leq 1\}$ times
a $\epsilon$ ball of $A'$ in $\W$. The first neighbourhood has $\nu_g$-measure $\mc{O}(\eps^{n-\kappa})$ and the latter $dA$-measure $\mc{O}(\epsilon^\kappa)$. This shows that $(v_g\otimes dA)(Z)=\mc{O}(\eps^n e^{nM\ell})$.

We conclude with a bound on the number of periodic tori. By Corollary~\ref{cor:separation}, we see that the periodic tori $T$ so that $L(T)\cap B(A,\delta)\not=\emptyset$ with $B(A,\delta):=\{A' \in\W\, |\, |A-A'|<\delta\}$ are separated by a distance at least $\delta e^{-M|A|}$, thus there are tubular neighborhoods of volume bounded below by $C\delta^{n-\kappa}e^{-(n-\kappa)M|A|}$ that do no intersect in $\M$.
By a covering argument we deduce that for each $A\in L(T)$ with $T\in \mc{T}$
\[\sharp\{ T' \in \mc{T}\, |\, L(T')\cap B(A,\delta)\not=\emptyset\}\leq C\delta^{\kappa-n}
e^{(n-\kappa)M|A|}\]
\[ \sharp\{ A' \in L(T')\cap\W\, |\, T'\in\mc{T}, A'\in B(A,\delta)\}\leq C\delta^{\kappa-n}
e^{(n-\kappa)M|A|}\]
(using again Corollary \ref{cor:separation} for the last bound) and therefore again by covering $\W\cap \{|A|\leq \ell\}$ by $\mc{O}(\ell^{\kappa})$ balls of radius $\delta$ we conclude that
\[ \sharp\{ A \in L(T)\cap\W\, |\, T\in\mc{T}, |A|\leq \ell\}\leq C\ell^{\kappa}e^{(n-\kappa)M\ell}.\qedhere\]
\end{proof}

\section{Multivariable analytic Fredholm theorem.}

The so-called ``Analytic Fredholm Theorem'' is a staple of spectral theory. In the main part of the article, we rely on this slight extension:
\begin{prop}\label{prop:Fredholm-several-variables}
Suppose that for $z\in \Omega \in \C^n$, a connected open set, $A(z)$ is a family of bounded Fredholm operators on $H$ (a fixed Hilbert space) depending holomorphically on $z$. If $A(z_0)^{-1}$ exists at a point $z_0 \in \Omega$, then $\Omega\owns z\mapsto A(z)^{-1}$ is a meromorphic family of operators. 
\end{prop} 

We have purposefully mimicked the statement \cite[Proposition 2.3]{Sjostrand-Zworski-Fredholm}, as we will follow their classical proof. 
\begin{proof}
Since $A(z_0)$ is invertible, its index is $0$, and thus so is the case for $A(z)$, $z\in\Omega$. Let $z_1\in\Omega$, and define 
\[
H_- = \ker A(z_1),\quad H_+ = \ran A(z_1)^\perp. 
\]
We can decompose
\[
H =  (H_-)^\perp \oplus H_- = (H_+)^\perp \oplus H_+ ,
\]
and in this decomposition, 
\[
A(z) = \begin{pmatrix}
E(z) & E_+(z) \\ E_-(z) & E_{-+}(z)
\end{pmatrix}. 
\]
At $z=z_1$, this becomes
\[
A(z_1) =\begin{pmatrix}
E(z_1) & 0 \\ 0 & 0
\end{pmatrix},
\]
with $E(z_1)$ invertible. For $z$ close to $z_1$, $E(z)$ remains invertible, so we can apply Schur's complement and deduce that $A(z)$ is invertible if and only if 
\[
P(z):= E_{-+}(z) - E_-(z) E^{-1}(z) E_+(z).
\]
is invertible. The upside is that $P(z)$ is now a \emph{matrix valued} holomorphic function of $z$. Taking $f(z) = \det P(z)$, we have found a holomorphic function in a neighbourhood of $z_1$, such that $f(z)=0$ if and only if $A(z)$ is not invertible. By connectedness of $\Omega$, and since $A(z_0)$ is invertible, we deduce that $f$ is a non constant function. Wherever $f(z)\neq 0$, we obtain 
\[
A(z)^{-1} = \begin{pmatrix}
E^{-1} + E^{-1} E_+ P^{-1} E_- E^{-1} & - E^{-1} E_+ P^{-1} \\ - P^{-1} E_- E^{-1} & P^{-1}
\end{pmatrix}.
\]
Using the comatrix formula for $P$, we can decompose this into
\[
A(z)^{-1} = \begin{pmatrix}
E(z)^{-1} & 0 \\ 0 & 0 
\end{pmatrix} + \frac{1}{f(z)} M(z),
\]
where $M(z)$ is a holomorphic family of bounded operators. 
\end{proof}

\bibliographystyle{amsalpha}
\bibliography{JRbib}

\end{document}